\documentclass[12pt]{amsart}

\usepackage{graphicx}
\usepackage{amsmath, comment}
\usepackage{amscd}
\usepackage{amsfonts}
\usepackage{amssymb}
\usepackage{color}
\usepackage{fullpage}
\usepackage{mathtools}
\usepackage[hypertexnames=false, colorlinks, citecolor=red,linkcolor=blue, urlcolor=black]{hyperref}
\usepackage{tikz-cd}
\usepackage[capitalise]{cleveref}
\usepackage{enumitem}

\usepackage[utf8]{inputenc}
\usepackage[T1]{fontenc}
\usepackage[french,english]{babel}

\numberwithin{equation}{section}

\newtheorem*{theorem*}{Theorem}
\newtheorem{theorem}{Theorem}[section]
\newtheorem{lemma}[theorem]{Lemma}
\newtheorem{fact}[theorem]{Fact}
\newtheorem{proposition}[theorem]{Proposition}

\newtheorem{corollary}[theorem]{Corollary}

\newtheorem{question}[theorem]{Question}

\theoremstyle{definition}
\newtheorem{example}[theorem]{Example}
\newtheorem*{prev}{Previously known bounds}

\newtheorem{remark}[theorem]{Remark}
\newtheorem{convention}[theorem]{Convention}
\newtheorem{definition}[theorem]{Definition}
\newtheorem{assumption}[theorem]{Assumption}

\newcommand{\cC}{\mathcal{C}}

\newcommand{\cO}{\mathcal{O}}

\newcommand{\cU}{\mathcal{U}}
\newcommand{\cV}{\mathcal{V}}

\newcommand{\bC}{\mathbb{C}}

\newcommand{\bF}{\mathbb{F}}

\newcommand{\bN}{\mathbb{N}}
\newcommand{\bQ}{\mathbb{Q}}
\newcommand{\bR}{\mathbb{R}}

\newcommand{\bZ}{\mathbb{Z}}

\newcommand{\Gm}{\mathbb{G}_m}

\newcommand{\id}{\operatorname{id}}

\newcommand{\ad}{\mathrm{ad}}
\newcommand{\et}{\text{\'etale}}
\newcommand{\Ru}{\mathcal{R}_{\mathrm{u}}}
\newcommand{\an}{\operatorname{an}}
\newcommand{\Ad}{\operatorname{Ad}}
\newcommand{\im}{\operatorname{im}}

\newcommand{\Gal}{\operatorname{Gal}}
\newcommand{\alg}{\operatorname{alg}}
\newcommand{\Aut}{\operatorname{Aut}}

\newcommand{\PGL}{\operatorname{PGL}}
\newcommand{\PGO}{\operatorname{PGO}}

\newcommand{\Hom}{\operatorname{Hom}}

\newcommand{\GL}{\operatorname{GL}}
\newcommand{\SL}{\operatorname{SL}}

\newcommand{\Spec}{\operatorname{Spec}}

\newcommand{\trdeg}{\operatorname{trdeg}}
\newcommand{\Sym}{\operatorname{Sym}}

\newcommand{\ed}{\operatorname{ed}}

\newcommand{\SymRank}{\operatorname{Rank}}
\newcommand{\ol}{\overline}
\newcommand{\rank}{\operatorname{rank}}

\newcommand{\sep}{\operatorname{sep}}

\newcommand{\colim}{\operatorname{colim}}

\newcommand{\Char}{\operatorname{char}}

\newcommand{\Stab}{\operatorname{Stab}}

\newcommand{\eps}{\varepsilon}

\newcommand{\stab}{\operatorname{stab}}

\newcommand{\tr}{\operatorname{tr}}
\newcommand{\res}{\operatorname{res}}

\newcommand{\Mat}{\operatorname{M}}

\newcommand{\fm}{{\mathfrak{m}}}

\newcommand{\Inf}{\operatorname{Res}}
\newcommand{\Ext}{\operatorname{Ext}}

\newcommand{\Otr}{\cO_{\tr}}

\newcommand{\ksep}{k_{\sep}}
\newcommand{\Ftr}{F_{\tr}}
\newcommand{\Ltr}{L_{\tr}}

\newcommand{\Falg}{F_{\alg}}

\newcommand{\oin}{\operatorname{in}}
\newcommand{\Oin}{\cO_{\oin}}
\newcommand{\Fin}{F_{\oin}}
\newcommand{\Fsep}{F_{\sep}}
\newcommand{\Lin}{L_{\oin}}
\newcommand{\vphi}{\varphi}
\newcommand{\Galtr}{\Gal_{\tr}}

\newcommand{\Tran}{\operatorname{Tran}}
\newcommand{\inn}{\operatorname{inn}}
\newcommand{\len}{\operatorname{len}}

\newcommand{\Frac}{\operatorname{Frac}}
\newcommand{\HSpin}{\operatorname{HSpin}}
\newcommand{\I}[1]{\Gamma_{#1}^{\vee}}
\newcommand{\lop}{\operatorname{loop}}
\usepackage{bbm}

\usepackage{textcase} 

\begin{document}
\author{Danny Ofek}   
\thanks{Danny Ofek was partially supported by a graduate fellowship from 
the University of British Columbia.}


%
%
%

\title{Lower bounds on the essential dimension of reductive groups}

\selectlanguage{english}
\begin{abstract}
    \begin{otherlanguage*}{english}
    We introduce a technique for proving lower bounds on the essential dimension of split reductive groups. As an application, we strengthen the best previously known lower bounds for various split simple algebraic groups, most notably for the exceptional group $E_8$. In the case of the projective linear group $\PGL_n$, we recover A.\kern-.1em\ Merkurjev's celebrated lower bound with a simplified proof. Our technique relies on decompositions of loop torsors over valued fields due to P.\kern-.1em\ Gille and A.\kern-.1em\ Pianzola.
\end{otherlanguage*}

\bigskip
\begin{otherlanguage*}{french}
\noindent\textsc{Résumé (Bornes inférieures sur la dimension essentielle des groupes réductifs).}
Nous introduisons une méthode pour établir des bornes inférieures sur la dimension essentielle d’un groupe réductif déployé. À titre d’application, nous renforçons les meilleures bornes inférieures connues jusqu’à présent pour divers groupes algébriques simples déployés, notamment pour le groupe exceptionnel $E_8$.
Dans le cas du groupe projectif linéaire $\PGL_n$, nous redémontrons, par une preuve simplifiée, la célèbre borne inférieure due à A.\kern-.1em\ Merkurjev.
Notre approche repose sur les décompositions de torseurs de lacets sur des corps valués, dues à P.\kern-.1em\ Gille et A.\kern-.1em\ Pianzola.
\end{otherlanguage*}
\end{abstract} 

\selectlanguage{english}
\maketitle

\setcounter{tocdepth}{1}
\tableofcontents

\section{Introduction}

Let $G$ be a smooth linear algebraic group over a field $k_0$ and $\gamma\in H^1(L,G)$ a $G$-torsor over a field $k_0\subset L$. A \emph{field of definition} for $\gamma$ is a subfield $k_0\subset F \subset L$ such that $\gamma$ lies in the image of the natural map:
$$H^1(F,G) \to H^1(L,G).$$
The \emph{essential dimension} of $\gamma$ is the minimal number of parameters needed to define $\gamma$. It is given by the formula:
$$\ed(\gamma) = \min \Big\{ \trdeg_{k_0}(F) \mid F \   \text{is a field of definition of }\gamma \Big\}.$$
The \emph{essential dimension} of $G$ is  defined as the supremum $ \ed(G) = \sup\{\ed(\gamma)\}$ taken over all overfields $k_0\subset L$ and torsors $\gamma\in H^1(L,G)$. It should be thought of as the minimal number of algebraically independent parameters needed to define an arbitrary $G$-torsor. Often $G$-torsors correspond bijectively to a class of algebraic objects, in which case $\ed(G)$ is the minimal number of independent parameters needed to define a member of that class. For example:
\begin{itemize}
    \item The essential dimension of the symmetric group $S_d$ is the number of independent variables required to define an arbitrary separable field extension of degree $d$. Equivalently, it is the number of parameters needed to define a generic polynomial of degree $d$ up to Tschirnhaus transformations. In this form, mathematicians have tried to compute it as early as the 17th century. For an overview, see \cite{garver1927tschirnhaus,buhler1999,buhler1997essential}.
    \item The essential dimension of $\PGL_d$ is the number of independent variables required to define an arbitrary central division algebra of degree $d$. This quantity has been studied since generic division algebras were first defined by C.\kern-.1em\ Procesi \cite[Section 2]{procesi}.
\end{itemize} 
The problem of computing $\ed(G)$ for a general algebraic group $G$ has been studied by many authors since it was first posed by J.\kern-.1em\ Buhler-Z.\kern-.1em\ Reichstein \cite{buhler1997essential} and Reichstein \cite{reichstein2000notion}. For a comprehensive survey of the developments in the field, we refer the reader to \cite{merkurjev-survey}. 

Let $p$ be a prime integer. The \emph{essential dimension at $p$} of $G$, denoted $\ed(G;p)$, measures how many parameters are required to construct an arbitrary $G$-torsor {up to prime-to-$p$} extensions. See Section~\ref{sect.notation} for a precise definition. The inequality $\ed(G) \geq \ed(G;p)$ always holds, and almost all existing techniques to prove lower bounds on $\ed(G)$ apply to $\ed(G;p)$ as well; see \cite[Section 5]{reichstein2010essential}. The same is true of the technique introduced in this paper.

\subsection{Overview of previous techniques}
Z.\kern-.1em\ Reichstein-B.\kern-.1em\ Youssin gave the first systematic way to prove lower bounds on the essential dimension of algebraic groups. We recall their main theorem, commonly referred to as  ``the fixed-point method'' because it was proven by an analysis of fixed points on generically free $G$-varieties. See \cite{gille2009lower,chernousov2006lower} for generalizations to positive characteristic. 

\begin{theorem}{\cite[Theorem 7.7]{reichstein-youssin}}\label{thm.fixed_pt}
    Assume $G$ is defined over an algebraically closed field of characteristic zero $k_0$ and $G^{\circ}$ is semisimple. Let $A\subset G(k_0)$ be a finite abelian group and $p$ a prime number.
    \begin{enumerate}
        \item If $C_G(A\cap G^{\circ})$ is finite, then $\ed(G) \geq \rank(A)$. 
        \item If $A$ is a $p$-group and $C_G(A\cap G^{\circ})$ is finite, then $\ed(G;p)\geq \rank(A)$.
    \end{enumerate}
\end{theorem}
 While Theorem~\ref{thm.fixed_pt} applies to many groups, it usually gives bounds which are far from tight. This is partially explained by the fact that the $G$-torsors witnessing the lower bound can be constructed over an iterated Laurent series field $k_0(\!(t_1)\!)\dots(\!(t_r)\!)$ which is a relatively simple field when $k_0$ is algebraically closed (here $r = \rank(A)$). 

P.\kern-.1em\ Brosnan-Reichstein-A.\kern-.1em\ Vistoli \cite{genericity0,brosnan2007essential} and later N.\kern-.1em\ Karpenko-A.\kern-.1em\ Merkurjev \cite{karpenko2008essential} introduced stack-theoretic techniques to construct $G$-torsors of high essential dimension over function fields of Severi-Brauer varieties. Stack-theoretic techniques give much stronger lower bounds than is possible using Theorem~\ref{thm.fixed_pt} for some groups, like finite $p$-groups and algebraic tori \cite{karpenko2008essential,lotscher2013essential}. However, they give trivial lower bounds for most semisimple groups, including all adjoint simple groups.  In \cite{merkurjev-PGLn}, Merkurjev overcame this limitation of the stack-theoretic methods for adjoint groups of type $A_n$ by proving $\ed(\PGL_n;p) \geq \ed(T;p)$ for a certain torus $T$. He then computed $\ed(T;p)$ using \cite{lotscher2013essential} to obtain the lower bounds
\begin{equation}\label{e.lower_PGL_n}
    \ed(\PGL_{p^{r}};p) \geq (r-1)p^r+1
\end{equation}
which are orders of magnitude stronger than the lower bounds previously obtained by Theorem~\ref{thm.fixed_pt}. Merkurjev's arguments are specific to groups of type $A_{n}$ because they rely on explicit computations in the Brauer group (see \cite{chernousov2013essential} for generalizations to $\SL_n/\mu_d$). We note that a proof of \eqref{e.lower_PGL_n} for $p=r =2$ was first announced by M. Rost \cite{rost2000computation}. Merkurjev proved \eqref{e.lower_PGL_n} for $r= 2$ and $p$ arbitrary by different means in \cite{merkurjev-PGLp^2}.

\subsection{Main results}
In this paper, we introduce a technique to prove lower bounds on $\ed(G),\ed(G;p)$, which applies whenever $G^{\circ}$ is split reductive.  We proceed in two steps:
\begin{enumerate}
    \item We first prove $\ed(G)\geq \ed(C_G(A))$ for any finite diagonalizable subgroup $A\subset G$ satisfying certain conditions.
    \item We choose $A$ in a systematic way, so that $C_G(A)$ is an extension of a torus by a finite group. This allows us to apply the results of \cite{lotscher2013essential2} to give a strong lower bound on $\ed(C_G(A))$.
\end{enumerate}
In this way we obtain new lower bounds on the essential dimension of some simple groups as well as recover \eqref{e.lower_PGL_n}, see Theorem~\ref{thm.concrete_bounds} below and Section~\ref{subsect.PGL_n}. The next theorem gives sufficient conditions for the inequality $\ed(G)\geq \ed(C_G(A))$ to hold. Note that we do not assume $G^{\circ}$ is split. Thus the first step in our technique works more generally and the assumption that $G^{\circ}$ is split is only needed for the second step. Recall that a $G$-torsor $[c_\sigma]\in H^1(F,G)$ is called \emph{anisotropic} if the twisted group ${}_c G$ contains no copy of $\mathbb{G}_{m}$.  A finite algebraic group $A$ over $k_0$ is called \emph{diagonalizable}, if it is isomorphic to $\mu_{n_1}\times \dots \times \mu_{n_r}$ for some $n_1,\dots,n_r$ coprime to $\Char k_0$.

\begin{theorem}\label{thm.main}
    Let $G$ be a smooth linear algebraic group over a field $k_0$. Assume either $k_0$ is perfect or $G^{\circ}$ is reductive. Let $A \subset G$ be a finite diagonalizable subgroup.
    \begin{enumerate}
        \item Let $ p\neq \Char k_0$ be a prime. If $A$ is a $p$-group and $C_G(A)$ admits an anisotropic torsor over some $p$-closed field $k_0\subset k$, then we have:
        $$\ed(G;p) \geq \ed(C_G(A);p).$$
        
        \item Assume $\Char k_0$ is good for $G$ (see Definition~\ref{def.char}). If $C_G(A)$ admits an anisotropic torsor  over some field $k_0\subset k$, then we have:
        $$\ed(G) \geq \ed(C_G(A)).$$

    \end{enumerate}
\end{theorem}

Let $F$ be a Henselian valued field with value group of finite rank. Our proof of Theorem~\ref{thm.main} relies on the decompositions of \emph{loop torsors} over $F$. Loop torsors and their decompositions were introduced by P.\kern-.1em\ Gille-A.\kern-.1em\ Pianzola for iterated Laurent series over a characteristic zero field in the context of the classification of loop algebras \cite{gille2013torsors}. We will use both \cite{gille2013torsors} and the recent generalizations to valuation rings of positive characteristic obtained by Gille \cite{gille2024newloop}. 

\begin{remark}
    Theorem~\ref{thm.fixed_pt} is a special case of Theorem~\ref{thm.main}. Indeed,
    if $C_G(A\cap G^{\circ})$ is finite, then so is $C_G(A)$.  
    Combining \cite[Lemma 4.1]{buhler1997essential} and \cite[Theorem 6]{buhler1997essential} gives:
$$\ed(C_G(A))\geq \ed(A) = \rank(A).$$ 
Since $C_G(A)$ is finite, all $C_G(A)$-torsors are anisotropic. Therefore Theorem~\ref{thm.main} implies:
    $$\ed(G) \geq \ed(C_G(A))\geq \rank(A).$$
Note that the above argument works under the assumption $|C_G(A)| <\infty$, which is strictly weaker than $|C_G(A\cap G^{\circ})|<\infty$.
\end{remark} 

As noted in the previous remark, any $C_G(A)$-torsor is anisotropic when $C_G(A)$ is finite. However, in general, there is no simple criterion to determine whether $C_G(A)$ admits anisotropic torsors. J. Tits classified the simple simply connected split groups that do not admit anisotropic torsors \cite{tits1990strongly}. For a generalization of his work to arbitrary simple groups and for semisimple groups of certain types, see \cite{ofek2022reduction}.
We have no examples where the inequalities of Theorem~\ref{thm.main} fail, so it is natural to ask if one may improve the theorem as follows:
\begin{question}\label{question1}
    Is the conclusion of Theorem~\ref{thm.main} true without the assumption that $C_G(A)$ admits anisotropic torsors?
\end{question}
After the proof of Theorem~\ref{thm.main} in Section~\ref{sect.proof_of_main_thm}, we give an example where all $C_G(A)$-torsors are isotropic and our proof breaks down. However, we have no reason to expect a negative answer to Question~\ref{question1}.

In Section~\ref{sect.gradings_on_char}, we give a streamlined root-theoretic approach to choosing diagonalizable subgroups of split groups that satisfy the conditions of Theorem~\ref{thm.main}. This leads to the following new lower bounds:

\begin{theorem}\label{thm.concrete_bounds}
    Assume  $\Char k_0 \neq 2,3$. 
    \begin{enumerate}
        \item $\ed(E_8;2) = \ed(\HSpin_{16};2) \geq 56$
        \item $ \ed(E_8;3) \geq 13$
        \item $\ed(E_6^{\ad};3) \geq 6$
        \item $ \ed(E_7^{\ad};2) \geq 19$
        \item If $n  =  2^r m\geq 4$ for $r\geq 1$ and $m$ is odd, then $$\ed(\PGO^+_{2n};2) \geq   (r -1)2^{r+1} +n.$$
        \item If $n \geq 3$ is odd, then  $$\ed(\PGO^+_{2n};2) \geq   3n-4.$$
    \end{enumerate}
\end{theorem}
 The inequalities in Theorem~\ref{thm.concrete_bounds} improve on long-standing lower bounds. To put them into context, we include a list indicating the best lower and upper bounds for these groups that were known prior to this paper.
 \begin{prev}\itshape
 Assume  $\Char k_0 \neq 2,3$. 
\begin{enumerate}
        \item $120 \geq \ed(E_8;2)\geq 9$
        \item $73 \geq \ed(E_8;3) \geq 5$
        \item $21\geq \ed(E_6^{\ad};3) \geq 4$
        \item $57\geq \ed(E_7^{\ad};2) \geq 8$
        \item If $n  =  2^r m$ for some $r\geq 0$ and $m\geq 3$ is odd, then 
        $$(m - 1)2^{2(r+1)} - n \geq\ed(\PGO^+_{2n};2) \geq  2n-2.$$
        \item If $n = 2^r$  for some $r\geq 2$, then  $$n^2 - n \geq \ed(\PGO^+_{2n};2) \geq   r 2^r.$$
    \end{enumerate}
 \end{prev}
All of the lower bounds in (1)-(4) follow from the characteristic-free versions of Theorem~\ref{thm.fixed_pt}; see \cite{chernousov2006lower,gille2009lower}. For references for the upper bounds in (1)-(4), see \cite[3h]{merkurjev-survey}. 
The upper bounds in (5) and (6) are due to M. Macdonald \cite[Section 0.2]{macdonald}. For the lower bounds in (5),(6), see \cite[Theorem 13]{chernousov2006lower} and \cite[Lemma 2.6]{baek2015essential} respectively. The inequality $\ed(\HSpin_{16};2) \geq \ed(E_8;2)$ was known to experts, but the inequality $\ed(\HSpin_{16};2) \leq \ed(E_8;2)$ is new; see the proof of Proposition~\ref{prop.partofconcreteboundsthm1}. We note that (5) implies $$3n-4 \geq \ed(\PGO^+_{2n} ;2)$$
for odd $n$; see also \cite[Corollary 3.2]{baek2015essential}. Therefore the inequality in Theorem~\ref{thm.concrete_bounds}(6) is tight. 

Recall that the \emph{reductive rank} of $G$ is the dimension of a maximal torus $T\subset G$. Clearly if $H\subset G$ is a subgroup, then $\rank(H)\leq \rank(G)$. A subgroup $H$ is said to be of \emph{maximal rank} if $\rank(H) = \rank(G)$. In Section~\ref{sect.pf_of_cor_ofmainthm}, we use Theorem~\ref{thm.main} to prove 
$$\ed(G)\geq \ed(H), \ \ \ed(G;p)\geq \ed(H;p)$$
for reductive subgroups $H\subset G$ of maximal rank under certain conditions; see Corollary~\ref{cor.ofmainthm}. We use this corollary to establish Parts (1),(2) and (4) of Theorem~\ref{thm.concrete_bounds}.  The proof relies on Borel–de Siebenthal's theorem on reductive subgroups of maximal rank \cite{borel1949sous}.

\subsection{Outline of the paper}
The rest of the paper is structured as follows.
Sections~\ref{sect.notation}-\ref{sect.pre_isotrop} deal with notation and preliminaries. 
In Section~\ref{sect.loop_torsors} we introduce the definition of loop torsors and their decompositions. In Section~\ref{sect.a_theorem_of_GP} we adapt a theorem of Gille-Pianzola about the uniqueness of decompositions of  anisotropic loop torsors to our needs. Section~\ref{sect.funct_of_decomp} deals with the functoriality of decompositions of loop torsors with respect to extension of scalars. In Section~\ref{sect.all_torsors_are_loops} we give criteria to determine whether a $G$-torsor is a loop torsor. The proof of Theorem~\ref{thm.main} is completed in Section~\ref{sect.proof_of_main_thm}. Section~\ref{sect.gradings_on_char} contains a method of choosing subgroups $A\subset G$ that satisfy the conditions of Theorem~\ref{thm.main} and such  that $C_G(A)^{\circ}$ is a split torus. The last four sections contain the computations needed to prove Theorem~\ref{thm.concrete_bounds}.

\subsection{Acknowledgments} I am grateful to A.\kern-.1em\ Soofiani for many fruitful discussions, and to V.\kern-.1em\ Chernousov for pointing out the connection between my ideas and his work with P.\kern-.1em\ Gille and A.\kern-.1em\ Pianzola. This connection was further explained to me by Gille and it simplified this paper greatly.  I would like to thank my supervisor Z.\kern-.1em\ Reichstein for many useful meetings, suggestions and advice throughout the work on this project.  Thanks are also due to the anonymous referees whose
suggestions helped improve the exposition and correct multiple typos. 


\section{Notation}\label{sect.notation}

Throughout this paper, $G$ will denote a smooth linear algebraic group over a base field $k_0$. All fields and rings we consider will contain $k_0$. We will sometimes assume $G$ and $k_0$ satisfy the following additional assumptions.
\begin{definition}\label{def.char}
    Let $\Ru(G)\subset G^{\circ}$ denote the unipotent radical of $G$ and set $\overline{G} = G/\Ru(G)$; see \cite[Definition 6.44]{milne2017algebraic} for the definition of $\Ru(G)$.
    We say that the characteristic of $k_0$ \emph{is good for} $G$ if the following hold:
    \begin{enumerate}
        \item The group $\overline{G}^{\circ}$ is reductive.
        \item If $\Char k_0> 0$, then there exists a maximal torus $T\subset \overline{G}$ split by a prime-to-$\Char k_0$ extension of $k_0$ and such that $|N_{\overline{G}}(T)/T|$ is coprime to $\Char k_0$ (note that $|N_{\overline{G}}(T)/T|$ is finite because $\overline{G}^{\circ}$ is reductive).
    \end{enumerate}
    Note that condition (1) is automatic if $G^{\circ}$ is reductive or $k_0$ is perfect \cite[Proposition 19.11]{milne2017algebraic}. If $k_0$ is imperfect, then there exists a purely inseparable extension $k_1/k_0$ such that (1) holds for $G_{k_1}$; see \cite[Proposition 1.1.9(2)]{conrad-gabber-prasad}.
\end{definition}

We fix a system of compatible primitive roots of unity $\zeta_n \in  k_{0,\sep}$ for all $n$ not divisible by $\Char k_0$. Let $k_0\subset k$ be a field. We will denote the cyclotomic character of $k$ by:
$$\theta: \Gal(k) \to (\hat{\bZ}')^*.$$
Here $ \hat{\bZ}' = \prod_{p\neq \Char k_0} \hat{\bZ}_p$ is the prime-to-$\Char k_0$ part of the profinite integers $\hat{\bZ}$.
The cyclotomic character is the unique homomorphism satisfying:
$$ {}^\sigma \zeta_n = \zeta_n^{\theta(\sigma)}$$
for all $n$ not divisible by $\Char k_0$ and $\sigma\in \Gal(k)$. Recall that a \emph{finite diagonalizable} group $A$ over $k$ is an \'etale group scheme isomorphic to $\mu_{n_1}\times \dots \times \mu_{n_r}$ for some $n_1,\dots,n_r$ coprime to $\Char k_0$. The following characterization of diagonalizable groups will be used implicitly throughout the paper. We leave the proof as an exercise to the reader.
\begin{fact}\label{fact.split_diag}
    Let $A$ be an abelian \'etale group over $k$ such that $|A|$ is finite and coprime to $\Char k$. Then $A$ is diagonalizable if and only if for any $a\in A(\ksep)$ and $\sigma\in \Gal(k)$:
    \begin{equation}\label{def.split_diag}
    {}^{\sigma} a = a^{\theta(\sigma)}.
    \end{equation}
\end{fact}

The letter $\cO$ will denote a Henselian valuation ring containing $k_0$. We set $F =\Frac(\cO)$ and let $\nu: F^*\to \Gamma_F$ be the corresponding valuation. The residue field of $\nu$  will be denoted $k$. Note there exists a natural embedding $k_0\subset k$ induced from the inclusion $k_0\subset \cO$. Since $\nu$ is Henselian, it admits a unique extension to any algebraic extension $F\subset L$. We will denote the extension of $\nu$ to $L$ by $\nu$ again by abuse of notation. Unless explicitly stated we always assume:
\begin{assumption}\label{ass.finite_rank}
    The value group $\Gamma_F$ is a finitely generated (free) abelian group.
\end{assumption}
We will say $(F,\nu)$ is an \emph{iterated Laurent series field}, if it is isomorphic to $k(\!(t_1)\!)\dots(\!(t_r)\!)$ equipped with the $(t_1,\dots,t_r)$-adic valuation $\nu_r$ for some $r\geq 0$. Recall that $\nu_r$ is given by:
\begin{equation}\label{e.def_of_t-adic_val}
    \nu_r\big(\sum_{i_1,\dots,i_r}a_{i_1,\dots,i_r}t_1^{i_1}t_2^{i_2}\dots t_{r}^{i_r} \big) = \min\big\{ (i_1,\dots,i_r) \mid a_{i_1,\dots,i_r} \neq 0 \big\},
\end{equation}
where the minimum is taken with respect to the lexicographic ordering on $\bZ^r$.
Let  $F_{\oin} \subset \Ftr$ be the maximal inertial (or unramified) and maximal tamely ramified extensions of $F$ inside a fixed separable closure $\Fsep$.    For example, if $F = k(\!(t)\!)$ and $\nu$ is the $t$-adic valuation, then  $\Fin$ is the compositum of $F$ and $\ksep$ and   $\Ftr = \underset{(n,\Char k)=1}{\bigcup} \Fin(t^{1/n})$. The corresponding Henselian valuation rings are denoted $\Oin\subset \Otr$. We will use the following notation for the tamely ramified part of the absolute Galois group of $F$:
$$\Galtr(F) = \Gal(\Ftr/F).$$
We will denote the homomorphism $G(\Otr)\to G(\ksep)$ induced from the residue homomorphism $\Otr\to \ksep$  by $g\mapsto \overline{g}$. We refer the reader to  \cite{efrat2006valuations}, \cite[Appendix A]{tignol2015value} for general results on valuation theory.

\begin{remark}\label{rem.equivariance}
    All of $\cO,\Oin,\Otr$ are stabilized by $\Galtr(F)$ because the extension of $\nu$ to $\Ftr$ is preserved by $\Galtr(F)$.  Moreover, $\Galtr(F)$ stabilizes the maximal ideal of $\Otr$ and so it acts on its residue field $\ksep$. Therefore $\Galtr(F)$ acts on $G(\Oin),G(\Otr),G(\ksep)$ and the induced map $G(\Otr)\to G(\ksep)$ is $\Galtr(F)$-equivariant.
\end{remark}

The symbol $H^1(F,G)$ will stand for the Galois cohomology set $H^1(\Gal(F),G(\Fsep))$. We will denote cohomology classes in $H^1(F,G)$ by $[c_\sigma]$ where $c_\sigma\in Z^1(\Gal(F),G(\Fsep))$ is a $\Gal(F)$-cocycle.  Let $F\subset L$ be a field extension and assume $\Fsep \subset L_{\sep}$. The restriction map of absolute Galois groups
$$\Gal(L) \to \Gal(F), \ \sigma\mapsto \sigma_{|\Fsep}$$
gives rise to restriction maps $$H^1(F,G) \to H^1(L,G), \  \gamma \mapsto \gamma_L.$$
Here $\gamma_L$ is the cohomology class of the restriction of $c_\sigma$ to $L$ given by $$\Inf_{L/F}(c)_\sigma = c_{\sigma_{|\Fsep}}.$$
The set of \emph{tamely ramified torsors} will be denoted:
$$H^1_{\tr}(F,G) = H^1(\Galtr(F),G(\Ftr)).$$ 
There is always a natural inclusion $H^1_{\tr}(F,G) \subset H^1(F,G)$ which is a bijection if $G^{\circ}$ is reductive and $\Char k_0$ is good for $G$ in the sense of Definition~\ref{def.char} (see Proposition~\ref{lem.loop_char_good}).

The morphism $\Spec \Oin \to \Spec \cO$ is the universal pro-\'etale cover of $\Spec \cO$ and therefore there is a natural identification $\pi_1(\Spec \cO) = \Gal(\Fin/F)$ \cite[Section 1, Example 5.2.(d)]{milne1980etale}. 
For any cocycle $a_\sigma\in Z^1(\Gal(\Fin/F),G(\Oin))$, we will write ${}_a G$ for the twist of $G_{\cO}$ by the $G$-torsor defined by $a_\sigma$ over $\cO$; see \cite[Section 2.2]{gille2016classification}.

A variety $X$ over $k_0$ is a reduced (but possibly reducible) quasi-projective scheme of finite type over $k_0$.
Finally, we recall the definition of the essential dimension of $G$ at a prime.
\begin{definition}\label{def.ess_dim_p}
     Let $\gamma\in H^1(F,G)$ be a $G$-torsor over a field $k_0\subset F$ and $p$ a prime integer. An algebraic extension $F\subset L$ is called a \emph{prime-to-$p$} extension if any finite subextension $F\subset L'\subset L$ is of degree coprime to $p$. The \emph{essential dimension at $p$} of $\gamma$ is the minimal number of parameters required to define $\gamma$ if one ignores prime-to-$p$ extensions: $$\ed(\gamma;p) = \min\Big\{\ed(\gamma_L) \mid F\subset L \text{ is a prime-to-}p \ \text{field extension}\Big\}.$$
We define the \emph{essential dimension at $p$} of $G$ by  $\ed(G;p) = \sup\{\ed(\gamma;p)\}$, the supremum taken over all $G$-torsors as before.
\end{definition}

\section{Preliminaries about Henselian valued fields}\label{sect.pre_fields}

\subsection{A fundamental exact sequence} In this section and in all remaining sections we keep the notation from Section~\ref{sect.notation}.
The Galois extensions $F\subset \Fin\subset \Ftr$ induce an exact sequence
\begin{equation}\label{e.seq_inertia_Foin}
    1\to \Galtr(F_{\oin}) \to \Galtr(F) \to \Gal(\Fin/F)\to 1.
\end{equation}
The groups $\Galtr(\Fin)$ and $\Gal(\Fin/F)$ can be understood in terms of the value group $\Gamma_F$ and the residue field $k$ respectively. 
There is a natural isomorphism $\Gal(\Fin/F) \cong \Gal(k)$ given by taking an automorphism $\sigma\in \Gal(\Fin/F)$ to the unique automorphism $\overline{\sigma}\in \Gal(k)$ satisfying  for all  $x\in \Oin$ :
\begin{equation}\label{e.residue_Gal_iso}
\overline{\sigma}(\overline{x}) = \overline{\sigma(x)},
\end{equation}
where $\overline{x}\in \ksep$ is the residue class of $x$ \cite[Theorems A.23]{tignol2015value}. 
\begin{convention}\label{rem.Fin}
    We will often identify $\Gal(\Fin/F)$ and $\Gal(k)$ using this isomorphism. For example, we might act on $G(\Fin)$ using $\Gal(k)$ or act on $G(\ksep)$ using $\Gal(\Fin/F)$.
\end{convention}
The group $\Galtr(F_{\oin})$ is the \emph{tame inertia group} of $F$. Let $ \hat{\bZ}' = \prod_{p\neq \Char k} \hat{\bZ}_p$ be the prime-to-$\Char k$ part of $\hat{\bZ}$ and set $\I{F}:= \Hom(\Gamma_F,\hat{\bZ}')$. There is an  isomorphism
\begin{equation}\label{e.inertia_iso}
    \Phi: \Galtr(F_{\oin}) \tilde{\to} \I{F}.
\end{equation}
The isomorphism $\Phi$ is determined by a choice of roots of unity (or equivalently by an isomorphism between $\hat{\bZ}'$ and the Tate module $\hat{\bZ}'(1) =\underset{(n,\Char k)= 1}{\lim} \mu_n$) and     by the following equation, which holds for all $n$ coprime to $\Char k$, $x\in F^*$ and $\sigma\in \Galtr(\Fin)$:
\begin{equation}\label{e.10}
    {\frac{\sigma(x^{1/n})}{x^{1/n}}} = {\zeta_n}^{\Phi(\sigma)(\nu(x))}.
\end{equation} 
See \cite[Theorem A.24]{tignol2015value} for example. Note that the above expression does not depend on a choice of an $n$-th root of $x$ because $\Fin$ contains all $n$-th roots of unity.
\begin{remark}\label{rem.quotients_tame_inertia}
    Since we are assuming $\Gamma_{F}\cong \bZ^r$ for some $r$, we have $\I{F} \cong \Hom(\bZ^r,\hat{\bZ}') \cong \hat{\bZ}'^r$. In particular, any finite quotient of $\I{F}$ is an abelian group of order coprime to $\Char k$.
\end{remark}
 Next we introduce uniformizers. They will help us describe splittings of \eqref{e.seq_inertia_Foin} in a systematic way.

\begin{definition}
    A left inverse $\pi : \Gamma_F \to F^*$ to $\nu: F^* \to 
    \Gamma_F$ will be called a \emph{uniformizing parameter}.
    Since the group operation in $\Gamma_F$ is written additively while $F^*$ is written multiplicatively, it will be convenient for us to use the exponential notation $\pi^{\gamma}$ in place of $\pi(\gamma)$, for any
    $\gamma\in  \Gamma_F$.
\end{definition}

The next proposition gives a convenient splitting of \eqref{e.seq_inertia_Foin}. It is originally due to J.\kern-.1em\ Neukirch \cite[Satz 2]{neukirch1968verzweigungstheorie}.

\begin{proposition}\label{prop.decomp_of_Gal}
    Let $\pi: \Gamma_{\Ftr} \to \Ftr^*$  be a uniformizer such that $\pi^{\gamma} \in F$ for all $\gamma\in \Gamma_{F}$. 
    \begin{enumerate}
        \item The field $F_{\pi} := F(\pi^{\gamma}; \gamma\in \Gamma_{\Ftr})$ is a complement of $\Fin$ in $\Ftr$. That is, it satisfies $F_{\pi} F_{\oin} = \Ftr$ and $F_{\pi}\cap \Fin = F$.
        \item There exists a unique section $s_{\pi} : \Gal(\Fin/F)\to \Gal_{\tr}(F)$ of the homomorphism $\Galtr(F)\to \Gal(\Fin/F)$ in \eqref{e.seq_inertia_Foin} whose image fixes $F_{\pi}$.
        \item Let $\Phi: \Galtr(\Fin) \to \I{F}$ be the isomorphism  \eqref{e.inertia_iso} and denote by $$\overline{s_\pi}: \Gal(k) \to \Galtr(F)$$ the composition of $s_\pi$ with the isomorphism $\Gal(k)\tilde{\to} \Gal(\Fin/F)$. There exists an isomorphism: $$\Psi_{\pi} : \Gal(k)\ltimes \I{F} \tilde{\to} \Galtr(F), \  (\sigma,f)\mapsto \overline{s}_\pi(\sigma) \Phi^{-1}(f)$$ Here the $\Gal(k)$-action on $\I{F}$ is given by pointwise multiplication by the cyclotomic character $\theta: \Gal(k) \to (\hat{\bZ}')^*$.
        \item For any $(\sigma,f)\in \Gal(k)\ltimes \I{F}$, $u\in \Fin$, $m$ coprime to $\Char(k)$ and $\gamma\in \Gamma_F$, we have:
        \begin{equation}\label{e.phi(1,1)}
        \Psi_\pi(\sigma,f)(u\pi^{\gamma/m}) = \sigma(u)\zeta_{m}^{\theta(\sigma) f(\gamma)} \pi^{\gamma/m}.
        \end{equation}
        Here we let $\Gal(k)$ act on $\Fin$ as in Remark~\ref{rem.Fin}.
        
    \end{enumerate}
\end{proposition}
\begin{proof}
    Let $\gamma_1,\dots,\gamma_r$ be a $\bZ$-basis for $\Gamma_F$. The field $F_n = F(\pi_1^{\gamma_1/n},\dots,\pi_r^{\gamma_r/n})$
    is totally tamely ramified with value group $\frac{1}{n}\Gamma_F$ for all $n$ coprime to $\Char k$. Since $\Gamma_{\Ftr} = \bigcup_{n \text{ coprime to }p} \frac{1}{n} \Gamma_{F}$ \cite[Theorem A.24]{tignol2015value}, we have:
    $$F_\pi = \bigcup_{n \text{ coprime to }p} F_n.$$
    Therefore the residue field of $F_\pi$ is $k$ and $\Gamma_{F_{\pi}} = \Gamma_{\Ftr}$. 
    This implies Parts (1) and (2) by \cite[Satz 2]{neukirch1968verzweigungstheorie} and the discussion following it; see also the proof of \cite[Theorem 22.1.1]{efrat2006valuations}. We note that the restriction map $\res^{\Ftr}_{\Fin}:\Gal(\Ftr/F_\pi)\to \Gal(\Fin/F)$ is an isomorphism and $s_\pi$ is the composition of $(\res^{\Ftr}_{\Fin})^{-1}$ with the canonical inclusion $\Gal(\Ftr/F_\pi)\subset \Galtr(F)$. Part (3) follows from the fact that $s_\pi$ splits the exact sequence \eqref{e.seq_inertia_Foin}; see \cite[Corollary 22.1.2]{efrat2006valuations}. To prove Part (4), we compute using Part (3):
\begin{align*}
\Psi_\pi(\sigma,f)(u\pi^{\gamma/m}) &= \overline{s}_\pi(\sigma) \Phi^{-1}(u \pi^{\gamma/m}) \\
                    &= \overline{s}_\pi(\sigma)(u \zeta_m^{f(\gamma)} \pi^{\gamma/m}) \\
                    &=\sigma(u \zeta_m^{f(\gamma)}) \overline{s}_\pi(\sigma)(\pi^{\gamma/m})  \\
                    &= \sigma(u) \zeta_m^{\theta(\sigma)f(\gamma)} \pi^{\gamma/m}.
\end{align*}
    Here the last three inequalities follow from \eqref{e.10}, \eqref{e.residue_Gal_iso} and the fact that $\overline{s}_\pi(\sigma)$ fixes $F_\pi$.
 \end{proof}
The decomposition  $\Galtr(F)\cong \Gal(k)\ltimes \I{F}$ of Proposition~\ref{prop.decomp_of_Gal} is illustrated by the following diagram of Galois extensions of fields:
\begin{equation}
\begin{tikzcd}[row sep = huge, column sep = huge, arrow style=tikz,arrows=-]
 & \Ftr = F_{\pi} \otimes_F \Fin
   \arrow[dl, "\Gal(k)"']
   \arrow[dr, "\I{F}"] & \\
F_\pi
   \arrow[dr, "\I{F}"']
 & {\scriptstyle \Galtr(F) \  \cong\  \Gal(k)\ltimes\I{F}} \arrow[u] & \Fin \\
 & F
   \arrow[u]
   \arrow[ur, "\Gal(k)"'] &
\end{tikzcd}
\end{equation}
For example, if $F = \bR(\!(t)\!)$ is equipped with the $t$-adic valuation, then $F_\pi = \bigcup_{n\in \bN}\bR(\!(t^{1/n})\!)$ is the field of Puiseux series over $\bR$, $\Fin = \bC(\!(t)\!)$ and $\Ftr = F_{\alg} = \bigcup_{n\in \bN} \bC(\!(t^{1/n})\!)$ is the field of Puiseux series over $\bC$. An element $\sigma \in\Gal(\bR)$ acts on $f = \sum_{q\in \bQ} a_q t^q \in F_{\alg}$ by the formula: $$\sigma.f = \sigma.(\sum_{q\in \bQ} a_{q} t^q) = \sum_{q\in \bQ} \sigma(a_q) t^q.$$ The group $\hat{\bZ}$ acts by $k.f=k.(\sum_{q\in \bQ} a_{q} t^q) = \sum_{q\in \bQ} a_q (e^{2\pi  iq })^{k}t^q$.  By Proposition~\ref{prop.decomp_of_Gal}, these two types of automorphisms generate $\Gal(\bR(\!(t)\!))$ and induce a decomposition $$\Gal(\bR(\!(t)\!)) \cong \Gal(\bR)\ltimes\I{F} \cong \bZ/2\bZ \ltimes \hat{\bZ}.$$

\begin{remark}\label{rem.uniformizers}
    From now on, whenever we use a uniformizer $\pi$ of $\Ftr$, we will always assume $\pi^{\gamma}\in F$ for all $\gamma\in \Gamma_F$.  For any $\pi_1,\dots,\pi_r\in F^*$ such that $\nu(\pi_1),\dots,\nu(\pi_r)$ is a basis for $\Gamma_F$ there exists a uniformizer $\pi$ for $F$ such that for all $1\leq i\leq r$:
    $$\pi^{\nu(\pi_i)} = \pi_i.$$
    Moreover, $\pi$ may be extended to a uniformizer of $\Ftr$. To see this pick a compatible system of roots $\pi^{1/n}$ for all $n$ coprime to $q = \Char k_0$ as in \cite[Lemma 3.2]{reichstein-scavia-specialization}. By \cite[Section 16.2]{efrat2006valuations}, $\Gamma_{\Ftr} = \bZ_{(q)}\otimes_{\bZ} \Gamma_F$, where:
    $$\bZ_{(q)} = \{ \frac{m}{n} \in \bQ \mid n \text{ coprime to } q\}.$$
    Therefore $\nu(\pi_1),\dots,\nu(\pi_r) \in \Gamma_{\Ftr}$ is a $\bZ_{(q)}$ basis and the homomorphism
    $$\bZ_{(q)}\otimes_{\bZ} \Gamma_F \to \Ftr, \ \ \sum_{i=1,\dots,r} \frac{m_i}{n_i} \nu(\pi_i) \mapsto \prod_{i=1,\dots,r} \pi_i^{m_i/n_i},$$
    is a uniformizer for $\Ftr$ extending $\pi$.
\end{remark}

\subsection{Extensions of Henselian fields}
Let $L/F$ be an extension of Henselian fields and denote the corresponding extension of residue fields by $l/k$ . Pick separable closures such that $\Fsep \subset L_{\sep}$. Then we have $\Ftr \subset \Ltr$ and $\Fin \subset L_{\oin}$. Therefore there are well-defined restriction maps:
$$\Galtr(L) \to \Galtr(F), \ \ \Gal(L_{\oin}/L) \to \Gal(\Fin/F), \ \Galtr(L_{\oin}) \to \Galtr(\Fin).$$
There is also a restriction map $\I{L}\to \I{F}, f\mapsto f_{|F}$, given by restricting a homomorphism $f\in \I{L}$ to $\Gamma_F\subset \Gamma_L$.
Next we state a functoriality lemma for the decomposition of $\Galtr(F)$.

\begin{lemma}\label{lem.func_Galois_dec}
    Let $F\subset L$ be an extension of Henselian valued fields and denote by $\nu$ the valuation on $L$. Denote the residue field of $F, L$ by $k$ and $l$ respectively. Let $\pi,\tau$ be uniformizers for $\Ftr$ and $\Ltr$ respectively. Denote the induced section $\overline{s}_\tau : \Gal(l) \to \Galtr(L)$ and isomorphisms as in Proposition~\ref{prop.decomp_of_Gal} by $$\Psi^F_\pi: \Gal(k)\ltimes \I{F} \to \Galtr(F),\  \Psi^L_\tau: \Gal(l)\ltimes \I{L} \to \Galtr(L).$$
    For any $\gamma\in \Gamma_{\Ftr}$ set $u^\gamma = \pi^{\gamma} \tau^{-\gamma}$. For any $\sigma\in \Gal(l)$, there exists a unique homomorphism $\chi_\sigma\in \I{F}$ satisfying the equation:
    \begin{equation}\label{e.25}
    \frac{\overline{s}_\tau(\sigma)(u^{\gamma/n})}{u^{\gamma/n}}= \zeta^{\chi_\sigma(\gamma)}_n
    \end{equation}
    for all $\gamma\in \Gamma_F$ and $n$ coprime to $\Char k$. 
    We have for all $\sigma \in \Gal(l), f\in \I{L}$:
        $$\Psi^L_\tau(\sigma,f)_{|\Ftr} = \Psi^F_\pi(\sigma_{|\ksep},f_{|F} + \theta(\sigma)^{-1} \chi_\sigma).$$
    
\end{lemma}
\begin{proof} For any  $\gamma\in \Gamma_F$ and $n$ coprime to $\Char k$, $u^{\gamma/n}$ is an $n$-th root of $u^\gamma\in \Lin$. Since $u^\gamma$ is a unit and the residue field of $\Lin$ is separably closed, Hensel's lemma implies $u^{\gamma/n} \in \Lin$. In particular, we have:
$$\overline{s}_\tau(\sigma)(u^{\gamma/n}) = \sigma(u^{\gamma/n}).$$
Therefore for any $v\in \Fin, \gamma\in \Gamma_F$  applying \eqref{e.phi(1,1)} twice gives: 
\begin{align*}
 \Psi^L_{\tau}(\sigma,f)(v\pi^{\gamma/n}) &= \Psi^L_{\tau}(\sigma,f)(vu^{\gamma/n}\tau^{\gamma/n}) \\
 &=^{\eqref{e.phi(1,1)}} \sigma(v) \sigma(u^{\gamma/n}) \zeta_n^{\theta(\sigma)f(\gamma)} \tau^{\gamma/n} \\
 &= \sigma(v) \frac{\sigma(u^{\gamma/n})}{u^{\gamma/n}} \zeta_n^{\theta(\sigma)f(\gamma)} u^{\gamma/n}\tau^{\gamma/n}\\
 &= \sigma(v) \zeta_n^{\chi_\sigma(\gamma)} \zeta_n^{\theta(\sigma)f(\gamma)}  \pi^{\gamma/n} \\
&= \sigma(v) \zeta_n^{\theta(\sigma)f(\gamma) + \chi_\sigma(\gamma)} \pi^{\gamma/n} \\
&=^{\eqref{e.phi(1,1)}}\Psi_\pi^F(\sigma_{\ksep},f_{|F}+\theta(\sigma)^{-1} \chi_\sigma)(v\pi^{\gamma/n})
\end{align*}
The claim follows because elements of the form $v\pi^{\gamma/n}$ generate $\Ftr$ over $F$ by Proposition~\ref{prop.decomp_of_Gal}(1).
\end{proof}

Next, we recall A.\kern-.1em\ Ostrowski's foundational theorem in valuation theory. It will be used often and sometimes implicitly; see \cite[Theorem 17.2.1]{efrat2006valuations} for a modern proof.
\begin{theorem}{\cite{ostrowski1935untersuchungen}}\label{thm.ostrowski}
    Let $F\subset L$ be a finite extension of Henselian valued fields with value groups $\Gamma_F\subset \Gamma_L$ and residue fields $k\subset l$. There exists an integer $\delta_{L/F}$ such that:
    \begin{equation}\label{e.ostrowski}
        [L:F] = \delta_{L/F} [\Gamma_L:\Gamma_F][l:k] . 
    \end{equation}
    If $\Char k > 0$, then $\delta_{L/F}$ is a power of $\Char k$ and $\delta_{L/F} = 1$ otherwise. In particular, both $[\Gamma_L:\Gamma_F]$ and $[l:k]$ divide $[L:F]$.
\end{theorem}
We finish this section with a technical lemma that will be used to avoid working with infinite algebraic extensions. Recall that a finite extension $L/F$ is tamely ramified if and only if $[\Gamma_L:\Gamma_F]$ is coprime to $\Char k$, $l/k$ is separable and $\delta_{L/F} = 1$ in \eqref{e.ostrowski}.
\begin{lemma}\label{lem.limit_of_ext}
    Let $(F,\nu)$ be an ascending union of a countable chain of valued field $F_1\subset F_2 \subset \dots $  and let $L/F$ be a finite tamely ramified Galois extension. For all large enough $i$ there exists a tamely ramified Galois extension $L_i/F_i$ such that $L = L_i \otimes_{F_i} F$  and the restriction map $\Gal(L/F)\to \Gal(L_i/F_i)$ is an isomorphism. 
\end{lemma}
\begin{proof}
    Let $x\in L$ be an element such that $L = F(x)$ and let $P(t)\in F[t]$ be the minimal polynomial of $x$. For large enough $i$, $P(t)\in F_i[t]$ and so $P(t)$ is also the minimal polynomial of $x$ over $F_i$. Set $L_i := F_i(x)$ and note $$[L_i:F_i] = \deg P =[L:F].$$ By enlarging $i$ further we can make $P(t)$ factor fully over $L_i$ so that $L_i/F_i$ is Galois. Let $l_i/k_i$ denote the residue field extension of $L_i/F_i$. A similar argument shows that $l_i/k_i$ is Galois of degree $[l:k]$ for large enough $i$. One also easily checks that for large enough $i$ we have $[\Gamma_{L_i}:\Gamma_{F_i}] = [\Gamma_L : \Gamma_F]$. Combining all of this together, we find that for large $i$:
    $$[L_i: F_i] = [L:F] = [\Gamma_L:\Gamma_F][l:k] = [\Gamma_{L_i}:\Gamma_{F_i}][l_i:k_i].$$
    Therefore $L_i/F_i$ is defectless (i.e. $\delta_{L_i /F_i} = 1$) and tamely ramified.   The restriction map $\Gal(L/F)\to \Gal(L_i/F_i)$ is an isomorphism because it is injective and $$|\Gal(L/F)|=[L:F] =[L_i :F_i] = |\Gal(L_i/F_i)|.$$
    Indeed, any two different elements $\sigma, \tau\in \Gal(L/F)$ differ on $L_i$ for large enough $i$.
\end{proof}

\section{Preliminaries about anisotropic torsors}\label{sect.pre_isotrop}

In this final section of preliminaries, we record a couple of facts about versal torsors and anisotropic torsors. We first recall the definition of a versal torsor.
\begin{definition}\label{def.versal}
    A versal torsor $\gamma \in H^1(l,G)$ is the generic fiber of a $G$-torsor for the \'etale topology $T\to B$ such that:
    \begin{enumerate}
        \item $B$ is a geometrically irreducible smooth variety over $k_0$.
        \item For any field extension $E/k_0$ with $|E|=\infty$ and any open dense subset $U\subset B$, there exists $u\in U(E)$ such that:
        $$\gamma \cong T_u := T \times_u \Spec E.$$
    \end{enumerate}
\end{definition}
Note that there exists a versal $G$-torsor \cite[Section 5.3]{serre-ci}.
\begin{definition}
    A $G$-torsor $T$ over a $k_0$-scheme $B$ is called isotropic if there exists an embedding $\Gm \to {}_T G$ over $B$, where ${}_T G = \Aut_G(T)$ is the group of $G$-equivariant automorphisms of $T$; see \cite[Page 6]{gille2016classification}.
\end{definition} 
The following spreading out lemma shows that the notion of isotropy plays well with inductive limits of base rings. 

\begin{lemma}\label{lem.isotrop_induct_lim}
    Let $T\to B$ be a $G$-torsor over an irreducible variety $B$ over $k_0$.
    \begin{enumerate}
        \item If $T$ is isotropic, then so is $T_u$ for any  $k_0$-algebra $R$ and $u\in B(R)$.
        \item Let $(\Lambda ,\leq)$ be a partially ordered filtered set with minimal element $0\in \Lambda$. Let $(R_\lambda)_{\lambda\in \Lambda}$ be an inductive system of $k_0$-algebras over $\Lambda$ with transition maps $\sigma_{\lambda \mu} :R_\lambda \to R_{\mu}$ for all $\lambda \leq \mu$. Let $u_0: \Spec R_0 \to B$ be a point and set $u_\lambda = u_0\circ \sigma_{0 \lambda}$ for any $\lambda\in \Lambda$. Assume $(R_\lambda)_{\lambda\in\Lambda}$ has an inductive limit $R$ and let $u = u_0 \circ \iota$, where $\iota: \Spec R \to \Spec R_0$ is the canonical morphism.  If $T_u$ is isotropic, then $T_{u_\lambda}$ is isotropic for some $\lambda\in \Lambda$.
        \item Let $u \in B(L)$ be the generic point of $B$. If $T_u$ is isotropic, then there exists an open $U\subset B$ such that $T_U$ is isotropic.
        \item Let $\gamma \in H^1(k,G)$ be a torsor over a field extension $k/k_0$. If $\gamma_l$ is isotropic for some $l/k$, then $\gamma_{l'}$ is isotropic for some finitely generated subextension $k\subset l'\subset l$.
    \end{enumerate}
\end{lemma}
\begin{proof}
Let $G' = {}_T G$ be the twisted group over $B$ defined by $T$. For any $u\in B(R)$, let $G'_u = {}_{T_u} G$ denote the group $G'\times_{u} \Spec R$ over $\Spec R$. 

(1) \ If $T$ is isotropic, then there exists an embedding $\mathbb{G}_{m,B} \subset G'$. This embedding specializes to an embedding $\mathbb{G}_{m,R} \subset G'_u$ and so $T_u$ is  isotropic. 

\smallskip
(2) \  Since $T_u$ is isotropic, there exists an embedding $\mathbb{G}_{m,R} \subset G'_u$. This embedding is induced from an embedding $\mathbb{G}_{m,R_\lambda}\subset G'_{u_\lambda}$ for some $\lambda\in \Lambda$ by \cite[Lemma 10.62]{gortz2010algebraic}. Therefore $T_{u_\lambda}$ is isotropic. 

\smallskip
(3) \ Assume without loss of generality that $B = \Spec R$ for some integral domain $R$ of finite type over $k_0$. Let $L$ be the fraction field of $R$. The generic point $u$ corresponds to the inclusion $R\subset L$ . Since $L$ is the inductive limit of all localization $R[f^{-1}]$ for $f\in R\setminus\{0\}$, the result follows from Part (2).

(4) \ The result follows immediately from Part (2) because $l$ is the inductive limit of all finitely generated subextensions $k\subset l'\subset l$.
\end{proof}

Recall that a field $k$ is called \emph{$p$-closed} (or \emph{$p$-special}) for some prime $p$ if the degree of any finite extension of $k$ is a power of $p$. An algebraic extension $k^{(p)}/k$ is called a \emph{$p$-closure of }$k$ if $k^{(p)}$ is $p$-closed and any finite subextension $k\subset k'\subset k^{(p)}$ is of prime-to-$p$ degree.  For any prime $p\neq \Char k$, the fixed field of any $p$-Sylow subgroup of $\Gal(k_{\mathrm{perf}})$ is a  $p$-closure of $k$; see \cite[Proposition 101.16]{ekm}.

\begin{proposition}\label{prop.versal_aniso}
    Let $\gamma \in H^1(k,G)$ be a versal torsor over some field $k/k_0$.
    \begin{enumerate}
        \item If $G$ admits an anisotropic torsor $\eta$ over some other field $k_0\subset l$, then $\gamma$ is anisotropic.
        
        \item Let $p\neq \Char k_0$ be a prime. If $G$ admits an anisotropic torsor $\eta$ over some $p$-closed field $k_0\subset l$, then $\gamma_{k^{(p)}}$ is anisotropic for any $p$-closure $k^{(p)}/k$.
    \end{enumerate}
    
\end{proposition}
\begin{proof}
    Assume $\gamma$ is the generic fiber of a $G$-torsor $T\to B$ as in Definition~\ref{def.versal}. Let $\eta$ be an arbitrary $G$-torsor over a field extension $k_0\subset l$. 

    (1) \   Let $u\in B(k)$ be a generic point of $B$ and assume that $\gamma = T_u$ is isotropic.  By Lemma~\ref{lem.isotrop_induct_lim}(3), there exists a dense open $U\subset B$ such that $T_U$ is isotropic. We may assume $|l|=\infty$ because passing to $l(\!(t)\!)$ does not affect whether $\eta$ is anisotropic or not; see e.g. \cite[Proposition 4.9]{gille2024newloop}. By versality, there exists a point $v\in U(l)$ such that $T_{v} = \eta$. Therefore $\eta$ is isotropic by Lemma~\ref{lem.isotrop_induct_lim}(1). This proves Part (1). 

\smallskip
    (2) \ Assume $l$ is $p$-closed. If $\gamma_{k^{(p)}}$ is isotropic, then $\gamma_{k'}$ is isotropic for some prime-to-$p$ extension $k\subset k'\subset k^{(p)}$ by Lemma~\ref{lem.isotrop_induct_lim}(4). The extension $k\subset k'$ is induced from a morphism $f: V \to B$ of varieties for some irreducible variety $V$ over $k_0$ with function field $k_0(V)=k'$. By Lemma~\ref{lem.isotrop_induct_lim}(3), we may replace $V$ with an open subset to assume $T_{V} = T\times_f V$ is isotropic.  Since $f$ is generically finite, there exist dense opens $U\subset B$, $W\subset V$ such that $f(W)\subset U$ and the restriction $f:W\to U$ is finite and flat \cite[Tag 02NX]{stacks-project}. Replace $B$ by $U$ and $V$ by $W$ to assume $f$ is finite and flat.
    Now let $u\in B(l)$ be a point such that $T_u = \eta$ (note that $l$ is infinite because it contains all roots of unity of order coprime to $p$ and $\Char k_0$). Since $f$ is flat, finite of degree prime-to-$p$ and $l$ is $p$-closed, we can lift $u$ to a point $v\in V(l)$ such that $$f(v) = u.$$ 
    See Lemma~\ref{lem.finite_lift} below. Associativity of fiber products gives a canonical $G$-equivariant isomorphism:
    $$T_u =T \times_{u} \Spec l \cong T_{V} \times_v \Spec l.$$
    Therefore $$\eta = T_u = T_{V}\times_{v} \Spec l$$ is isotropic by Lemma~\ref{lem.isotrop_induct_lim}(1).
\end{proof}

\begin{lemma}\label{lem.finite_lift}
    Let $k_0\subset l$ be a $p$-closed field for some prime $p\neq \Char k_0$. Let $f: V\to U$ be a finite flat map of varieties over $k_0$.  If the degree of $f$ is prime-to-$p$, then the induced map $V(l) \to U(l)$ is surjective.
\end{lemma}
\begin{proof}
    Let $u: \Spec l\to U$ be a point. The scheme $\Spec l \times_U V$ is finite of degree coprime to $p$ over $\Spec l$ by our assumption on $f$. Therefore $\Spec l \times_U V = \Spec R$ for some finite $l$-algebra $R$ of dimension coprime to $p$. By the structure theorem for Artin rings there exist local Artin rings $R_1,\dots,R_n$ such that:
    $$R\cong R_1\times \dots \times R_n$$
    as $l$-algebras. Comparing dimensions, we see that $\dim_l R_i$ is coprime to $p$ for some $1\leq i\leq n$. Let $\fm\subset R_i$ be its maximal ideal and $d = [R_i/\fm : l]$ the residue field degree. By \cite[Lemma A.1.3]{fulton2013intersection}, we have:
    $$\dim_l(R_i) = d \len(R_i),$$
    where $\len(R_i)$ denotes the length of $R_i$ as a module over itself. In particular, $d$ is coprime to $p$. Since $l$ has no prime-to-$p$ extensions, it follows that $R_i/\fm = l$. The composition:
    $$R \tilde{\to} R_1 \times \dots \times R_n \twoheadrightarrow R_i \to R_i/\fm = l,$$
    gives a section of the inclusion $l\subset R$. Dualizing, we obtain a section of the left column of the following Cartesian square:
    \[
\begin{tikzcd}
\Spec R = \Spec l \times_U V \arrow[r, "f^*(u)"] \arrow[d] & V \arrow[d, "f"] \\
\Spec l \arrow[u, "s", bend left] \arrow[r, "u"] & U
\end{tikzcd}
\]
    The composition $f^*(u) \circ s : \Spec l\to V$ is a lift of $u$ to $V(l)$. 
\end{proof}

 The next lemma will be used in Section~\ref{sect.reduction_to_reductive} to reduce the proof of Theorem~\ref{thm.main} to the case where $G^{\circ}$ is reductive.
\begin{lemma}\label{lem.aniso_image}
    Let $f: G\to H$ be a homomorphism of smooth linear algebraic groups over $k_0$. Let $k_0\subset k$ be a field and denote the induced pushforward map by $f_*: H^1(k,G)\to H^1(k,H)$. If $\gamma\in H^1(k,G)$ is anisotropic and one of the following conditions hold, then $f_*(\gamma)$ is anisotropic.
    \begin{enumerate}
        \item The restriction of $f$ to $G^{\circ}$ is an isomorphism onto $H^{\circ}$
        \item The base field $k_0$ is perfect and $f$ is a quotient map with a unipotent kernel.
    \end{enumerate}
\end{lemma}
\begin{proof}
Assume $\gamma = [c_\sigma]$ for some cocycle $c_\sigma$ and denote its pushforward to $H$ by $f_*(c)$. Consider the homomorphism of twisted groups defined by $f$:
$${}_c f : {}_{c} G \to {}_{f_*(c)} H.$$
(1) By assumption ${}_c f$ restricts to an isomorphism ${}_c G^{\circ} \cong {}_{f_*(c)}H^{\circ}$. A split torus $T\subset {}_{f_*(c)}H$ is contained in ${}_{f_*(c)}H^{\circ}$ because it is connected. Since ${}_c G^{\circ} \cong {}_{f_*(c)}H^{\circ}$ is anisotropic, this implies  $T=\{e\}$. Therefore $[f_*(c)]=f_*(\gamma)$ is anisotropic.

\smallskip
(2) Assume $f$ is a quotient map with unipotent kernel $U\subset G$. Let $T\subset {}_{f_*(c)} H$ be a split torus and denote by $S$ the (scheme-theoretic) preimage of $T$ under ${}_c f$. Restricting ${}_c f$ to $S$ gives an exact sequence \cite[Theorem 5.55]{milne2017algebraic}:
\begin{equation}\label{e.preli_aniso_1}
1 \to U \to S \to T\to 1.
\end{equation}
Since $k_0$ is perfect and $T$ is connected the exact sequence \eqref{e.preli_aniso_1} splits by \cite[Expose XVII, Theorem 5.1.1]{SGA3}. We obtain an embedding $T\to S \subset {}_c G$. Since ${}_c G$ is anisotropic, we conclude $T$ must be trivial. This shows $[f_*(c)]=f_*(\gamma)$ is anisotropic.

\end{proof}

\section{Loop torsors}\label{sect.loop_torsors}

Loop torsors were introduced by Gille and Pianzola in the monograph \cite{gille2013torsors}. They can be defined in a few equivalent ways. The useful point of view for us will be:

\begin{definition}
    A torsor $\gamma\in H^1_{\tr}(F,G)$ is called a \emph{loop torsor} if it can be represented by a $\Galtr(F)$-cocycle taking values in $G(\Oin)$. We denote the subset of all loop torsors by $H^1_{\lop}(F,G) \subset H^1_{\tr}(F,G)$.
\end{definition}

We use loop torsors for two reasons. They can be decomposed in tandem with the decomposition of $\Galtr(F)$, and they are integral, so one can apply the homomorphism $G(\Oin)\to G(\ksep)$ to obtain $G(\ksep)$-valued cocycles from them. We start by describing the decomposition of loop torsors introduced in \cite[Section 3.3]{gille2013torsors}. 

Let $\Psi^F_\pi : \Gal(k)\ltimes \I{F} \to \Galtr(F)$ be the isomorphism induced by a uniformizer $\pi$ as in Proposition~\ref{prop.decomp_of_Gal}.
Any loop cocycle $c_\tau \in Z^1(\Galtr(F),G(\Oin))$ defines a $\Gal(k)$-cocycle $a_\sigma\in Z^1(\Gal(k),G(\Oin))$ and a homomorphism $\vphi: \I{F} \to G(\Oin)$ by the formulas:
$$a_\sigma = c_{\Psi_\pi(\sigma,0)}, \vphi(x) = c_{\Psi_\pi(1,x)}.$$
Note that $\vphi$ is a homomorphism because the tame inertia group of $F$ acts fixes $G(\Fin)$. The cocycle $a_\sigma$ is called \emph{the arithmetic part} of $c_\tau$ and $\vphi$ is called its \emph{geometric part}. Clearly $c_\tau$ is uniquely determined by $a_\sigma$ and $\vphi$. We will use the notation:
$$c_\tau = \langle a_\sigma,\vphi\rangle_\pi,$$
to denote a loop cocycle $c_\tau$ with arithmetic part $a_\sigma$ and geometric part $\vphi$. The corresponding loop torsor will be denoted:
$$[a_\sigma, \vphi]_\pi := [\langle a_\sigma,\vphi\rangle_\pi] \in H^1_{\lop}(F,G).$$
Any torsor $\gamma\in H^1_{\lop}(F,G)$ is by definition of the form $\gamma=[a_\sigma,\vphi]_\pi$ for some $a_\sigma,\vphi$, but the symbol $\langle a_\sigma,\vphi\rangle_\pi$ is not defined for an arbitrary pair $a_\sigma,\vphi$.
\begin{definition}\label{def.compatible}
We will call a cocycle $a_\sigma \in Z^1(\Gal(k),G(\Oin))$ and homomorphism $\vphi: \I{F}\to G(\Oin)$ \emph{compatible} if there exists a cocycle $c_\tau\in Z^1(\Galtr(F),G(\Oin))$ such that $c_\tau = \langle a_\sigma, \vphi\rangle_\pi$.
\end{definition}
Let $\theta: \Gal(k)\to (\hat{\bZ}')^*$ be the cyclotomic character. One can check that a cocycle $a_\sigma$ and a homomorphism $\vphi$ as above are compatible if and only if they satisfy:
\begin{equation}\label{e.compatibility}
    a_\sigma {}^\sigma \vphi(f) a_{\sigma}^{-1} = \vphi(f)^{\theta(\sigma)},
\end{equation}
for all $f\in \I{F},\sigma\in \Gal(k)$; see \cite[Lemma 3.7]{gille2013torsors} for a similar computation. Using this, we see that centralizers of finite abelian subgroups are a natural source for loop torsors.
\begin{example}\label{ex.main}
    Let $\vphi: \I{F} \to G(\Oin)$ be a continuous homomorphism onto $A(\Oin)$ for some finite diagonalizable $\cO$-subgroup $A\subset G_{\cO}$ and let $a_\sigma \in Z^1(\Gal(k),C_G(A)(\Oin))$ a cocycle. The formula:
    $$c_{\Psi_\pi(\sigma,f)} = a_\sigma {}^{\sigma} \vphi(f)$$
    defines a loop cocycle $c_\tau = \langle a_\sigma,\vphi\rangle_\pi$ because $${}^{\sigma}\vphi(f) = \vphi(f)^{\theta(\sigma)}$$ for all $\sigma\in \Gal(k), f\in \I{F}$ by \eqref{def.split_diag}. That is, $a_\sigma$ and $\vphi$ are compatible in the sense of Definition~\ref{def.compatible}.
\end{example}


One can often assume a loop torsor is of the form given in Example~\ref{ex.main} using the next lemma.
\begin{lemma}\label{lem.centralize_over_residue}
    Let $c_\tau = \langle a_\sigma ,\vphi \rangle_\pi$ be a loop cocycle. If the group
    $$\overline{A}(\ksep):= \{ \overline{\vphi(f)} \in G(\ksep) \mid f\in\I{F}\}$$
    is the $\ksep$-points of a diagonalizable $k$-subgroup $\ol{A}\subset G_k$, then $\overline{a_\sigma}\in C_G(\ol{A})(\ksep)$ for all $\sigma\in \Gal(k)$.
\end{lemma}
\begin{proof}
Let $\theta: \Gal(k)\to (\hat{\bZ}')^*$ denote the cyclotomic character. Since $c_\tau$ is a cocycle, \eqref{e.compatibility} gives:
$$a_\sigma {}^\sigma \vphi(f) a_\sigma^{-1} = \vphi(f)^{\theta(\sigma)}.$$
Reducing modulo the maximal ideal of $\Oin$ we get:
$$\overline{a_\sigma} {}^{\sigma} \overline{\vphi(f)} \overline{a_\sigma}^{-1} = \overline{\vphi(f)}^{\theta(\sigma)}.$$
We have $${}^{\sigma} \overline{\vphi(f)} = \overline{\vphi(f)}^{\theta(\sigma)}$$ because $\overline{A}$ is diagonalizable by Fact \ref{fact.split_diag}. Therefore $\overline{a_\sigma}$ centralizes ${}^{\sigma} \overline{\vphi(f)}$. Since $f\in \I{F}$ was arbitrary, we conclude that $\overline{a_\sigma} \in C_G(\ol{A})(\ksep)$.
\end{proof}

We finish this section by examining when two loop cocycles $\langle a_\sigma,\vphi\rangle_{\pi},\langle a_\sigma,\vphi\rangle_{\pi}$ give rise to the same class in $H^1_{\tr}(F,G(\Oin))$.
\begin{lemma}\label{lem.cohomolog_criterion}
    Two loop cocycles $\langle a_\sigma,\vphi\rangle_\pi, \langle b_\sigma,\psi\rangle_\pi$ are cohomologous in $Z^1_{\tr}(F,G(\Oin))$ if and only if there exists $s\in G(\Oin)$ such that:
    \begin{equation}\label{e.27}
    s^{-1} a_\sigma {}^\sigma s = b_\sigma , \ \ s^{-1} \vphi  s =\psi.
    \end{equation}
\end{lemma}
\begin{proof}
    For any $s\in G(\Oin)$, $\sigma\in \Gal(k),f\in \I{F}, \tau:= \Psi_\pi(\sigma,f)$ we have:
    \begin{align*}
        s^{-1} \langle a_\sigma , \vphi\rangle_\pi (\tau) {}^\tau s &=s^{-1} a_\sigma {}^\sigma \vphi(f) {}^{\sigma} s \\
        &=s^{-1} a_\sigma {}^\sigma s {}^{\sigma} (s^{-1} \vphi(f) s) = \langle s^{-1} a_\sigma {}^\sigma s , s^{-1}\vphi s\rangle_\pi(\tau).  
    \end{align*}
    This proves the claim.
\end{proof}

\section{A theorem of Gille and Pianzola}\label{sect.a_theorem_of_GP}
Recall that a valued field $(F,\nu)$ is an iterated Laurent series field, if it is isomorphic to $k(\!(t_1)\!)\dots(\!(t_r)\!)$ equipped with the usual $(t_1,\dots,t_r)$-adic valuation as in \eqref{e.def_of_t-adic_val}.
In this section we prove the following adaptation of a theorem of Gille and Pianzola.

\begin{theorem}\label{thm.uniqueness}
    Let $(F,\nu)$ be an iterated Laurent series field and assume $G^{\circ}$ is reductive. Two anisotropic loop torsors $[a_\sigma,\vphi]_\pi,[b_\sigma,\psi]_\pi \in H^1_{\lop}(F,G)_{\an}$ are equal if and only if there exists $s\in G(\ksep)$ such that:
    \begin{equation}\label{e.28}
    s^{-1} \overline{a_\sigma} {}^\sigma s = \overline{b_\sigma}, \ \ s^{-1} \overline{\vphi(f)} s = \overline{\psi(f)} 
    \end{equation}
    for all $f\in \I{F}, \sigma\in \Gal(k)$.
\end{theorem}

Since $F$ is an iterated Laurent series field, it admits a coefficient field. That is, there exists an inclusion $k\subset \cO$ splitting the map onto the residue field $\cO \to k$. This inclusion extends to an embedding $\ksep \subset \Oin$ splitting the homomorphism $\Oin\to \ksep$, which in turn gives a section $G(\ksep)\subset G(\Oin)$ of the homomorphism $G(\Oin) \to G(\ksep)$.  Gille proved Theorem~\ref{thm.uniqueness} in case $a_\sigma,\vphi,b_\sigma,\psi$ all take values in $G(\ksep)\subset G(\Oin)$ \cite[Corollary 4.11]{gille2024newloop} (before that Gille-Pianzola proved it under the additional assumption $\Char k =0$ in \cite{gille2013torsors}). 
Thus in order to prove Theorem~\ref{thm.uniqueness} it suffices to establish the following proposition:

\begin{proposition}\label{prop.hensel_coeffic_field2}
    Let $G(\ksep)\subset G(\Oin)$ denote the inclusion induced from the inclusion of a coefficient field $\ksep\subset \Oin$. The corresponding map of cohomology sets 
    \begin{equation}\label{e.reduction_map_coh}
        H^1_{\tr}(F,G(\ksep))\to H^1_{\tr}(F,G(\Oin))
    \end{equation} is a bijection. 
\end{proposition}
For the proof of this proposition we will need the following consequences of Hensel's lemma.

\begin{lemma}\label{lem.conj_of_morphism}
    Let $\vphi,\psi: \I{F}\to G(\Oin)$ be two continuous homomorphisms and let $\overline{\vphi},\overline{\psi}: \I{F}\to G(\ksep)$ be the composition of $\vphi,\psi$ with the reduction homomorphism $G(\Oin) \to G(\ksep), g\mapsto \overline{g}$. If $s \overline{\vphi} s^{-1}= \overline{\psi}$ for some $s\in G(\ksep)$, then $\tilde{s} \vphi \tilde{s}^{-1} = \psi$ for some $\tilde{s}\in G(\Oin)$.
\end{lemma}
\begin{proof}
    Let $\I{F}\to A$ be a large enough finite quotient of $\I{F}$ so that both $\vphi$ and $\psi$ factor through $A$. Let $\vphi',\psi' : A\to G_{\Oin}$ denote the induced homomorphisms and denote by $\Tran_G(\vphi',\psi')$ the {transporter} of $\vphi'$ and $\psi'$. This is a closed subscheme of $G_{\Oin}$ such that for any ring homomorphism $\Oin \to R$:
    $$\Tran_G(\vphi',\psi')(R) = \{ g\in G(R) \mid g \vphi'_{R} g^{-1} =\psi'_R\}.$$
    Note that $A$ is a finite abelian group of order coprime to $\Char k$ by Remark~\ref{rem.quotients_tame_inertia}. Therefore $\Tran_G(\vphi',\psi')$ is smooth because $A$ is of multiplicative type; see \cite[Expose XI, Corollaire 5.2]{SGA3}. By assumption $\Tran_G(\vphi',\psi')(\ksep)\neq \emptyset$. It follows from Hensel's lemma that $\Tran_G(\vphi',\psi')(\Oin)\neq \emptyset$ and the result follows. 
\end{proof}

Next we restate \cite[Expose XXIV,Proposition 8.1]{SGA3} in the language of Galois cohomology.

\begin{lemma}\label{lem.bijectivity_passing_special_fiber}
    For any smooth algebraic group $C$ over $\cO$, the reduction homomorphism $C(\Oin)\to C(\ksep)$ induces a bijection on Galois cohomology sets:
    $$H^1(\Gal(k),C(\Oin))\to H^1(\Gal(k),C(\ksep)).$$
\end{lemma}
\begin{proof}
    Note that $\Oin$ is the strict Henselization of $\cO$. Any $C$-torsor for the \'etale topology over $\Spec\cO$ is split by the universal cover $\Spec\Oin \to \Spec\cO$ because $\Spec \cO$ is local. The fundamental group of $\Spec \cO$ is $\Gal(k)$ by \cite[Section 5, Example 5.2.d]{milne1980etale}. Therefore \cite[Section 2.2]{gille2016classification} gives isomorphisms  $$H^1(\Gal(k),C(\Oin))\tilde{\to} H^1_{\et}(\Spec\cO,C),\  \ H^1(\Gal(k),C(\ksep)) \tilde{\to} H^1_{\et}(\Spec k,C).$$ 
    Now the  lemma follows from \cite[Expose XXIV,Proposition 8.1]{SGA3}, which states that the bottom row in the following commutative square is a bijection:
\[\begin{tikzcd}
	{H^1(\Gal(k),C(\Oin))} & {H^1(\Gal(k),C(\ksep))} \\
	{H^1_{\et}(\Spec \cO, C)} & {H^1_{\et}(\Spec k,C)}
	\arrow[from=2-1, to=2-2]
	\arrow[from=1-1, to=2-1]
	\arrow[from=1-2, to=2-2]
	\arrow[from=1-1, to=1-2]
\end{tikzcd}\]
\end{proof}

Next we prove that any loop torsor is represented by a cocycle taking values in $G(\ksep)$. Note that $\Galtr(F)$ acts on the coefficient field $\ksep\subset \Oin$ and on $G(\ksep)\subset G(\Oin)$ because $k\subset F$ is fixed by $\Galtr(F)$ and $\ksep/k$ is Galois. Let $[a_\sigma,\vphi]_\pi \in H^1_{\lop}(F,G)$ be a loop torsor. Denote the images of $a_\sigma$ under the reduction homomorphism $G(\Oin)\to G(\ksep)$ by $\overline{a}_\sigma$ and define $\overline{\vphi}:\I{F}\to G(\ksep)$ by the formula
$$\overline{\vphi}(f) = \overline{\vphi(f)}$$
for all $f\in \I{F}$. By Remark~\ref{rem.equivariance}, the map :
$$\langle \overline{a}_\sigma,\overline{\vphi}\rangle_\pi: \Gal(k)\ltimes \I{F} \to G(\Oin), \  (\sigma,f) \mapsto \overline{a}_\sigma {}^\sigma \overline{\vphi}(f)$$
is a $\Galtr(F)$-cocycle in $G(\Oin)$. Therefore $\overline{a}_\sigma,\overline{\vphi}$ are compatible in the sense of Definition~\ref{def.compatible} and the loop torsor $[\overline{a}_\sigma,\overline{\vphi}]_\pi\in H^1_{\lop}(F,G)$ is defined. 
\begin{lemma}\label{lem.Hensel_coefficient_field}
    Let $G(\ksep)\subset G(\Oin)$ denote the inclusion induced from the inclusion of a coefficient field $\ksep\subset \Oin$. Any loop cocycle $\langle a_\sigma, \vphi\rangle_\pi$ is cohomologous in $H^1_{\tr}(F,G(\Oin))$ to $\langle \overline{a}_\sigma,\overline{\vphi}\rangle_\pi$. In particular, we have:
    $$[a_\sigma,\vphi]_\pi = [\overline{a}_\sigma,\overline{\vphi}]_\pi.$$
\end{lemma}

\begin{proof}
    Let $\langle a_\sigma,\vphi\rangle_\pi$ be a loop cocycle. We need to prove  $\langle a_\sigma,\vphi\rangle_\pi$ and  $\langle \overline{a}_\sigma,\overline{\vphi}\rangle_\pi$ are cohomologous in $H^1_{\tr}(F,G(\Oin))$. By Lemma~\ref{lem.cohomolog_criterion}, it suffices to find $s\in G(\Oin)$ such that:
    \begin{equation}\label{e.goal_lem_Hens_coef_field}
    s^{-1}  a_\sigma {}^{\sigma} s= \overline{a}_\sigma, \ \ s^{-1} \vphi(f) s = \overline{\vphi}(f)
    \end{equation}
     for all $\sigma\in \Gal(k), f\in \I{F}$.
    We start by tackling the special case where $\vphi(f)\in G(\ksep)$ for all $f$ and so $\overline{\vphi} = \vphi$. Define for all $\sigma\in \Gal(k)$: $$c_\sigma = \overline{a_\sigma} a_\sigma^{-1}.$$ It is simple to check that $c_\sigma$ is a cocycle in ${}_a G(\Oin)$ (here ${}_a G$ is the twisted group defined by $a$, see Section~\ref{sect.notation}). We check $c_\sigma$ centralizes $\im\vphi$. Since both $a_\sigma,\vphi$ and $\overline{a_\sigma},\vphi$ are compatible  \eqref{e.compatibility} gives for any $f\in \I{F}$:
    \begin{align*}
        a_{\sigma^{-1}} {}^{\sigma^{-1}}\vphi(f) a_{\sigma^{-1}}^{-1} =   \vphi(f)^{\theta(\sigma^{-1})} = \overline{a_{\sigma^{-1}}} {}^{\sigma^{-1}} \vphi(f) \overline{a_{\sigma^{-1}}}^{-1}.
    \end{align*}
    Therefore $\overline{a_{\sigma^{-1}}}^{-1} a_{\sigma^{-1}}$ centralizes ${}^{\sigma^{-1}} \vphi(f)$. Applying ${}^{\sigma}(\cdot)$ to both sides and using the cocycle identity we get:
    $$C_{{}_a G}(\vphi(f))(\Oin) \ni {}^{\sigma}(\overline{a_{\sigma^{-1}}}^{-1} a_{\sigma^{-1}}) = {}^{\sigma}\overline{a_{\sigma^{-1}}}^{-1}  {}^{\sigma}a_{\sigma^{-1}} = \overline{a_\sigma} a_\sigma^{-1} = c_\sigma.$$
    Since $f \in \I{F}$ was arbitrary this shows $c_\sigma\in C_{{}_a G}(\vphi)$.
    The class $[c_\sigma] \in H^1(\Gal(k),C_{{}_a G}(\vphi)(\Oin))$ clearly goes to zero under the map  $$H^1(\Gal(k),C_{{}_a G}(\vphi)(\Oin))\to H^1(\Gal(k),C_{{}_a G}(\vphi)(\ksep)).$$ 
    This implies $[c_\sigma]$ is split in $H^1(\Gal(k),C_{{}_a G}(\vphi)(\Oin))$ by Lemma~\ref{lem.bijectivity_passing_special_fiber} (note that $C_G(\vphi)$ is smooth by \cite[Expose XI, Corollaire 5.3]{SGA3}). Therefore there exists $s\in C_{{}_a G}(\vphi)(\Oin)$ such that for all $\sigma\in \Gal(k)$:
    $$\overline{a_\sigma} a_\sigma^{-1} =c_\sigma = s^{-1} a_\sigma {}^{\sigma} s a_\sigma^{-1}.$$
    Multiply by $a_\sigma$ on the right to obtain $$ s^{-1} a_\sigma {}^{\sigma} s =\overline{a_\sigma} .$$
    Since $s$ centralizes $\vphi = \overline{\vphi}$, we conclude that it satisfies \eqref{e.goal_lem_Hens_coef_field}.

    We now tackle the case where $\vphi \neq \overline{\vphi}$. By Lemma~\ref{lem.conj_of_morphism}, $\vphi$ and $\overline{\vphi}$ are conjugate by an element $s\in G(\Oin)$. That is, for any $f\in \I{F}$ we have: 
    $$s^{-1} \vphi(f)  s = \overline{\vphi}(f).$$
    Applying the reduction map $G(\Oin) \to G(\ksep), g\mapsto \overline{g}$ to the above equation we get:
    $$\overline{s}^{-1} \overline{\vphi}(f) \overline{s} = \overline{\vphi}(f).$$
    Therefore the element $t = s \overline{s}^{-1} \in G(\Oin)$ satisfies $\overline{t} =1$ and
    \begin{equation}\label{e.34}
        t^{-1} \vphi(f)  t = \overline{s} s^{-1} (f) s \overline{s}^{-1}= \overline{s} \overline{\vphi}(f) \overline{s}^{-1} = \overline{\vphi}(f).
    \end{equation}
    By Lemma~\ref{lem.cohomolog_criterion}, \eqref{e.34} implies $\langle a_\sigma,\vphi\rangle_\pi$ and $\langle t^{-1} a_\sigma {}^\sigma t, \overline{\vphi}\rangle_\pi$ represent the same class in $H^1_{\tr}(F,G(\Oin))$.  By the first case we tackled and because $\overline{t}=1$, we conclude:
     $$[\langle t^{-1} a_\sigma {}^\sigma t, \overline{\vphi}\rangle_\pi] = [\langle \overline{t^{-1}} \overline{a}_\sigma {}^\sigma \overline{t}, \overline{\vphi}\rangle_\pi] = [\langle \overline{a}_\sigma, \overline{\vphi}\rangle_\pi]  \ \text{in }H^1_{\tr}(F,G(\Oin)).$$
    Therefore $\langle  \overline{a}_\sigma, \overline{\vphi}\rangle_\pi$ and $\langle a_\sigma,\vphi\rangle_\pi$ represent the same class in $H^1_{\tr}(F,G(\Oin))$. This finishes the proof.
\end{proof}

We now prove Proposition~\ref{prop.hensel_coeffic_field2}, which completes the proof of Theorem~\ref{thm.uniqueness}. 
\begin{proof}
    The map \eqref{e.reduction_map_coh} is injective because it has a left inverse induced by the reduction homomorphism $G(\Oin)\to G(\ksep), g\mapsto \overline{g}$. To prove surjectivity, one needs to show any loop cocycle is cohomologous in $H^1_{\tr}(F,G(\Oin))$ to a cocycle taking values in $G(\ksep)$. This is immediate from Lemma~\ref{lem.Hensel_coefficient_field}.
\end{proof}

\section{Functoriality of decompositions of loop torsors}\label{sect.funct_of_decomp}

Let $L/F$ be an extension of Henselian valued fields with uniformizers $\pi$ and $\tau$. It is clear from the definition that $\gamma_L$ is a loop torsor for any $\gamma\in H^1_{\lop}(F,G)$. Therefore there is an induced map $H^1_{\lop}(F,G)\to H^1_{\lop}(L,G)$. Our goal in this section is to understand the functoriality of the decompositions $\gamma = [a_\sigma,\vphi]_\pi$. That is, we wish to write $\gamma_L = [b_\sigma,\psi]_\tau$ for some compatible $b_\sigma,\psi$ explicitly described in terms of $a_\sigma$ and $\vphi$. The following lemma shows we have to account for the choice of uniformizers $\pi,\tau$ (note that it cannot always be assumed that $\tau$ extends $\pi$). 

\begin{proposition}\label{prop.basic_func_decomp}
    For any $\sigma\in \Gal(l)$, let $\chi_\sigma \in \I{F}$ be as in \eqref{e.25} (note that $\chi_\sigma$ depends on $\pi$ and $\tau$). The following holds for any compatible pair $a_\sigma,\vphi$.
     $$([a_\sigma,\vphi]_{\pi})_L = [\vphi(\chi_\sigma)\Inf_{l/k}(a)_\sigma ,\vphi_{|L}]_\tau.$$ 
     Here $\Inf_{l/k}(a)$ is the restriction of $a_\sigma$ to $l$ and $\vphi_{|L}$ is the composition of $\vphi$ with the restriction map $\I{L}\to \I{F}$.
\end{proposition}
\begin{proof}
    Let $\Psi^F_\pi: \Gal(k)\ltimes \I{F} \to \Galtr(F)$, $\Psi^L_\tau: \Gal(l)\ltimes \I{L} \to \Galtr(L)$ be the isomorphisms of Proposition~\ref{prop.decomp_of_Gal}. 
    Lemma~\ref{lem.func_Galois_dec} gives for any $\sigma\in \Gal(l),f\in \I{L}$:
    \begin{align*}
        \langle a_\sigma, \vphi\rangle_{\pi}(\Psi^L_\tau(\sigma,f)_{|\Ftr}) &= \langle a_\sigma, \vphi\rangle_{\pi}(\Psi^F_\pi(\sigma_{|\ksep},f_{|F} + \theta(\sigma)^{-1}\chi_\sigma)) \\
         &= a_{\sigma_{|\ksep}} {}^{\sigma} ( \vphi(\chi_\sigma)^{\theta(\sigma)^{-1}} \vphi(f_{|F})) \\
        &= a_{\sigma_{|\ksep}} {}^{\sigma}  \vphi(\chi_\sigma)^{\theta(\sigma)^{-1}} {}^{\sigma} \vphi(f_{|F}) \\
        &=^{\eqref{e.compatibility}} \vphi(\chi_\sigma) a_{\sigma_{|\ksep}} {}^{\sigma} \vphi(f_{|F}) \\
        &=\langle \vphi(\chi_\sigma) \Inf_{l/k}(a)_\sigma ,\vphi_{|L}\rangle_{\tau}(\Psi^L_\tau(\sigma,f)).
    \end{align*}
    The result follows.
\end{proof}

\begin{corollary}\label{cor.field_ext_funct}
   If $\Gamma_{F} = \Gamma_L$, then the following holds for any compatible pair $a_\sigma,\vphi$:
     $$([a_\sigma,\vphi]_{\pi})_L = [\Inf_{l/k}(a)_\sigma,\vphi]_\pi.$$ 
    Note that $\pi$ is a uniformizer for $\Ltr$ because $\Gamma_{\Ftr} = \Gamma_{\Ltr}$.
\end{corollary}
\begin{proof}
    The result is immediate from Proposition~\ref{prop.basic_func_decomp} once we note that $\chi_\sigma = 1$ for all $\sigma\in \Gal(l)$ by  \eqref{e.25}.
\end{proof}

To improve Proposition~\ref{prop.basic_func_decomp} we need to understand the cocycle $\vphi(\chi_\sigma) \in Z^1(\Gal(l),A(l))$ for arbitrary uniformizers $\tau$ of $L$ and finite diagonalizable groups $A$ over $k$.

\begin{lemma}\label{lem.undertanding_chisigma}
    Let  $A=\mu_{n_1} \times \dots \times \mu_{n_r}$ for some $n_1,\dots,n_r$ coprime to $\Char k$ and $\vphi: \I{F} \to A(\Oin)$ be a continuous homomorphism. 
    \begin{enumerate}
        \item There are elements $\gamma_1,\dots,\gamma_r\in \Gamma_F$ such that for all $f\in \I{F}$:
    \begin{equation}
        \vphi(f) = (\zeta_{n_1}^{f(\gamma_1)},\dots,\zeta_{n_r}^{f(\gamma_r)}).
    \end{equation}
    
        \item Let $u^\gamma = \pi^{\gamma} \tau^{-\gamma}$ for $\gamma\in \Gamma_{\Ftr}$ as in Lemma~\ref{lem.func_Galois_dec}. The class $[\vphi(\chi_\sigma)]\in H^1(l,A)$ represents $(\overline{u^{\gamma_1}},\dots,\overline{u^{\gamma_r}})$ under the Kummer isomorphism $ H^1(l,A)\cong l^*/l^{*n_1}\times \dots\times l^*/l^{*n_r}$.

    \end{enumerate}
    
\end{lemma}
\begin{proof}
    One proves Part (1) by choosing a basis for $\Gamma_F$. The details are left to the reader. We prove Part (2). Recall that by the definition of $\chi_\sigma$ in \eqref{e.25}, we have:
    $$\zeta_{n_i}^{\chi_\sigma(\gamma_i)} = \frac{\overline{s}_\tau(\sigma)(u^{\gamma_i/n_i})}{u^{\gamma_i/n_i}},$$
    for any $1\leq i \leq r$.
    Note that $\nu(u^{\gamma_i/n_i})=0$, so the residue class $\overline{u^{\gamma_i/n_i}} \in l_{\sep}$ is defined. 
    Since $\overline{s}_\tau$ is a section of $\Galtr(L)\to \Gal(l)$, applying \eqref{e.residue_Gal_iso} gives:
    \begin{equation}\label{e.notimport_dependence_on_uni}
        \zeta_{n_i}^{\chi_\sigma(\gamma_i)} = \overline{\zeta_{n_i}^{\chi_\sigma(\gamma_i)}} = \frac{\sigma(\overline{u^{\gamma_i/n_i}})}{\overline{u^{\gamma_i/n_i}}}.
    \end{equation}
     Since $\overline{u^{\gamma_i/n_i}}$ is an $n_i$-th roof of $\overline{u^{\gamma_i}}$, the right hand side in \eqref{e.notimport_dependence_on_uni} is a cocycle representing $\overline{u^{\gamma_i}}$ under the classical kummer isomoprhism $H^1(l,\mu_{n_i})\cong l^*/l^{*n_i}$.
    Now Part (2) follows from Part (1) which states:
    \[\vphi(\chi_\sigma) = (\zeta_{n_1}^{\chi_\sigma(\gamma_1)},\dots,\zeta_{n_r}^{\chi_\sigma(\gamma_r)}).\qedhere\]
\end{proof}

\begin{corollary}\label{cor.choosing_uniformizers}
    Let $\pi$ be a uniformizer for $\Ftr$ and $\vphi: \I{F}\to A(\Oin)$ a continuous homomorphism onto a finite diagonalizable $\cO$-subgroup $A\subset G_{\cO}$. If $[L:F]$ is finite and coprime to $|A|$, then there exists a uniformizer $\tau: \Gamma_{\Ltr}\to \Ltr^*$ such that $$([a_\sigma,\vphi]_\pi)_L = [\Inf_{l/k}(a)_\sigma,\vphi_{|L}]_\tau$$ for any $\Gal(k)$-cocycle $a_\sigma\in G(\Oin)$ centralizing $A$.
\end{corollary}
\begin{proof}
There exists a basis $e_1,\dots,e_r$ of ${\Gamma}_L$ and integers $d_1,\dots,d_r$ such that $d_1 e_1,\dots,d_r e_r$ is a basis for $\Gamma_F$ and $\prod_i d_i = [{\Gamma}_L:\Gamma_F]$ \cite[Theorem III.7.8]{lang2002algebra}. By Theorem~\ref{thm.ostrowski}, $[\Gamma_L:\Gamma_F]$ is coprime to $|A|$. Choose integers $m_1,\dots,m_r$ such that:
\begin{equation}\label{e.27b}
    m_i d_i \equiv 1 \mod |A|.
\end{equation}
Let $x_i \in L^*$ be such that $\nu(x_i) = e_i$. Set $\pi_i = \pi^{d_i e_i}$ and define $$\tau_i :=  x_i (\pi_i x_i^{-d_i})^{m_i}.$$
We have for any $1\leq i\leq r$ :
$$\nu(\tau_i) = \nu(x_i) + m_i (\nu(\pi_i) - d_i\nu(x_i)) = e_i + m_i(d_i e_i - d_i e_i ) = e_i.$$
Therefore there exists a uniformizer $\tau$ for $\Ltr$ such that $\tau^{e_i} =\tau_i$ for all $i$ by Remark~\ref{rem.uniformizers}. For this uniformizer we have:
$$\pi^{d_i e_i} \tau^{-d_i e_i} = \pi_i \tau_i^{-d_i} = \pi_i x_i^{-d_i} (\pi_i x_i^{-d_i})^{-d_i m_i} = (\pi_i x_i^{-d_i})^{1-d_i m_i}.$$
Therefore $\pi^{d_i e_i} \tau^{-d_i e_i}$ is an $|A|$-th power for all $i$ by \eqref{e.27b}. This gives $[\chi_\sigma] =1$ in $H^1(k,A)$ by  Lemma~\ref{lem.undertanding_chisigma}(2). Let $a\in A(\Oin)$ be such that for all $\sigma\in \Gal(k)$:
$$\chi_\sigma = a^{-1}\ {}^{\sigma} a.$$
We use Proposition~\ref{prop.basic_func_decomp} and the fact that  $a_\sigma$ centralizes $A$ to get:
$$([a_\sigma,\vphi]_\pi)_L= [ a^{-1} {}^{\sigma} a \Inf_{l/k}(a)_\sigma,\vphi_{|L}]_\tau= [a^{-1} \Inf_{l/k}(a)_\sigma {}^{\sigma} a,\vphi_{|L}]_\tau.$$
We finish the proof by noting that Lemma~\ref{lem.cohomolog_criterion} gives:
\[[a^{-1} \Inf_{l/k}(a)_\sigma {}^{\sigma} a,\vphi_{|L}]_\tau= [\Inf_{l/k}(a)_\sigma,a \vphi_{|L}a^{-1}]_\tau= [\Inf_{l/k}(a)_\sigma,\vphi_{|L}]_\tau.\qedhere\]
\end{proof}

Let $F$ be an iterated Laurent series field and $\vphi: \I{F} \to A(\ksep)$ be a continuous homomorphism onto a diagonalizable $p$-group $A\subset G_k$ defined over $k$. Let $a_\sigma \in Z^1(\Gal(k),C_G(A)(\ksep))$ be a cocycle and set
$$\gamma =[a_\sigma,\vphi]_\pi$$
as in Example~\ref{ex.main}. If $G^{\circ}$ is reductive, \cite[Proposition 4.9]{gille2024newloop} implies $\gamma$ is anisotropic if and only if $[a_\sigma]\in H^1(k,C_G(A))$ is anisotropic. We will use Corollary~\ref{cor.choosing_uniformizers} to upgrade this to a criterion deciding whether $\gamma$ is anisotropic over a $p$-closure of $F$.
\begin{lemma}\label{lem.aniso_prime-to-p}
Assume $G^{\circ}$ is reductive.
 If $k$ is $p$-closed, and $[a_\sigma]\in H^1(k,C_G(A))$ is anisotropic, then for any prime-to-$p$ extension $F\subset L$:
 \begin{enumerate}
     \item $\im\vphi_{|L} = A(\ksep)$
     \item $\gamma_L$ is anisotropic
 \end{enumerate}
\end{lemma}
\begin{proof}
    The residue field of $L$ is a prime-to-$p$ extension of $k$  by Ostrowski's theorem and so it must be $k$ because $k$ is $p$-closed.  Since $A$ is a $p$-group, Corollary~\ref{cor.choosing_uniformizers} implies that there exists uniformizer $\tau: \Gamma_{L_{\tr}} \to L_{\tr}^*$ such that:
    $$([a_\sigma,\vphi]_\pi)_L = [a_\sigma,\vphi_{|L}]_\tau.$$
    The inclusion $\Gamma_F \to \Gamma_L$ induces an exact sequence:
    $$\I{L} \to \I{F} \to \Ext^1_{\bZ}(\Gamma_L/\Gamma_F, \hat{\bZ}').$$
    Since $\Ext^1_{\bZ}(\Gamma_L/\Gamma_F,\hat{\bZ}')$ is killed by $[\Gamma_L:\Gamma_F]$, so is the cokernel of $\I{L}\to \I{F}$. In particular, this cokernel is of order coprime to $p$ because $[\Gamma_L:\Gamma_F]$ divides $[L:F]$ by Ostrowski's Theorem.  Since $\im\vphi=A(\ksep)$ is a $p$-group, it follows that $$\vphi_{|L}: \I{L}\to A(\ksep)$$ is a surjection (recall that $\vphi_{|L}$ is the composition of $\vphi$ with the restriction map $\I{L}\to \I{F}$).
    The torsor $[a_\sigma]\in H^1(k,C_G(A))$ is anisotropic by assumption. Applying \cite[Proposition 4.9]{gille2024newloop} gives that $\gamma_L=[a_\sigma,\vphi_{|L}]_\tau$ is anisotropic because $C_{G}(\vphi_{|L}) = C_{G}(A)$.
\end{proof}

\begin{corollary}\label{cor.if_cent_admits_aniso_so_does_G}
    Let $A\subset G$ be a finite diagonalizable $p$-subgroup for some prime $p\neq \Char k_0$. Assume $G^{\circ}$ is reductive and $C_G(A)$ admits an anisotropic torsor over some $p$-closed field $k_0\subset k$. There exists a valued field $k_0\subset F$ with residue field $k$ such that $G$ admits anisotropic torsors over any $p$-closure $F^{(p)}$ of $F$. 
\end{corollary}
\begin{proof}
    Set $r=\rank(A)$. Let $F = k(\!(t_1)\!)\dots(\!(t_r)\!)$ be an iterated Laurent series field equipped with its $(t_1,\dots,t_r)$-adic valuation and let $\pi$ be a uniformizer for $\Ftr$.  Since $\Gamma_F = \bZ^r$, we have $$\I{F} = \Hom(\Gamma_F,\hat{\bZ}')= \hat{\bZ}'^r.$$ Therefore there exists a surjection $\vphi : \I{F} \to A(\ksep)$. If $[a_\sigma]\in H^1(k,C_G(A))$ is anisotropic, then
    $$\gamma= [a_\sigma,\vphi]_\pi \in H^1_{\lop}(F,G)$$
    is anisotropic over any prime-to-$p$ extension $F\subset F'$ by Lemma~\ref{lem.aniso_prime-to-p}. Therefore $\gamma_{F^{(p)}}$ is anisotropic for any $p$-closure $F\subset F^{(p)}$ by Lemma~\ref{lem.isotrop_induct_lim}.
\end{proof}

\section{Almost all torsors are loop torsors}~\label{sect.all_torsors_are_loops}

In the proof of Theorem~\ref{thm.main}, we will need to show certain $G$-torsors are in fact loop torsors. To achieve this, in this section we collect propositions that show all torsors under reductive groups are loop torsors ``up to wild ramification". In particular, when $G^{\circ}$ is reductive and the characteristic of $k_0$ is good for $G$, all $G$-torsors are loop torsors.

\begin{proposition}\label{lem.loop_char_good}
    Assume $G^{\circ}$ is reductive. If the characteristic of $k_0$ is good for $G$ (see Definition~\ref{def.char}), then $$H^1_{\lop}(F,G) = H^1_{\tr}(F,G) = H^1(F,G).$$
\end{proposition}
\begin{proof}
 A torsor induced from a finite subgroup $S\subset G$ of order coprime to $\Char k_0$ is a loop torsor; see \cite[Lemma 3.12]{gille2013torsors}. If the characteristic of $k_0$ is good for $G$, then all $G$-torsors are induced from some such subgroup $S\subset G$ defined over $k_0$ \cite[Corollary 18]{lucchini2015groupe}. The result follows.
\end{proof}

The next lemma is the only occasion on which we relax Assumption~\ref{ass.finite_rank}. The proof is inspired by \cite[6.2]{gille2013torsors}.

\begin{lemma}\label{lem.loop_normalizer_torus}
    Let $(F,\nu)$ be a Henselian valued field with residue field $k$ and value group $\Gamma_F$ not necessarily finitely generated. Assume $\Ftr = \Fsep= \Falg$. If $G^\circ$ is a torus, then $$H^1_{\lop}(F,G) = H^1_{\tr}(F,G) = H^1(F,G).$$
\end{lemma}
\begin{proof}
    Set $T = G^{\circ}$ and $W = G/T$. We have an exact sequence:
    $$1\to T \to G\to W\to 1.$$
    Since $G$ is smooth, $W$ is an \'etale group over $k_0$ \cite[Expose VI, Proposition 5.5.1]{SGA3}. This implies that $W_{k_{0,\sep}}$ is a constant finite group and in particular $$W(k_{0,\sep})=W(\Oin) = W(\Fsep).$$ Therefore $T(\Oin) \subset G(\Fsep)$ is a \emph{normal} subgroup. Consider the induced exact sequence:
    \begin{equation}\label{e.8}
        1\to T(\Fsep)/T(\Oin) \to G(\Fsep)/T(\Oin) \to W(k_{0,\sep}) \to 1.
    \end{equation}
    Let $\gamma \in H^1(F,G)$ and denote by $\gamma'\in H^1(F,G(\Fsep)/T(\Oin))$ the image of $\gamma$ under the natural map:
    $$H^1(F,G) \to H^1(F,G(\Fsep)/T(\Oin)).$$
    The torsor $\gamma$ lies in the image of $H^1(F,G(\Oin))\to H^1(F,G)$ if and only if $\gamma'$ lies in the image of the map
    $$H^1(F,G(\Oin)/T(\Oin))\to H^1(F,G(\Fsep)/T(\Oin)).$$ 
    See for example \cite[Lemma 3.1]{ofek2022reduction}.
    
    The exact sequence \eqref{e.8} splits.
    Indeed, the subgroup $$G(\Oin)/T(\Oin) \subset G(\Fsep)/T(\Oin)$$ intersects $T(\Fsep)/T(\Oin)$ trivially and it surjects onto $W(k_{0,\sep})$ because $G(k_{0,\sep})\subset G(\Oin)$.  Therefore we may rewrite \eqref{e.8} as:
    \begin{equation}\label{e.9}
        1\to T(\Fsep)/T(\Oin) \to (T(\Fsep)/T(\Oin)) \rtimes (G(\Oin)/T(\Oin)) \to W(k_{0,\sep}) \to 1.
    \end{equation}
    Since \eqref{e.9} is split, to show that $\gamma$ is induced from $H^1(F,G(\Oin)/T(\Oin))$ it suffices to prove that $T(\Fsep)/T(\Oin)$ is a uniquely divisible abelian group  by Lemma~\ref{lem.divisible_semi_direct} below. Let $X_*(T)$ be the $\Gal(F)$-module of cocharacters of $T$. We have isomorphisms of $\Gal(F)$-modules:
    $$T(\Fsep) \cong X_*(T) \otimes_\bZ \Fsep^*, \ T(\Oin) \cong X_*(T) \otimes_\bZ \Oin^*.$$
    Therefore to show $$T(\Fsep)/T(\Oin) \cong X_*(T)\otimes (\Fsep^*/\Oin^*)$$ is uniquely divisible it suffices to show $ \Fsep^*/\Oin^*$ is uniquely divisible. The group $\Fsep^*/\Oin^*$ is divisible because $\Fsep^* = \Falg^*$ is divisible. To show $ \Fsep^*/\Oin^*$ is torsion-free,  let $\alpha \in \Fsep^*$ be such that $\alpha^n \in \Oin^*$. Our goal is to prove $\alpha\in \Oin^*$. Let $n = q^r m$ for some $m$ coprime to $q=\Char k$ and integer $r$. The polynomial $x^m - \alpha^{n}$ is separable and splits fully over the residue field of $\ksep$ of $\Fin$. By Hensel's lemma, it follows that all roots of $x^m - \alpha^{n}$ lie in $\Oin^*$. In particular, $\alpha^{q^r}\in \Oin^*$. Since $F$ is perfect,  so is $\Fin$. Therefore the unique $q^r$-th root of $\alpha^{q^r}$ must already lie in $\Fin$. That is, $\alpha \in \Fin$. Since $\nu(\alpha)=0$ we conclude that $\alpha \in \Oin^*$. This shows $\Fsep^*/\Oin^*$ is torsion-free and finishes the proof. 
\end{proof}

\begin{lemma}\label{lem.divisible_semi_direct}
Let $A,B$ be $\Gal(F)$ groups for some field $F$ such that $B$ acts on $A$ compatibly with the $\Gal(F)$-action. If $A$ is a uniquely divisible abelian group, then the natural map:
\begin{equation}\label{e.lem.divisibility1}
H^1(\Gal(F),B) \to H^1(\Gal(F),A\rtimes B)
\end{equation}
is a bijection.
\end{lemma}
\begin{proof}
The map \eqref{e.lem.divisibility1} is injective because it is split by the map 
\begin{equation}\label{e.lem.divisibility2}
H^1(\Gal(F),A\rtimes B) \to H^1(\Gal(F),B)
\end{equation} 
induced by the projection $A\rtimes B\to B$. Moreover, \eqref{e.lem.divisibility1} is surjective because its left inverse \eqref{e.lem.divisibility2} is injective. Indeed, the fibers of \eqref{e.lem.divisibility2} over $[b_\sigma]\in H^1(\Gal(F),B)$ is in bijective correspondence with $H^1(\Gal(F),{}_b A)$ by \cite[Page 52, Corollary 2]{serre1997galois} and $H^1(\Gal(F),{}_b A)$ is a singleton because $A$ is uniquely divisible \cite[Corollary 4.2.7]{gille_szamuely_2006}. This shows \eqref{e.lem.divisibility1} is a bijection.
\end{proof}

We will need the following definition from valuation theory.
\begin{definition}
    Let $q = \Char F$ be non-zero. An extension of Henselian valued fields $F\subset L$ with value groups $\Gamma_{F}\subset \Gamma_L$ and residue fields $k\subset l$ is called \emph{purely wild} if the following conditions hold:
\begin{enumerate}
    \item The group $\Gamma_L/\Gamma_F$ is a $q$-group.
    \item The field extension $l/k$ is algebraic and purely inseparable.
\end{enumerate}
\end{definition}
We will use the following proposition in the proof of part (1) of Theorem~\ref{thm.main} to avoid assuming $\Char k$ is good for $G$.
\begin{proposition}\label{prop.loop_prime-to-p} 
    Assume $G^{\circ}$ is reductive. For any $G$-torsor $\gamma\in H^1(F,G)$ there exists a purely wild finite extension $F\subset L$ of degree a power of $\Char k = q$ such that $\gamma_L$ is a loop torsor.
\end{proposition}
\begin{proof}
    Let $F\subset L$ be a maximal algebraic purely wild extension of $F$; see \cite[Section 4]{kuhlmann1986immediate}. The field $L$ satisfies 
    $$\Ftr \cdot L = \Falg,\  \Ftr\cap L = F, \ \ \Ltr = L_{\sep}=L_{\alg}.$$
    See \cite[Theorem 4.3]{kuhlmann1986immediate} and  \cite[Lemma 6.3]{kuhlmann1986immediate} respectively. Moreover, $L/F$ is a pro-$q$ extension. Indeed, for any finite subextension $F\subset L'\subset L$ with residue field extension $l'/k$, both $[\Gamma_{L'}:\Gamma_F]$ and $[l':k]$ are powers of $q$ because $L'/F$ is purely wildly ramified \cite[Section 4]{kuhlmann1986immediate}. Therefore $[L':F]$ is a power of $q$ by Ostrowski's theorem.  
    
    Let $T\subset G$ be a maximal torus defined over $k_0$ and $N = N_G(T)$ its normalizer. Note that $N$ is smooth because $T$ is of multiplicative type and $G$ is smooth \cite[Expose XI, Corollaire 2.4]{SGA3}. By Lemma~\ref{lem.loop_normalizer_torus}, we have $$H^1_{\lop}(L,N)=H^1(L,N).$$ 
    Since $H^1(L,N) \to H^1(L,G)$ is surjective \cite[Corollary 5.3]{chernousov2008reduction}, this gives $$H^1_{\lop}(L,G)=H^1(L,G).$$
    Write $L$ as a union $L= \bigcup L_i$ of a chain $L_1\subset L_2 \subset \dots $ of finite purely wild extensions of $F\subset L_i$ of degree a power of $q$. We need to show $\gamma_{L_i}$ is a loop torsor for some $i$. Since $\gamma_L$ is a loop torsor, it is represented by a cocycle:
    $$c_\sigma : \Galtr(L)\to G(\cO_{L,\oin}).$$
    Since $c_{(\cdot)}$ is continuous it factors through $\Gal(E/L)$ for some finite Galois extension $E/L$. By Lemma~\ref{lem.limit_of_ext}, for all large enough $i$, $E$ is a compositum $E= E_i L$ for some tamely ramified Galois extension $E_i/L_i$ such that the restriction map $$\res_i: \Gal(E/L)\to \Gal(E_i/L_i)$$ is an isomorphism. We have $\cO_{L,\oin} = \bigcup_i \cO_{L_i,\oin}$, which implies  $G(\cO_{L,\oin}) =\bigcup_i G(\cO_{L_i,\oin})$ by \cite[Lemma 10.62]{gortz2010algebraic}. Choose an integer $i$ large enough such that: $$\forall\sigma\in \Gal(E/L), \  c_\sigma \in G(\cO_{L_i,\oin}).$$ 
    Therefore $c_\sigma$ is the restriction of the cocycle $c^{(i)}_\sigma: \Gal(E_i/L_i) \to  G(\cO_{L_i,\oin})$ given by:
    $$c^{(i)}_\sigma = c_{\res_i^{-1}(\sigma)}.$$
    Set $\eta = [c^{(i)}_\sigma]\in H^1_{\lop}(L_i,G)$.  Since $c_\sigma$ is the restriction of $c^{(i)}_\sigma$ to $L$, we have $$ \gamma_L=\eta_L .$$ 
Since $H^1(L,G) = \colim H^1(L_i,G)$ by \cite[Theorem 2.1]{margaux2007passage}, this implies $$\gamma_{L_j} =\eta_{L_j}$$
 is a loop torsor for large enough $i \leq j$. 
\end{proof}

We have seen that in the absence of wild ramification all torsors are loop torsors. This might lead one to guess that all tamely ramified torsors are loop torsors. We finish with an example showing this is not the case.
\begin{example}
    Let $k$ be a field of characteristic $2$ equipped with a surjective homomorphism $$\eps: \Gal(k) \to \GL_1(\bZ) =\{\pm 1\}.$$ The homomorphism $\eps$ defines a $1$-dimensional torus $T_\eps := {}_\eps \Gm$. The $\Gal(k)$-action on an element $\alpha\in T_\eps(\ksep) = \ksep^*$ is given by:
    $${}^\sigma \alpha = \sigma(\alpha)^{\eps(\sigma)}.$$
    Let $F = k(\!(t)\!)$ and consider the decomposition $\Psi_t: \Gal(k)\ltimes \hat{\bZ}' \to \Galtr(F)$ of Proposition~\ref{prop.decomp_of_Gal}. We claim the following $\Galtr(F)$-cocycle represents a tamely ramified $T_\eps$-torsor which is not a loop torsor:
    $$c_{\Psi_t(\sigma,n)} =  \begin{cases}
        t & \text{ if }\eps(\sigma) = -1 \\
        1 & \text{ if }\eps(\sigma) = 1
    \end{cases}$$
    To see this, we assume that $c_\tau$ is cohomologous to a cocycle taking values in $T_\eps(\Oin)$ and reach a contradiction. Let $s\in T_\eps(\Ftr) = \Ftr^*$ be such that for all $\tau\in \Galtr(F)$: 
    \begin{equation}\label{e.ex_not_loop1}
        s^{-1} c_\tau {}^\tau s \in T_{\eps}(\Oin) = \Oin^{*}.
    \end{equation}
    We choose $\sigma\in \Gal(k)$ such that $\eps(\sigma) =-1$ and compute:
    $$s^{-1}c_{\Psi_t(\sigma,0)} {}^{\Psi_t(\sigma,0)} s = t s^{-1} {}^{\sigma} s = t s^{-1} \sigma(s)^{\eps(\sigma)}=t (s \sigma(s))^{-1}.$$
    Let $\nu$ be the Henselian valuation on $\Ftr$ extending the $t$-adic valuation on $F$. Equation \eqref{e.ex_not_loop1} gives $\nu(t(s \sigma(s))^{-1}) = 0$ and therefore:
    $$2\nu(s) = \nu(s) + \nu(\sigma(s)) = \nu(t).$$
    This contradicts the fact that $\Ftr/F$ is a tamely ramified extension and so $\Gamma_{\Ftr}/\Gamma_F$ contains no $2$-torsion by definition \cite[Appendix A.1]{tignol2015value}. Let $l/k$ be the separable quadratic extension defined by $\eps$. We note that a standard computation shows $T_\eps$ is isomorphic to the norm one torus $R^{(1)}_{l/k}(\Gm)$. 
\end{example}

\section{Proof of Theorem~\ref{thm.main}: First reduction}\label{sect.reduction_to_reductive}
In this section we reduce Theorem~\ref{thm.main} to the case where $G^{\circ}$ is reductive.  We will assume $k_0$ is perfect throughout the section, because this is the only case in which $G^{\circ}$ being reductive is an extra assumption.  The following lemma is false over imperfect base fields; see \cite{tan2013essential}.

\begin{lemma}\label{lem.ed_unipotent_normal}
    Let $U\subset G$ be a normal smooth and connected unipotent subgroup. For any $p\neq \Char k_0$, we have:
    $$\ed(G) = \ed(G/U) , \ \ \ed(G;p)=\ed(G/U;p).$$
\end{lemma}
\begin{proof}
    This follows from the fact that the induced map 
\begin{equation}\label{e.coh_of_unipotent}
H^1(k,G)\to H^1(k,G/U)
\end{equation}
    is a bijection for all fields $k_0 \subset k$ \cite[Lemma 1.13]{sansuc1981groupe}. 
\end{proof}

Let $A\subset G$ be a finite diagonalizable subgroup. Denote the unipotent radical of $G$ by $\Ru(G)$ and set $\overline{G}:= G/\Ru(G)$. Note that $\Ru(G)$ is normal in $G$ because it is a characteristic subgroup of $G^{\circ}$. Let $\pi: G\to \overline{G}$ be the natural quotient map. By \cite[Proposition 9.6]{borel_lag} and the remark following it, the restriction of $\pi$ to $C_G(A)$ gives a quotient map:
$$\pi_{|C_G(A)}: C_G(A) \to C$$
onto a smooth subgroup $C \subset C_{\overline{G}}(\pi(A))$ with
\begin{equation}\label{e.reduction6}
    \dim(C) = \dim(C_{\overline{G}}(\pi(A))).
\end{equation} 
Let $\Ru(G)^{A} = \Ru(G)\cap C_G(A)$ be the subgroup fixed by the conjugation action of $A$ on $\Ru(G)$. There is a short exact sequence:
$$1\to \Ru(G)^A \to C_G(A) \overset{\pi}{\to} C\to 1.$$
The unipotent group $\Ru(G)^A$ is smooth and connected by \cite[Theorem 13.7]{milne2017algebraic} and \cite[Proposition 9.4]{borel_lag} respectively. Therefore Lemma~\ref{lem.ed_unipotent_normal} gives:
\begin{equation}\label{e.reduction2}
\ed(C_G(A)) = \ed(C) , \ \ \ed(C_G(A);p)=\ed(C;p).
\end{equation}
Since $ \dim(C) = \dim(C_{\overline{G}}(\pi(A)))$, \cite[Proposition 3.15]{merkurjev-survey} and \eqref{e.reduction2} give:
\begin{equation}\label{e.reduction3}
    \ed(C_{\overline{G}}(\pi(A))) \geq \ed(C) = \ed(C_G(A))
\end{equation}
Similarly, we obtain for any prime $p\neq \Char k_0$:
\begin{equation}\label{e.reduction4}
    \ed(C_{\overline{G}}(\pi(A));p) \geq \ed(C_G(A);p)
\end{equation}
To finish the reduction to the reductive case we will need the following lemma.
\begin{lemma}\label{lem.reduction_to_reductive_aniso}
    If $C_G(A)$ admits anisotropic torsors over a field $k_0\subset k$, then so does $C_{\overline{G}}(\pi(A))$.
\end{lemma}
\begin{proof}
    Assume $\gamma\in H^1(k,C_G(A))$ is anisotropic. The push-forward $\pi_*(\gamma)\in H^1(k,C)$ is anisotropic by Lemma~\ref{lem.aniso_image}(2). The push-forward of $\pi_*(\gamma)$ to $H^1(k,C_{\overline{G}}(\pi(A)))$ is anisotropic by Lemma~\ref{lem.aniso_image}(1) because $\dim(C) = \dim(C_{\overline{G}}(\pi(A)))$.
\end{proof}

\begin{proposition}\label{prop.reduction_to_reductive}
    To establish Theorem~\ref{thm.main}, it suffices to prove it under the additional assumption that $G^{\circ}$ is reductive.
\end{proposition}
\begin{proof}
    We assume that Theorem~\ref{thm.main} holds for $G/\Ru(G)=\overline{G}$ and show that it holds for $G$ (note that $\overline{G}^{\circ} = G^{\circ}/\Ru(G)$ is reductive because $k_0$ is perfect \cite[Proposition 19.11]{milne2017algebraic}).

    Let $A\subset G$ be a finite diagonalizable subgroup such that $C_G(A)$ admits anisotropic torsors over a field $k_0\subset k$. Then $C_{\overline{G}}(\pi(A))$ admits anisotropic torsors over $k$ as well by Lemma~\ref{lem.reduction_to_reductive_aniso}. We prove the two parts of Theorem~\ref{thm.main} separately.

    \smallskip
    (1) Let $p\neq \Char k_0$ be a prime. If $k$ is $p$-closed, and $A$ is a $p$-group, then Theorem~\ref{thm.main}(1) gives:
    $$ \ed({\overline{G}};p)\geq  \ed(C_{\overline{G}}(\pi(A));p).$$
    Combining this with Lemma~\ref{lem.ed_unipotent_normal} and \eqref{e.reduction4}, we get:
    $$\ed(G;p) = \ed({\overline{G}};p)\geq  \ed(C_{\overline{G}}(\pi(A));p) \geq \ed(C_G(A);p).$$
    Therefore the conclusion of  Theorem~\ref{thm.main}(1) holds for $G$.

\smallskip
    (2) If the characteristic of $k_0$ is good for $G$, then it is good for $\overline{G}$ as well by Definition~\ref{def.char}. Therefore Theorem~\ref{thm.main}(2) gives:
    $$ \ed({\overline{G}})\geq  \ed(C_{\overline{G}}(\pi(A))).$$
    By Lemma~\ref{lem.ed_unipotent_normal} and \eqref{e.reduction3}, we have:
    $$\ed(G) = \ed({\overline{G}})\geq  \ed(C_{\overline{G}}(\pi(A))) \geq \ed(C_G(A)).$$
    This shows  Theorem~\ref{thm.main}(2) holds for $G$.
\end{proof}

\section{Conclusion of the proof of Theorem~\ref{thm.main}}\label{sect.proof_of_main_thm}

In this section, $F$ will be an iterated Laurent series field $k(\!(t_1)\!)\dots(\!(t_r)\!)$ over $k$ for some field $ k_0\subset k$. The valuation $\nu$ on $F$ is the $(t_1,\dots,t_r)$-adic valuation. We can and shall assume that $G^{\circ}$ is reductive by Proposition~\ref{prop.reduction_to_reductive}. We start the proof of Theorem~\ref{thm.main} with an elementary lemma.

\begin{lemma}\label{lem.conjugation}
    Let $x = (x_1,\dots,x_d), y = (y_1,\dots,y_d)\in G(\ksep)^d$ be $d$-tuples of elements. Assume that $x_1,\dots,x_d$ generate a finite abelian group $A\subset G(\ksep)$ such that $|A|$ is coprime to $\Char k$. Let $G$ act on $G^d$ by conjugation. If $x,y$ are conjugate in $G(l_{\sep})^d$  for some field extension $k\subset l$ , then $x,y$ are conjugate in $G(k_{\sep})^d$.
\end{lemma}
\begin{proof}
     Let $X\subset G^d$ denote the $G$-orbit of $x$ under the conjugation action on $G^d$.  By assumption $y\in X(l_{\sep})$. Since $y\in G^d(\ksep)$ and $X\subset G^d$ is locally closed, we get $$y\in X(l_{\sep})\cap G^d(\ksep) = X(\ksep).$$ The stabilizer $\stab_{G}(x)$ is $C_G(A)$ and so the orbit map induces an isomorphism $X \cong G/C_G(A)$ \cite[Corollary 7.13]{milne2017algebraic} over $\ksep$. The homomorphism $$G(\ksep)\to G/C_G(A)(\ksep), \ g\mapsto gxg^{-1}$$ is surjective because $C_G(A)$ is smooth \cite[Theorem 13.9]{milne2017algebraic}. Therefore $y\in X(\ksep)$ implies $y= gxg^{-1}$ for some $g\in G(\ksep)$.
\end{proof}
The following descent lemma is the main ingredient of the proof of Theorem~\ref{thm.main}.

\begin{lemma}\label{lem.descent}
    Let $\vphi: \I{F}\to A(\ksep)$ be a continuous surjection onto a diagonalizable $k$-subgroup $A\subset G_k$ and $[a_\sigma]\in H^1_{\an}(k,C_G(A)(\ksep))$ an anisotropic torsor. Let $F_1 \subset F$ be a Henselian subfield with $\Gamma_{F_1} = \Gamma_F$ and choose a uniformizer $\pi$ for $F_{1,\tr}$.  If $[a_\sigma,\vphi]_\pi\in H^1(F,G)$ descends to a loop torsor $\gamma_1 \in H^1_{\lop}(F_1,G)$, then $[a_\sigma]$ descends to the residue field of $F_1$.
\end{lemma}
\begin{proof}
   
    Let $k_1$ denote the residue field of $F_1$ and $\cO_{F_1,\oin}$ denote the inertial closure of the valuation ring of $F_1$. By assumption $\gamma_1$ is a loop torsor and so we can write $\gamma_1=[a'_\sigma,\vphi']_{\tau}$ for some $\vphi': \I{F_1}\to G(\cO_{F_1,\oin}), \ a'_\sigma\in G(\cO_{F_1,\oin})$. Since $\Gamma_{F_1} = \Gamma_F$, we have by Corollary~\ref{cor.field_ext_funct}:
    \begin{equation*}
        [\Inf_{k/k_1}(a')_\sigma, \vphi']_\pi = (\gamma_1)_F = \gamma = [a_\sigma,\vphi]_\pi.
    \end{equation*}
   Since $[a_\sigma]$ is anisotropic, so is $\gamma$ by \cite[Proposition 4.9]{gille2024newloop}. Therefore Theorem~\ref{thm.uniqueness} implies that there exists $s\in G(\ksep)$ such that: 
    \begin{equation}\label{e.22}
       s^{-1} a_\sigma {}^{\sigma} s = \overline{\Inf_{k/k_1}(a')_\sigma}, \ \ s^{-1} {\vphi(x)} s = \overline{\vphi'(x)}.
    \end{equation}
    By Lemma~\ref{lem.conjugation}, there exists $s_1 \in G(k_{1,\sep})$ such that
    $$s_1^{-1} {\vphi(x)} s_1 = \overline{\vphi'(x)}$$
    Lift $s_1$ to an element $h \in G(\cO_{F_1,\oin})$ using Hensel's lemma to get $$\overline{h}^{-1}\vphi(x)\overline{h} = \overline{\vphi'(x)}.$$ We replace $\langle a'_\sigma,\vphi'\rangle_\pi$ by the cohomologous cocycle $\langle h^{-1}a'_\sigma {}^\sigma h, h^{-1}\vphi'h\rangle_\pi$ using Lemma~\ref{lem.cohomolog_criterion} to assume $$\overline{\vphi'(x)} = {\vphi(x)}$$ without loss of generality. 
    Then \eqref{e.22} shows $s \in C_G(A)(\ksep)$ because $A(\ksep) = \im \vphi$. 
    We have $\overline{a'_\sigma}\in C_G(A)(k_{0,\sep})$ for all $\sigma\in \Gal(k_1)$ by Lemma~\ref{lem.centralize_over_residue}. Since $s \in C_G(A)(\ksep)$, we get
    \begin{equation}\label{e.13}
        [\Inf_{k/k_1}\overline{a'_\sigma}] =[a_\sigma] \in H^1(k,C_G(A)).
    \end{equation}
    Therefore \eqref{e.13} proves $[a_\sigma] \in H^1(k,C_G(A))$ descends to $k_1$ as we wanted to show.
\end{proof}

Without the assumption $\Gamma_{F_1}= \Gamma_F$ in Lemma~\ref{lem.descent}, we cannot say with certainty that $[a_\sigma]$ descends to the residue field of $F_1$. However, we can still get a lower bound on the transcendence degree of $F_1$ in terms of the essential dimension of $[a_\sigma]$.

\begin{corollary}\label{cor.lower_bound_corollary}
    Let $\vphi: \I{F}\to A(\ksep)$ be a continuous surjection onto a diagonalizable $k$-subgroup $A\subset G_k$ and $\eta=[a_\sigma]\in H^1_{\an}(k,C_G(A)(\ksep))$ be an anisotropic torsor. If $[a_\sigma,\vphi]_\pi\in H^1(F,G)$ descends to a loop torsor $\gamma_1 \in H^1_{\lop}(F_1,G)$ for some Henselian subfield $F_1 \subset F$ (with respect to $\nu_{|F_1}$), then $\trdeg_{k_0}(F_1) \geq \ed(\eta)$.
\end{corollary}
\begin{proof}
    Set $F_2 = F_1(\pi^{\gamma};\gamma\in \Gamma_{F})$ and let $k_2$ denote the residue field of $F_2$. Since $\Gamma_{F_2}=\Gamma_F$, Lemma~\ref{lem.descent} implies $\eta$ descends to $k_2$ and so $\trdeg_{k_0}(k_2)\geq \ed(\eta)$. Therefore:
    $$\trdeg_{k_0}(F_1) \geq \trdeg_{k_0}(F_2) - r \geq \trdeg_{k_0}(k_2) \geq \ed(\eta).$$
    Here the first inequality follows because  $F_2$ is generated by $r$ elements over $F_1$. The second inequality from the fact that $k_2$ is the residue field of a valuation of rank  $r$ on $F_2$; see \cite[Chapter VI, Theorem 3, Corollary 1]{zariski1960commutative}.
\end{proof}

We proceed to prove Theorem~\ref{thm.main}. Assume $C_G(A)$ admits an anisotropic torsor $[a_\sigma]=\eta \in H^1_{\an}(k,C_G(A))$ over $k$. We may assume that $\eta$ is versal by Proposition~\ref{prop.versal_aniso} and so:
    \begin{equation}\label{e.23}
    \ed(\eta) = \ed(C_G(A)) \ \text{ and } \ \ed(\eta;p) = \ed(C_G(A); p).
    \end{equation}
Set $r=\rank (A)$, let $F= k(\!(t_1)\!)\dots(\!(t_r)\!)$ be a field of iterated Laurent series equipped with the $(t_1,\dots,t_r)$-adic valuation. Since $\Gamma_F = \bZ^r$, we have $\I{F} = \Hom(\Gamma_F,\hat{\bZ}')= \hat{\bZ}'^r$. Choose a surjection $\vphi: \I{F}\to A(\ksep)$ and set $$\gamma = [a_\sigma,\vphi]_\pi \in H^1_{\lop}(F,G).$$ 
Note that $[a_\sigma,\vphi]_\pi$ is well-defined by Example~\ref{ex.main} and anisotropic by \cite[Proposition 4]{gille2024newloop}. We prove the two parts of Theorem~\ref{thm.main} separately. We start from Part (2) because it is simpler.

\smallskip
\begin{proof}[Proof of Theorem~\ref{thm.main}(2)] Assume $\gamma$ descends to a field $k\subset F_1 \subset F$ with:  
    $$\trdeg_{k_0}(F_1) = \ed(\gamma).$$ 
    We may replace $F_1$ by its Henselization inside of $F$ to assume $F_1$ is Henselian; see \cite[Appendix A.3]{tignol2015value}. Indeed, the Henselization $F_1 \subset F_1^h \subset F$ is algebraic over $F_1$ and therefore $\trdeg_{k_0}(F_1) = \trdeg_{k_0}(F_1^h)$.
    Under the characteristic assumptions in Part (2), $\gamma$ descends to a loop torsor over $F_1$ by Proposition~\ref{lem.loop_char_good}. Therefore Corollary~\ref{cor.lower_bound_corollary} and \eqref{e.23} give  
    \begin{equation}\label{e.equation_that_fails_in_example}
        \ed(\gamma)=\trdeg_{k_0}(F_1)\geq \ed(\eta) = \ed(C_G(A)).
    \end{equation}
    This finishes the proof because $\ed(G)\geq \ed(\gamma)$ by definition.
\end{proof}
\begin{proof}[Proof of Theorem~\ref{thm.main}(1)] Assume $A$ is a $p$-group and $k$ is $p$-closed. Let $L/F$ be a prime-to-$p$ extension and $k_0\subset L_1\subset L$ a field such that  $\gamma_L$ descends to $\gamma_1 \in H^1(L_1,G)$ and:
\begin{equation}\label{e.24}
    \trdeg_{k_0}(L_1) = \ed(\gamma;p).
\end{equation}
Note that $L$ is Henselian \cite[Proposition A.30]{tignol2015value}. As in the proof of Part (1), we replace $L_1$ by its Henselization inside of $L$ to assume it is Henselian without changing $\trdeg_{k_0}(L_1)$.
By Proposition~\ref{prop.loop_prime-to-p}, there exists a prime-to-$p$ extension $L'_1/L_1$ such that $(\gamma_1)_{L'_1}$ is a loop torsor. Note that $L'_1$ is contained in a prime-to-$p$ extension $L\subset L'$ \cite[Lemma 6.1]{merkurjev2009essential}. The residue field of $L'$ is $k$ because $k$ is $p$-closed by Ostrowski's theorem.
By Corollary~\ref{cor.choosing_uniformizers}, there exists a uniformizer $\tau$ for $L'_{\tr}$ such that:
$$\gamma_{L'} = [a_\sigma, \vphi_{|L'}]_\tau$$
Note that $\gamma_{L'}$ is anisotropic and $\im\vphi_{|L'} = A(\ksep)$ by Lemma~\ref{lem.aniso_prime-to-p}. The field $L'$ is isomorphic (as a valued field) to an iterated Laurent series field because it is a finite extension of $F$; see the first paragraph of \cite[Page 199]{gille2009lower}.
Since $\gamma_{L'}$ descends to a loop torsor over $L'_1$, Corollary~\ref{cor.lower_bound_corollary} implies $$\trdeg_{k_0}(L'_1)\geq \ed(\eta) \geq \ed(\eta;p).$$
Since $L_1 \subset L'_1$ is an algebraic extension we get from \eqref{e.24}: 
$$\ed(\gamma;p) = \trdeg_{k_0}(L_1)= \trdeg_{k_0}(L'_1)\geq \ed(\eta;p) = \ed(C_G(A);p)$$
This finishes the proof because $\ed(G;p) \geq \ed(\gamma;p)$ by definition.
\end{proof}

We end this section with an example showing why the assumption that $C_G(A)$ admits anisotropic torsors is required for our approach to work.
\begin{example}
    Let $G = \PGL(V)$ be the projective linear group associated to the vector space $V = \bC^4 \otimes \bC^2$. Let $A\subset G$ be the subgroup generated by the equivalence class of the matrix $$a = \id_{\bC^4} \otimes d$$ where $d \in \GL_2(\bC)$ is the diagonal matrix:
    $$d = \begin{pmatrix} 1 & 0 \\ 0 & \zeta_3 \end{pmatrix}.$$
    Here $\zeta_3$ is a primitive $3$-rd root of unity. We will prove that for any torsor $[a_\sigma]\in H^1(k,C_G(A))$ and any surjection $\vphi: \I{F} \to A(\ksep)$, \eqref{e.equation_that_fails_in_example} fails for the loop torsor $\gamma= [a_\sigma,\vphi]_\pi$. That is, we will prove:
    $$\ed(\gamma) \lneq \ed(C_G(A)).$$
    The failure of \eqref{e.equation_that_fails_in_example} for all $C_G(A)$-torsors is explained by the fact that, as we shall see, $C_G(A)$ does not admits anisotropic torsors.
    
    A choice of a basis $b_1,b_2 \in \bC^2$ allows us to view matrices $g \in \GL(V)$ as block matrices:
    $$g = \begin{pmatrix} g_{11} & g_{12} \\ g_{21} & g_{22} \end{pmatrix},$$
    where $g_{ij} \in \Mat_4(\bC)$ are four by four matrices. The action of $g$ on $u \otimes b_j$ is given by
    $$g u\otimes b_j = g_{1j} u \otimes b_1 + g_{2j} u \otimes b_2. $$
    The corresponding block decomposition of $a = \id_{\bC^4} \otimes d$ is
    $$a = \begin{pmatrix} I_4 & 0 \\ 0 & \zeta_3 I_4 \end{pmatrix}.$$
    Here $I_4 \in \Mat_4(\bC)$ is the identity matrix. Using this, one see computes the centralizer of $A = \langle a\rangle$:
    $$C_{G}(A) = \Big\{  \bigg [ \begin{pmatrix} g_{11} & 0 \\ 0 & g_{22} \end{pmatrix} \bigg ]  \mid g_{11},g_{22}\in \GL_4(\bC) \Big \}.$$
    We identify the right hand side with $\GL_4\times \GL_4 / \lambda(\Gm)$, where $\lambda:\Gm \to \GL_4\times \GL_4$ is the embedding  $\lambda(s) = (s I_4 , s I_4)$. Consider the diagonal embedding:
    $$\Delta: \PGL_4 \to C_{G}(A), \ \ [g] \mapsto [g\otimes I_2]  = \bigg [ \begin{pmatrix} g & 0 \\ 0 &g \end{pmatrix} \bigg ].$$
    Let $\bC \subset F$ be a field. By \cite[Theorem A.1]{cernele-reichstein}, the quotient map $\pi: C_{G}(A) \to \PGL_4\times \PGL_4$ induces an embedding
    $$\pi_*: H^1(F, C_{G}(A)) \hookrightarrow H^1(F,\PGL_4)^{\times 2}$$
    onto the diagonal subset consisting of all pairs $(x,x)$ with $x\in H^1(F,\PGL_4)$. This subset is precisely the image of $\pi_* \circ \Delta_*$, where $\Delta_*: H^1(F,\PGL_4) \to H^1(F,C_{G}(A))$ is the map induced by $\Delta$. We conclude that $\Delta_*$ must be an isomorphism. In particular, we get:
    $$\ed(C_{G}(A)) = \ed(\PGL_4) = 5$$
    by \cite{merkurjev-PGLp^2}. Let $\gamma \in H^1(F,G)$ be induced from a $C_G(A)$-torsor. Since $\Delta_*$ is an isomorphism, $\gamma$ is represented by a cocycle of the form $\Delta(c_\sigma)$ for some cocycle $c_\sigma \in\PGL_4$. Recall that $\PGL_n$-torsors classify central simple algebras of degree $n$ up to isomorphism; see for example \cite[Chapter 2]{gille_szamuely_2006}. If the cocycle $c_\sigma$ corresponds to some central simple algebra $B$ of degree $4$ over $F$, then the cocycle $\Delta(c_\sigma)$ corresponds to the matrix algebra $M_2(B)=B \otimes M_2(\bC)$ over $B$ because $\Delta(c_\sigma)= c_\sigma \otimes I_2$. Both the algebra $M_2(B)$ and $\gamma$ descend to a field $F_0\subset F$ with $\trdeg_{\bC}(F_0)\leq 4$ by \cite[Corollary 1.4]{lorenz2003}. This gives the upper bound
    $$\ed(\gamma) \leq 4.$$
    In particular, for any loop torsor $\gamma= [a_\sigma,\vphi]_\pi$ with $\im\vphi =A(\ksep)$ over any Henselian field $\bC \subset F$ the essential dimension of $\gamma$ is strictly smaller than $\ed(C_G(A))$:
    $$\ed(\gamma) \leq 4 < 5 = \ed(C_G(A)).$$
    This shows the assumption that $\gamma$ is anisotropic is necessary for our proof of Theorem~\ref{thm.main} to work. Note that in this example any $C_G(A)$-torsor $\eta = [a_\sigma]$ is isotropic because $Z(C_G(A))\cong \Gm$ embeds into the twisted group ${}_{a_\sigma} C_G(A)$. 
\end{example}

\section{Reductive subgroups of maximal rank}\label{sect.pf_of_cor_ofmainthm}

Some of the inequalities in Theorem~\ref{thm.concrete_bounds} follow from the next useful corollary of Theorem~\ref{thm.main} whose proof relies on Borel–de Siebenthal theory. Recall that an algebraic subgroup $H\subset G$ is called a subgroup of maximal rank if a maximal torus in $H$ is remains maximal as a torus in $G$.

\begin{corollary}\label{cor.ofmainthm}
    Let $G$ be a group over a field $k_0$ with $G^{\circ}$ reductive and assume $\Char k_0 \neq 2,3$. Let $H\subset G$ be a (connected) reductive subgroup of maximal rank. Assume that $Z(H)$ is finite and diagonalizable.
     \begin{enumerate}
     
        \item Let $p\neq \Char k_0$ be a prime. If $Z(H)$ is a $p$-group and $H$ admits an anisotropic torsor  over some $p$-closed field $k_0\subset k$, then we have:
        $$\ed(G;p) \geq \ed(H;p).$$
        
        \item Assume $\Char k_0$ is good for $G$ (see Definition~\ref{def.char}). If $H$ admits an anisotropic torsor  over some field $k_0\subset k$, then we have:
        $$\ed(G) \geq \ed(H).$$

    \end{enumerate}
\end{corollary}
\begin{proof}
   By Borel-de Siebenthal's theorem, we have $$C_G(Z(H))^{\circ} = C_{G^{\circ}}(Z(H))^{\circ} =  H.$$ See \cite{borel1949sous} for the characteristic zero case and \cite{pepin2012subgroups} for the general case. Therefore $$\dim(C_G(Z(H))) = \dim H.$$ By \cite[Lemma 2.2]{brv-annals}, we get:
   $$\ed(C_G(Z(H))) \geq \ed(H), \ \ \ed(C_G(Z(H));p) \geq \ed(H;p).$$ 
   Now the corollary follows from Theorem~\ref{thm.main}.
\end{proof}

Corollary~\ref{cor.ofmainthm} allows us to prove the following parts of Theorem~\ref{thm.concrete_bounds}.
\begin{proposition}\label{prop.partofconcreteboundsthm1}
    Assume $\Char k_0 \neq 2,3$. 
    \begin{enumerate}
        \item $\ed(E_8;2) = \ed(\HSpin_{16};2)$
        \item $\ed(E_8;3) \geq 13$
        \item $\ed(E_7^{\ad};2) \geq 19$
    \end{enumerate}
\end{proposition}
\begin{proof}
     (1) There exists an embedding $\HSpin_{16}\subset E_8$ \cite[Example 4.3]{garibaldi2016E8} and an anisotropic $\HSpin_{16}$-torsor over a $2$-closed field $k_0\subset k$ by Lemma~\ref{lem.HSpin_admits_aniso}. Therefore by Corollary~\ref{cor.ofmainthm} $$\ed(E_8;2)\geq \ed(\HSpin_{16};2).$$ The inequality
    $$\ed(E_8;2)\leq \ed(\HSpin_{16};2)$$
    follows from the fact that any $E_8$-torsor admits reduction of structure to $\HSpin_{16}$ after an odd degree extension; see the first paragraph of the proof of \cite[Theorem 16.2]{garibaldi2016E8}.

    \smallskip
    (2) By \cite[Example 4.4]{garibaldi2016E8}, $\SL_9/\mu_3$ is a maximal rank subgroup of $E_8$ and we have $ \ed(\SL_9/\mu_3;3) = 13$ by results of S.\kern-.1em\ Baek-Merkurjev \cite{baek2012essential} and Chernosouv-Merkurjev \cite[Theorem 1.1]{chernousov2013essential}. Therefore Corollary~\ref{cor.ofmainthm} will give:
    $$\ed(E_8;3) \geq \ed(\SL_9/\mu_3;3) = 13,$$
    once we show $\SL_9/\mu_3$ admits anisotropic torsors over a $3$-closed field. This is simple to see using the theory of central simple algebras. The reader is referred to \cite{gille_szamuely_2006} for the relevant definitions.
    There exists a central division algebra $D$ over a field $k_0\subset k$ of period $3$ and index $9$ \cite[Lemma 4.8]{ofek2022reduction}. One can take $k$ to be $3$-closed because the index and period of $D$ are unaffected by passing to prime-to-$3$ extensions of $k$ \cite[Theorem 3.15]{saltman1999lectures}. The algebra $D$ corresponds to an anisotropic $\PGL_9$-torsor $\gamma \in H^1(k,\PGL_9)$ by \cite[Lemma 5.4]{ofek2022reduction}. Let $\pi_*: H^1(k,\SL_9/\mu_3) \to H^1(k,\PGL_9)$ be the pushforward map. There exists $\eta\in H^1(k,\SL_9/\mu_3)$ such that $\pi_*(\eta) = \gamma$  by \cite[Lemma 5.3]{ofek2022reduction} and $\eta$ is anisotropic by \cite[Corollary 3.4]{ofek2022reduction}. 

     \smallskip
    (3) There exists an embedding  $\SL_8/\mu_4 \subset E_7^{\ad}$ and $\rank(\SL_8/\mu_4) = \rank(E_7^{\ad})$; see \cite[Corollary 4.2]{cohen1987finite} or \cite[Table 4]{duan2013isomorphism}.  We have $\ed(\SL_8/\mu_4;2) = 19$ by \cite[Corollary 1.2]{chernousov2013essential}. To deduce the inequality $$\ed(E_7^{\ad};2)\geq  \ed(\SL_8/\mu_4;2) = 19$$ from Corollary~\ref{cor.lower_bound_corollary}, one needs to prove  the existence of an anisotropic $\SL_8/\mu_4$-torsor over some $2$-closed field. As in Part (2), this is done by first constructing a central division algebra $D$ of index $8$ and period $4$ over some $2$-closed field $k_0\subset k$, and then lifting the anisotropic $\PGL_8$-torsor corresponding to $D$ to an anisotropic $\SL_8/\mu_4$-torsor using  \cite[Lemma 5.3]{ofek2022reduction}. 
\end{proof}

\begin{remark}\label{rem.notions_of_isotropy}
    In \cite{ofek2022reduction}, a torsor $\gamma\in H^1(k,G)$ is defined to be isotropic if there exists a proper parabolic subgroup $P \subset G$ such that $\gamma$ lies in the image of the induced map $H^1(k,P)\to H^1(k,G)$. In the proof of Proposition~\ref{prop.partofconcreteboundsthm1} above, we used the fact that this alternative definition of isotropy is equivalent to our notion of isotropy when $G$ is a quasi-split semisimple group. To prove this fact, assume that $G$ is quasi-split and semisimple. Then $\gamma$ is isotropic in the sense of \cite{ofek2022reduction} if and only if ${}_\gamma G$ admits proper parabolic subgroups \cite[Lemma 2.2]{ofek2022reduction}. Since $Z({}_\gamma G)$ is finite, ${}_\gamma G$ admits proper parabolic subgroups if and only if it contains non-trivial split tori \cite[Corollary 4.17]{borel1965groupes}. Therefore $\gamma$ is isotropic in the sense of \cite{ofek2022reduction} if and only if it is isotropic according to our definition.
\end{remark}

\section{Abelian subgroups arising from gradings on the character lattice}\label{sect.gradings_on_char}

In this section we assume $G^{\circ}$ is a split reductive group. Let $T\subset G$ be a split maximal torus defined over $k_0$. We will give a systematic root-theoretic approach to choosing diagonalizable $p$-subgroups $A\subset T$ such that:
\begin{enumerate}
    \item The connected centralizer $C_G(A)^{\circ}$ is $T$.
    \item The group $A$ satisfies the conditions of Theorem~\ref{thm.main}(1).
\end{enumerate} 
These conditions allow us to bound $\ed(C_G(A);p)$ from below using \cite{lotscher2013essential2} and to upgrade our lower bound on $\ed(C_G(A);p)$ to a lower bound on $\ed(G;p)$ using Theorem~\ref{thm.main}. We will obtain a combinatorial formula for lower bounds on $\ed(G;p)$ in terms of gradings on the character lattice $X(T)$.

\begin{definition}
   Fix a prime $p\neq \Char k_0$ and an abstract abelian $p$-group $\cV$. A \emph{$\cV$-grading on $X(T)$} is a surjective homomorphism $$\eps: X(T)\to \cV.$$
\end{definition}
Let $X(\cV) = \Hom(\cV,\Gm)$ be the group of characters of $\cV$. The group $X(\cV)$ is diagonalizable of order $|\cV|$.  Any $\cV$-grading induces an embedding $\eps^* : X(\cV) \to T$ by the anti-equivalence between finitely generated abelian groups and groups of multiplicative type and finite type over $k$; see e.g. \cite[Corollary 8.3]{borel1965groupes}. We will denote the image of this embedding by \begin{equation}\label{e.def_Aeps}
A_\eps:= \eps^*(X(\cV)).
\end{equation}
The following lemma gives us a way to verify the condition $C_G(A_\eps)^{\circ}=T$ holds. 

\begin{lemma}\label{lem.cent_is_T}
    Let $\Phi \subset X(T)$ be the root system associated to $T\subset G$.
    We have $C_G(A_\eps)^{\circ}=T$ if and only if 
    \begin{equation}\label{e.31}
    \forall \alpha\in \Phi, \ \eps(\alpha)\neq 0.
    \end{equation}
\end{lemma}
\begin{proof}
   The inclusion $T\subset C_G(A_\eps)^{\circ}$ holds because $A_\eps\subset T$. Therefore the equality $C_G(A_\eps)^{\circ}=T$ is equivalent to:
   \begin{equation}\label{e.cent_is_T-goal}
       \dim T = \dim C_G(A_\eps).
   \end{equation}
   To check this equality of dimensions we can base change to $k_{0,\alg}$ to assume that $k_0$ is algebraically closed. Moreover, we may assume $G$ is connected by replacing $G$ with $G^{\circ}$ because this does not affect \eqref{e.cent_is_T-goal}.
   The proof is is a variation of \cite[Theorem 2.2]{humphreys1995conjugacy}. We include the details for completeness. Let $H\subset G$ be the subgroup generated by $N_G(T)\cap C_G(A_\eps)$ and all root subgroups $U_\alpha$ with $\eps(\alpha) =0$. It will suffice to prove $$H= C_G(A_\eps).$$
    To prove the inclusion $H\subset C_G(A_\eps)$ it suffices to prove $U_\alpha\subset C_G(A_\eps)$ for any $\alpha\in \Phi$ such that $$\eps(\alpha)=0.$$
    The group $A_\eps$ commutes with $U_\alpha$ if and only if
    $$\alpha(\eps^*(X(\cV))) =1.$$
    Now note that $\alpha \circ \eps^*$ is dual to the homomorphism $\eps \circ \alpha^* : X(\Gm) \to \cV$. Therefore: 
    \begin{equation}\label{e.41}
        \alpha(\eps^*(A_\eps)) =1\iff \eps(\alpha)= 0 \iff U_\alpha \subset C_G(A_\eps).
    \end{equation}

    \smallskip
    To prove $H\supset C_G(A_\eps)$, choose representatives $n_w\in N_{G}(T)$ for Weyl group elements $w\in W  = N_{G}(T)/T$. Choose an ordering of the roots $\Phi$ and let $U_+ ,U_-\subset G$ be the corresponding unipotent groups generated by the positive and negative root groups respectively. An element $x\in C_G(A_\eps)(k_0)$  has a unique Bruhat decomposition:
    $$x= un_w tv,$$
    where $v\in U_+,t\in T,u\in U_+ \cap n_w U_- n_w^{-1}$. The uniqueness of the decomposition forces $u,n_w,v\in C_G(A_\eps)$ because $A_\eps$ normalizes $U_+,N_{G}(T),U_-$. Therefore we can assume $x= u$ or $x=v$. The two cases are similar, we will focus on the case $x=u\in U_+$. Up to an ordering $\alpha_1,\dots,\alpha_d$ of the positive roots, $u$ can be written uniquely as a product:
    $$u = p_{\alpha_1}(t_1)\dots p_{\alpha_d}(t_d),$$
    where $p_\alpha$ is a parametrization of $U_\alpha$ and $t_i\in k_0$.  For any $v\in \cV$, $u\in C_G(A_\eps)$ implies:
    $$p_{\alpha_1}(t_1)\dots p_{\alpha_d}(t_d)= \eps^*(a) u \eps^*(a)^{-1} = p_{\alpha_1}(\alpha_1(\eps^*(a))t_1)\dots p_{\alpha_d}(\alpha_d(\eps^*(a))t_d).$$
    By uniqueness of the decomposition, we get for all $1\leq i\leq d$:
    $$\alpha_i(\eps^*(a))t_i = t_i.$$
    If $t_i\neq 0$, then $\alpha_i(\eps^*(A_\eps)) =1$ and \eqref{e.41} implies $\eps(\alpha) =0$. Therefore $p_{\alpha_i}(t_i)\in H$ for all $1\leq i\leq d$ and $u\in H$, as we wanted to show. 
\end{proof}

Let $\eps$ be a $\cV$-grading such that $\eps(\alpha)\neq 0$ for any $\alpha\in \Phi$.  By Lemma~\ref{lem.cent_is_T} we have $$C_G(A_\eps)^{\circ}=T.$$ We denote the component group by
$$W(\eps)  := C_G(A_\eps)/C_G(A_\eps)^\circ.$$
Since $T$ is commutative $W(\eps)$ acts naturally on $T= C_G(A_\eps)^{\circ}$ by conjugation. We consider the inclusion into the Weyl group:
\begin{equation}\label{e.emb_into_Weyl}
    W(\eps) \subset W = N_G(T)/T.
\end{equation}
Denote the natural action of $W$ on $\chi\in X(T)$ by $w.\chi$. The subgroup $W(\eps)\subset W$ can be explicitly described as the subgroup of $W$ preserving $\eps$: 
\begin{equation}\label{e.description_of_W(eps)}
    W(\eps) = \{ w\in W \mid  \eps(w.\chi) = \eps(\chi), \ \text{for any }\chi\in X(T)\}.
\end{equation}
Indeed, we have $C_G(A_\eps) \subset N_G(T)$ and therefore by definition:
$$C_G(A_\eps) = \{n\in N_G(T) \mid \Ad(n) \circ \eps^* = \eps^*\}.$$
Here $\Ad(n) \in \Aut(N_G(T))$ denotes conjugation by $n$. 
Since the functor $A \mapsto X(A)$ is an anti-equivalence between the category finitely generated abelian groups and the category of diagonalizable groups  \cite[Corollary 8.3]{borel1965groupes}, $\inn(n) \circ \eps^* = \eps^*$ if and only if $n T \in W$ preserves $\eps$.

To give a lower bound on the essential dimension of $C_G(A_\eps)$, it is useful to introduce the following definition:
\begin{definition}{\cite{macdonald}}
    Let $S$ be a finite group acting on a finitely generated abelian group $\cU$. A subset $\Gamma \subset \cU$ is called \emph{$p$-generating} if the subgroup generated by $\Gamma$ is of finite and coprime to $p$ index in $\cU$.
    Choose a Sylow $p$-subgroup $S_p\subset S$. The \emph{symmetric $p$-rank} of the $S$-action on $\cU$ is the following integer:
    $$\SymRank(S,\cU;p) = \min\{ | \Gamma| \mid  \Gamma\subset \cU \text{ is } S_p\text{-invariant and }p\text{-generating}\}. $$
\end{definition}
The notion of $p$-symmetric rank is related to essential dimension at $p$ by work of M.\kern-.1em\ Macdonald-R.\kern-.1em\ L\"{o}tscher-A.\kern-.1em\ Meyer-Reichstein.
\begin{proposition}{\cite[Theorem 1.10]{macdonald}}\label{prop.inequa_ed_C(eps)}
    Let $N$ be a smooth linear algebraic group over $k_0$ such that $N^{\circ} = T$ is a split torus. We have:
    $$\ed(N;p) \geq \SymRank(N/T,X(T);p) - \dim T.$$
    Here the $N/T$ action on $X(T)$ is induced from the conjugation action of $N/T$ on $T$.
\end{proposition}
Note that while $N$ is assumed to be the normalizer of a split maximal torus in a simple algebraic group in \cite[Theorem 1.10]{macdonald}, this assumption is not used in the proof of the lower bound $\ed(N;p) \geq \SymRank(N/T,X(T);p)$. Setting $C_G(A_\eps) = N$ in Proposition~\ref{prop.inequa_ed_C(eps)}, we obtain:
\begin{corollary}\label{cor.inequa_ed_C(eps)}
    If $\eps$ satisfies \eqref{e.31}, then $$\ed(C_G(A_\eps);p) \geq \SymRank(W(\eps),X(T);p) - \dim T .$$
\end{corollary}

Next we formulate a condition on $\eps$ that ensures the hypotheses of Theorem~\ref{thm.main}(1) hold for $A_\eps$ and $G$. 
\begin{lemma}\label{lem.C(Aeps)_admits_aniso}
    Assume $\eps$ satisfies \eqref{e.31} and choose a Sylow $p$-subgroup $W(\eps)_p\subset W(\eps)$. If we have \begin{equation}\label{e.32}
        X(T)^{W(\eps)_p} = \{0\},
        \end{equation} 
        then  $C_G(A_\eps)$ admits anisotropic torsors over some $p$-closed overfield $k_0\subset k$. 
\end{lemma}
\begin{proof}
    
Let $\pi: C_G(A_\eps) \to W(\eps) = C_G(A_\eps)/C_G(A_\eps)^\circ$ be the natural projection. There exists a finite  $p$-group $P\subset C_G(A_\eps)$ such that $\pi(P) = W(\eps)_p$ by \cite[Lemma 5.3]{lotscher2013essential2}. By assumption we have:
  \begin{equation}\label{e.40}
      X(T)^{P} = \{0\}.
  \end{equation}
    Replace $k_0$ by its algebraic closure $k_{0,\alg}$ to assume $P$ is a constant group without loss of generality.
  There exists a $p$-closed field $k$ containing $k_0$ and a surjection $\phi: \Gal(k) \to P$. For example, if $k_1 = k_{0,\alg}(t)$ and $k_1\subset k$ is a $p$-closure of $k_1$, then $\Gal(k)$ is a Sylow $p$-subgroup of $\Gal(k_1)$ \cite[Proposition
101.16]{ekm}. By \cite[Page 80, Example (b)]{serre1997galois} the cohomological dimension of $k$ is one. Note that $k_{0,\alg}(t)^*/k_{0,\alg}(t)^{*p}=H^1(k_1,\bF_p)$ embeds into $H^1(k,\bF_p)$ via the restriction map $H^1(k_1,\bF_p)\to H^1(k,\bF_p)$ because the degree of any finite subextension $k_1 \subset k_2 \subset k$ is coprime to $p$. Therefore $H^1(k,\bF_p)$ is infinite and $\Gal(k)$ is a free pro-$p$-group of infinite rank; see \cite[Page 30, Proposition 24]{serre1997galois}. In particular, there exists a surjection $\phi: \Gal(k) \to P$. This surjection defines a $C_G(A_\eps)$-torsor $[\phi]\in H^1(k,C_G(A_\eps))$ because $P$ is constant. The connected component of the twisted group ${}_\phi C_G(A_\eps)$ is the torus ${}_\phi T$ whose character $\Gal(k)$-module is $X(T)$ equipped with the $\Gal(k)$-action given by:
  $$\sigma. \chi = \phi(\sigma)(\chi).$$
  To show $[\phi]\in H^1(k,C_G(A_\eps))$ is anisotropic we assume there exists an embedding $f:\Gm\to {}_\phi T$ and reach a contradiction. The embedding $f$ corresponds by duality to a $P$-invariant surjection:
  $$f^*: X(T) \to X(\Gm).$$
  See \cite[Corollary 8.3]{borel1965groupes}.
  Let $\chi\in X(T)$ be such that $f^*(\chi)\neq 0$. By the $P$-invariance of $f^*$    :
  $$ f^*(\sum_{p\in P} p.\chi)= |P|f^*(\chi) \neq 0.$$
  In particular, $\sum_{p\in P} p.\chi \in X(T)^{P}$ is non-zero, contradicting \eqref{e.40}.
\end{proof}

Lemma~\ref{lem.C(Aeps)_admits_aniso} allows us to apply Corollary~\ref{cor.inequa_ed_C(eps)} and Theorem~\ref{thm.main} to get a lower bound on $\ed(G;p)$.
\begin{proposition}\label{prop.sym_rank_lower}
    Let $G,T$ and $\eps: X(T)\to \cV$ be as above. If $\eps$ satisfies \eqref{e.31} and \eqref{e.32}, then
        \begin{equation}\label{e.33}
        \ed(G;p) \geq \SymRank(W(\eps),X(T);p) - \dim T.
        \end{equation}
\end{proposition}
\begin{proof}
Lemma~\ref{lem.C(Aeps)_admits_aniso} and Theorem~\ref{thm.main}(1) imply $\ed(G;p) \geq \ed( C_G(A_\eps);p)$. Therefore \eqref{e.33} follows from Corollary~\ref{cor.inequa_ed_C(eps)}.
\end{proof}

\section{Proof of Theorem 1.5: The overall strategy}\label{sect.setup_computations}

\subsection{Setup for the next sections}
In the next four sections we prove all parts of Theorem~\ref{thm.concrete_bounds} except for Parts (2) and (4) which are included in Proposition~\ref{prop.partofconcreteboundsthm1}. We also reprove Merkurjev's lower bound \eqref{e.lower_PGL_n}. In each section we focus on the essential dimension of a group $G$ at a prime $p$. We choose a grading $\eps$ on the character lattice $X(T)$ of a maximal split torus $T\subset G$. We check that $\eps$ satisfies \eqref{e.31} and $W(\eps)$ satisfies \eqref{e.32}. Applying Proposition~\ref{prop.sym_rank_lower} then gives:
$$\ed(G;p) \geq \SymRank(W(\eps),X(T);p) - \dim T.$$
Finally, we prove a lower bound on $\SymRank(W(\eps),X(T);p)$ using  Lemma~\ref{lem.symrank_fancy}. This last step involves proving lower bounds on the size of $W(\eps)$-orbits. We will use the following notation:
\begin{itemize}
    \item For any set $I$ and ring $R$, $R^I$ is the free $R$-module with standard basis $\{e_i | i\in I \}$. The coordinates of an element $x\in R^I$ in the standard basis are denoted $(x_i)_{i\in I}$.

    \item  The group of permutations of a set $X$ is denoted $\Sym(X)$. If a finite group $S$ acts on $X$, we denote the $S$-orbit of $x\in X$ by $S x$ and its stabilizer by $\Stab_{S}(x)$.
    
    \item If $G$ is an adjoint simple group of type $\Delta$, we identify the character lattice $X(T)$ with the root lattice $Q(\Delta)$. We use the description of the roots $\Phi\subset Q(\Delta)$ and the Weyl group $W$ action given in \cite{conway2013sphere}.

    \item We always assume $\Char k_0\neq p$.
 
\end{itemize}

We also note that for any subgroup $S \subset W(\eps)$ we have:
\begin{equation}\label{e.symrank_monotone}
    \SymRank(W(\eps),X(T);p) \geq \SymRank(S,X(T);p).
\end{equation}
To see this, choose Sylow $p$-subgroup $S_p, W(\eps)_p$ of $S$ and $W(\eps)$ such that $S_p \subset W(\eps)_p$. Then \eqref{e.symrank_monotone} follows from the fact that a $p$-generating $W(\eps)_p$-invariant subset $\Gamma\subset X(T)$ is $S_p$-invariant.

\subsection{A lower bound on the symmetric rank}
Let $p$ be a prime and $S$  a finite $p$-group. Assume $S$ acts on a finitely generated abelian group $\cU$. The next elementary lemma gives a useful lower bound on $\SymRank(S,\cU;p)$.

\begin{lemma}\label{lem.symrank_fancy}
    Let $\eps: \cU \to \bF_p^{d}$ be an $S$-invariant surjective homomorphism and assume we are given a direct sum decomposition $L_1\oplus L_2 = \bF_p^d$. Let $\eps_i : U\to L_i$ be the composition of $\eps$ with the projection $\pi_i: U \to L_i$ for $i=1,2$.
     If $|S u|\geq C_i$ whenever $\eps_i(u)\neq 0$, then:
    $$\SymRank(S,\cU;p) \geq C_1 \dim_{\bF_p} L_1 + C_2 \dim_{\bF_p} L_2.$$
\end{lemma}
\begin{proof}
Denote $d_1 = \dim_{\bF_p} L_1$ and $d_2 = \dim_{\bF_p} L_2$.
Let $\Gamma\subset \cU$ be an $S$-invariant, $p$-generating subset such that 
$$|\Gamma| = \SymRank(S,\cU;p).$$
The image $\eps(\Gamma)$ is $p$-generating in $\bF_p^d$ because $\eps$ is surjective. Because $\bF_p^d$ contains no proper subgroups of index coprime to $p$,  $\eps(\Gamma)$ has to generate $\bF_p^d$. Therefore $\Gamma$ contains elements $\gamma_1,\dots,\gamma_d$ such that $\eps(\gamma_1),\dots,\eps(\gamma_d)$ is a basis for $\bF_p^d$. We denote $b_i = \eps(\gamma_i)$ for all $1\leq i\leq d$.

An elementary linear algebra argument shows that, up to relabeling, we can assume that the projections $\pi_1(b_1),\dots,\pi_1(b_{d_1})$ to $L_1$ are a basis for $L_1$, and the projections $\pi_2(b_{d_1+1}),\dots,\pi_2(b_{d_2})$ are a basis for $L_2$. We have for all $1\leq i\leq d_1$:
$$\eps_1(\gamma_i)= \pi_1\eps(\gamma_i) = \pi_1(b_i)  \neq 0.$$
By the last assumption of the Lemma, this implies $|S \gamma_i| \geq C_1$.
Similarly, we have
$$|S \gamma_i| \geq C_2 \text{ for all }d_1 < i\leq d_2.$$
Since $\eps$ is $S$-invariant, the orbits $S\gamma_1,\dots, S\gamma_d$ are disjoint. We conclude:
\[\SymRank(S,\cU;p) = |\Gamma|  \geq |S\gamma_1|+\dots+|S\gamma_d| \geq C_1 d_1+C_2 d_2. \qedhere \]
\end{proof}

\section{Essential dimension of \texorpdfstring{$\PGL_{p^n}$}{PGL} at \texorpdfstring{$p$}{p}}\label{subsect.PGL_n}
Our first application of Proposition~\ref{prop.sym_rank_lower} is to give a new proof of Merkurjev's lower bound on $\ed(\PGL_{p^n};p)$ \eqref{e.lower_PGL_n}. By Proposition~\ref{prop.sym_rank_lower}, it suffices to prove:
\begin{proposition}\label{prop.rank_PGLn}
Let $p \neq \Char k_0$ be a prime and $n\geq 1$.
Let $T\subset \PGL_{p^n}$ be a maximal split torus with corresponding Weyl group $W$. There exists an $\bF_p^n$-grading $\eps:X(T) \to \bF_p^n$ satisfying \eqref{e.31} and \eqref{e.32} such that
    \begin{equation}\label{e.symrankPGL_n}
    \SymRank(W(\eps),X(T);p) \geq n p^n.
\end{equation}
\end{proposition}

\subsection*{Definition of \texorpdfstring{$\eps$}{epsilon}}
We start by introducing the notation needed to define $\eps$. We set $\cV = \bF_p^n$ and identify the character lattice $X(T)$  of the diagonal torus $T\subset \PGL_{p^n}$ with the following sublattice of $\bZ^\cV$ as in \cite[Section 6]{conway2013sphere}:
    $$\bZ^\cV_0  = \{ x \in \bZ^\cV\mid \sum_{v\in \cV} x_v = 0\}$$
Define $\eps: \bZ^\cV_0 \to \cV$ to be the surjective homomorphism given by:
    $$\eps(x) = \sum_{v\in \cV} x_v v.$$
\subsection*{Verification of Condition \eqref{e.31}} 
     The roots of $\Phi\subset \bZ^\cV_0$ are the vectors $e_u - e_v\in \bZ^\cV_0$ with $u\neq v$.  For any such root we have:
    $$\eps(e_u - e_v) = u-v \neq 0.$$
    Therefore $\eps$ satisfies \eqref{e.31}.

\subsection*{Description of the Weyl group}
    The Weyl group  $W$ is the symmetric group $\Sym(\cV)$ acting on $\bZ^\cV_0$ by permuting the standard basis vectors. For any $u\in \cV$, let $\lambda_u\in \Sym(\cV)$ be the translation by $u$. The subgroup $W(\eps)\subset \Sym(\cV)$ of permutations preserving $\eps$ is the subgroup consisting of all such translations. Indeed, if $u\in \cV$ and $x \in \bZ^\cV_0$ then $\lambda_u\in W(\eps)$ because:
    $$\eps( \lambda_u(x) ) = \sum_{v\in \cV} x_{v} (u+v) = \sum_{v\in \cV} x_v v + (\sum_{v\in \cV} x_v) u = \sum_{v\in \cV} x_v v = \eps(x).$$
    Here the second to last equality follows from the fact that $\sum_{v\in \cV} x_v = 0$.
    Conversely, if $\sigma \in \Sym(\cV)$ preserves $\eps$, then for all $v\in \cV$:
    $$v = \eps(e_0 - e_v) =\eps(e_{\sigma(0)} - e_{\sigma(v)}) = \sigma(0)- \sigma(v).$$
    Therefore $\sigma(v) = v+ \sigma(0)$. That is, $\sigma = \lambda_{\sigma(0)}$. This proves:
    $$W(\eps)= \{ \lambda_u \mid u\in \cV\}.$$
    We will identify $W(\eps)$ with $\cV$ using the isomorphism $\cV\tilde{\to} W(\eps), u\mapsto \lambda_u$.
    
\subsection*{Verification of Condition \eqref{e.32}}
    Note that $\cV$ is its own Sylow $p$-subgroup. If an element $x \in \bZ^\cV_0$ is fixed by $\cV$, then for any $u,v\in \cV$ we have:
    \begin{equation}\label{e.PGL_notimport}
        x_{v} = x_{u+v}.
    \end{equation}
    Since $x\in \bZ^\cV_0$ this implies $$0 = \sum_{v\in \cV} x_v  = |\cV| x_0.$$
    Therefore $x_v =x_0 =0$ for all $v\in \cV$ by \eqref{e.PGL_notimport}. We conclude that $W(\eps)= \cV$ satisfies \eqref{e.32}.

 \subsection*{Proof of the lower bound on \texorpdfstring{$\SymRank(W(\eps),X(T);p)$}{the symmetric rank}}
    To finish the proof of \eqref{e.symrankPGL_n} we will show that for any element $x\in \bZ^\cV_0$ with $\eps(v)\neq 0$ we have
    \begin{equation}\label{e.PGL_notimport1}
        |\cV x|= p^n.
    \end{equation} 
    Once we do so, Lemma~\ref{lem.symrank_fancy} will give:
    $$\SymRank(W(\eps),X(T);p) \geq \dim_{\bF_p}\cV \cdot p^n = n p^n.$$
    By the stabilizer-orbit formula, \eqref{e.PGL_notimport1} is equivalent to
    $$\Stab_{\cV}(x) = \{0\}.$$
    Therefore we need to prove that $\Stab_{\cV}(x)=\{0\}$ for all $x\in \bZ^{\cV}_0$ with $\eps(x)\neq 0$.
    We prove a slightly more general lemma:
    
    \begin{lemma}\label{lem.PGLn_element}
        Let $x\in\bZ^\cV$. Assume either $p$ is odd, or $\sum_{v\in \cV} x_v$ is divisible by four. If $\Stab_{\cV}(x) \neq \{0\}$, then
        $$\sum_{v\in \cV} x_v v = 0.$$
        In particular, if $\eps(x)\neq 0$ for some $x\in \bZ^\cV_0$, then $\Stab_{\cV}(x) = \{0\}$.
    \end{lemma}
    \begin{proof}
        Assume that $x$ is fixed by $0\neq u\in \cV$ and let $\cC \subset \cV$ be a set of coset representatives for the quotient $\cV/\langle u\rangle$. Since $x$ is fixed by $u$, we have for all $v\in \cV$:
$$x_{v} = x_{u+v}.$$
Using this, we compute:
 \begin{align*}
     \sum_{v\in \cV} x_v v  &=\sum_{c\in \cC} \sum_{i= 0, 1,\dots,p-1} x_{c + iu} (c+iu)\\
            &= \sum_{c\in \cC} x_{c} \sum_{i= 0, 1,\dots,p-1}  (c+iu)\\
            &= \sum_{c\in \cC} x_{c} (pc+\binom{p}{2} u) = (\sum_{c\in \cC} x_{c})\binom{p}{2}u.
 \end{align*}  
Here in the last equality we used the fact that $pc = 0$ in $\cV$. If $p$ is odd, then it divides $\binom{p}{2}$ and we get
$$\sum_{v\in \cV} x_v v =(\sum_{c\in \cC} x_{c})\binom{p}{2}u = 0. $$
Otherwise $p=2$ and $4\mid \sum_{v\in \cV} x_v$. We have:
$$\sum_{v\in \cV} x_v = \sum_{c\in \cC}  x_{c} + x_{c+u} = 2\sum_{c\in \cC} x_{c}.$$
Since $\sum_{v\in \cV} x_v $ is divisible by four, this implies $\sum_{c\in \cC} x_{c}$ is even.
Therefore in $\cV = \bF_2^{n}$:
\[\sum_{v\in \cV} x_v v = (\sum_{c\in C} x_{c})\binom{p}{2}u = 0.\qedhere\]
    \end{proof}


\section{Essential dimension of \texorpdfstring{$\PGO^{+}_{2n}$}{PGO} at \texorpdfstring{$2$}{2}}

  In this section we prove Parts (5) and (6) of Theorem~\ref{thm.concrete_bounds}.  By Proposition~\ref{prop.sym_rank_lower}, it suffices to prove:
\begin{proposition}\label{prop.rank_Dn}
Let $n =2^r m \geq 4$ be an integer with $r\geq 0$ and $m$ odd. Let $T\subset \PGO^{+}_{2n}$ be a split maximal torus with corresponding Weyl group $W$. There exists a grading $\eps:X(T) \to \bF_2^{r+m-1}$ satisfying \eqref{e.31} and \eqref{e.32} such that
    \begin{equation}\label{e.symrankD_n}
    \SymRank(W(\eps),X(T);2) \geq \begin{cases}
        (r+m-1)2^{r+1} & \text{ if } r\geq 1 \\
        4(m-1) & \text{ if } r=0.
    \end{cases}
\end{equation}
\end{proposition}
\subsection*{Definition of \texorpdfstring{$\eps$}{epsilon}}

We set notation in order to describe $X(T)$ and define $\eps$. Define $\cV=\bF_2^{r}$, $K = \cV \times \{1,\dots,m\}$ and
$$ (\bF_2^{m})_0 = \{ x \in \bF_2^m \mid \sum_{i=1,\dots,m} x_i =0\}.$$
Since $\PGO^+_{2n}$ is the split adjoint group of type $D_n$, we can choose a split maximal torus $T\subset \PGO^+_{2n}$ such that the character lattice $X(T)=Q(D_n)$ is identified with the following sublattice of $\bZ^{K}$ \cite[Section 7.1]{conway2013sphere}:
$$Q(D_n) = \{ x\in \bZ^K \mid  \sum_{k\in K} x_k \text{ is even}\}.$$
Pick a basis $\{b_1,\dots,b_m\}$ for $\bF_2^{m}$ and let $\eps: Q(D_{n}) \to \cV\oplus (\bF_2^{m})_0$ be the grading given by $$\eps(x) = \sum_{(v,i)\in K} x_{v,i}(v + b_i)\in \cV\oplus (\bF_2^{m})_0.$$ 
Note that $\eps$ is surjective because for any $i\neq j$ and $v\in \cV$ we have: $$\eps(e_{0,i}+e_{0,j}) =b_i+b_j, \ \ \eps(e_{0,1}+e_{v,1}) = (v,0).$$

\subsection*{Verification of Condition \eqref{e.31}}

    The roots $\Phi \subset Q(D_{n})$ are given by:
$$\Phi = \{ \pm e_k \pm e_s \mid k,s\in K, \  k\neq s \}. $$
   For any root $\alpha = \pm e_{(v,i)} \pm e_{(v',j)}$ we have:
    $$\eps(\alpha) =  v+v'+ b_i+b_j \neq 0.$$
Therefore $\eps$ satisfies \eqref{e.31}.

\subsection*{Description of the Weyl group}
The Weyl group of $G$ is $\Sym(K) \ltimes (\bF_2^{K})_0$, where $\Sym(K)$ is the group of permutations of $K$ and $(\bF_2^{K})_0$ is the group:
$$(\bF_2^{K})_0=\{ \delta\in \bF_2^{K} \mid \sum_k \delta_k = 0\}.$$
The action of $(\sigma,\delta)\in \Sym(K) \ltimes (\bF_2^{K})_0$ on $Q(D_n)$ is the restriction of the action on $\bZ^K$ defined by:
\begin{equation}\label{e.act_W(eps)_Dn}
(\sigma,\delta)e_k = (-1)^{\delta_k} e_{\sigma(k)}
\end{equation}
for all $k\in K$.
\begin{lemma}\label{lem.WepsDn}
    Identify $\cV$ with the subgroup of $\Sym(K)$ consisting of all translations $\lambda_u: K\to K$ by an element $u\in \cV$:
$$\lambda_u(v,i) = (u+v,i).$$
    We have $W(\eps) = \cV \ltimes (\bF_2^{K})_0$.
\end{lemma}
\begin{proof}
The inclusion $(\bF_2^K)_0 \subset W(\eps)$ is easy to see. The inclusion $\cV\subset W(\eps)$ follows from the computation:
\begin{align*}
    \eps(\lambda_u(x)) &= \sum x_{v,i}(u+v + b_i)  \\
        &= \sum x_{v,i}(v+ b_i) +\sum x_{v,i}u \\
        &= \eps(x) + \sum x_{v,i} u = \eps(v).
\end{align*}
Here in the last equality we used the fact that $\sum x_{v,i}$ is even.
To verify the inclusion $W(\eps)\subset \cV\ltimes (\bF_2^{K})_0$, we assume $w=(\sigma,\delta) \in W(\eps)$  and show $\sigma = \lambda_u$ for some $u\in \cV$. Since $(\bF_2^K)_0\subset W(\eps)$, we can assume $w = \sigma$ without loss of generality. We start by denoting for any $(v,i)\in K$:
$$\sigma(v,i) = (f(v,i),g(v,i))\in K.$$
For any $(v,i),(v',j)\in K$ we have:
\begin{equation}\label{e.D_n1}
    f(v,i) + f(v',j) + b_{g(v,i)} + b_{g(v',j)} = \eps(\sigma(e_{v,i} + e_{v',j})) = \eps(e_{v,i} + e_{v',j}) = v+v' + b_i + b_j.
\end{equation}
Assume $i\neq j$. Then if $g(v,i) \neq i$, \eqref{e.D_n1} implies $$g(v,i) = j.$$ Since $m$ is odd if $i\neq j$ then $m\geq 3$. Therefore repeating the same argument with $1\leq j'\leq m$ different from $i$ and $j$ gives the contradiction $g(v,i) = j'$. Therefore we see that $$g(v,i) = i$$ for all $(v,i)\in K$. Setting $v'= 0, j =1$ in \eqref{e.D_n1}, we get for all $(v,i)\in K$:
$$f(v,i) - v = f(0,1).$$
Therefore $\sigma = \lambda_{f(0,1)}$ is the translation by $f(0,1)$. This finishes the proof.
\end{proof}

\subsection*{Verification of Condition \eqref{e.32}}
   Let $x \in Q(D_n)^{W(\eps)}$ be fixed by $W(\eps)$ and let $k\in K$ be some element. Since $ |K|= n \geq 3$ there exists $\delta\in (\bF_2^K)_0$ such that $\delta_k = 1$. The equation
   $$\delta x = x$$
   implies $x_k = 0$ because $\delta$ flips the sign of the $k$-th coordinate. Since $k\in K$ was arbitrary, we conclude that $x = 0$.
   Therefore $$Q(D_n)^{W(\eps)}=\{0\}.$$ Since $W(\eps)$ is its own Sylow $2$-subgroup, \eqref{e.32} follows.

\subsection*{Proof of the lower bound on \texorpdfstring{$\SymRank(W(\eps),X(T);2)$}{the symmetric rank}}
For the proof we will need the following elementary lemma.
\begin{lemma}\label{lem.elementary_act_on_subsets}
    Let $Y \subset K$ be a non-empty finite subset of $K = \cV\times \{1,\dots,m\}$. For any $u\in \cV$ denote:
    $$u+Y = \{ (u+v,i) \mid (v,i)\in Y\}.$$
    There are at most ${|Y|}$ elements $u\in \cV$ such that:
    \begin{equation}\label{e.calc-1}
        u+Y= Y.
    \end{equation}
\end{lemma}
\begin{proof}
Replace $Y$ with $u+Y$ for some $u\in \cV$ such that $(-u,i)\in Y$ to assume that $(0,i)\in Y$ without loss of generality. Then $u+Y =Y$ implies $(u,i)\in Y$ and so there are at most $|Y|$ elements satisfying \eqref{e.calc-1}. 
\end{proof}

\begin{proof}[Proof of \eqref{e.symrankD_n}]
    Note that $W(\eps)$ is its own Sylow $2$-subgroup. By Lemma~\ref{lem.symrank_fancy}, it suffices to prove that for any $x\in Q(D_n)$ with $\eps(x)\neq 0$, we have $$|W(\eps) x| \geq \begin{cases}
        2^{r+1} & \text{ if }r\geq 1 \\
        4 & \text{ if }r=0.
    \end{cases}$$
     Let $x \in Q(D_{n})$ be such that $\eps(x) \neq 0$. We define $$Y_x = \{ (v,i)\in K \mid x_{a,i} \neq 0 \},$$
     and split into cases.

  \smallskip
  \textbf{Assume $r=0$.} We have $|Y_x|\geq 2$. Otherwise, $x = x_k e_k$ for some $k\in K, x_k\in 2\bZ$ and so $\eps(v) = 0$. Write:
  $$x = x_{k_1}e_{k_1} + x_{k_2}e_{k_2} + \sum_{k\neq k_1,k_2} x_k e_k,$$
  for some $x_{k_1},x_{k_2}\neq 0$.
  Since $|K| = n \geq 3$, for any $(\delta_1,\delta_2)\in \bF_2^2$ there exists $\delta\in (\bF_2^K)_0$ such that $\delta_{k_1}=\delta_{k_2}=1$. Therefore the orbit $W(\eps) x$ contains elements of the form:
  $$ (-1)^{\delta_1} x_{k_1}e_{k_1} + (-1)^{\delta_2} x_{k_2} e_{k_2} + \sum_{k\neq k_1,k_2} x'_k e_k,$$
 for some $x'_k\in \bZ$. We conclude that $|W(\eps)x| \geq  4$.

  \smallskip
  \textbf{Assume $r\geq 1$.} The stabilizer-orbit formula gives:
        $$|W(\eps)x| = |W(\eps)|/|\Stab_{W(\eps)}(x)| = 2^{r+n -1}/|\Stab_{W(\eps)}(x)|.$$
    Therefore our goal is to show:
        \begin{equation}\label{e.calc2}
            |\Stab_{W(\eps)}(x)| \leq 2^{n -2}
        \end{equation}
    If $w=(u,\delta)\in W(\eps)$  fixes $x$, then:
    $$u+ Y_x = Y_x.$$
    Therefore the projection onto the first component gives a short exact sequence:
    $$0\to \Stab_{(\bF_2^K)_0}(x) \to \Stab_{W(\eps)}(x) \to \Stab_{\cV}(Y_x).$$
    Here $\Stab_{\cV}(Y_x)$ is the stabilizer of $Y_x$ with respect to the action of $\cV$ on subsets of $K$ by translation. Thus we get an inequality:
    \begin{equation}\label{e.calc3}
        |\Stab_{W(\eps)}(x) | \leq |\Stab_{(\bF_2^K)_0}(x)| |\Stab_{\cV}(Y_x)|.
    \end{equation}
    By \eqref{e.act_W(eps)_Dn}, the stabilizer $\Stab_{(\bF_2^{K})_0}(x)$ is the following subspace of $(\bF_2^{K})_0$:  
    \begin{equation}\label{e.Dn_def_of_F^m_0}
    (\bF_2^{K\setminus Y_x})_0  := \{ \delta\in (\bF_2^{K})_0  \mid \forall k\in Y_x: \delta_k = 0\}.
    \end{equation}
    If $Y_x = K$ then $(\bF_2^{K\setminus Y_x})_0 = \{0\}$ and \eqref{e.calc3} gives:
    $$|\Stab_{W(\eps)}(x) | \leq  |\Stab_{\cV}(Y_x)| = |\cV| = 2^r.$$
    One can check that $2^r \leq  2^{n-2}$ using $n\geq 3$ (recall that $n = 2^r m)$.
    Therefore \eqref{e.calc2} holds if $Y_x = K$.
    Assume $Y_x \neq K$. We have:
    $$| (\bF_2^{K\setminus Y_x})_0| = |\bF_2^{|K\setminus Y_x| -1}| = 2^{n- |Y_x|-1}.$$
    Since $x\neq 0$, $Y_x$ is non-empty. Applying Lemma~\ref{lem.elementary_act_on_subsets} gives:
    $$|\Stab_{\cV}(Y_x)|\leq |Y_x|.$$
    Plugging this inequality into \eqref{e.calc3} we get:
    $$|\Stab_{W(\eps)}(x) | \leq 2^{n- |Y_x|-1} |Y_x| \leq 2^{n- |Y_x|-1}  2^{|Y_x|-1} = 2^{n -2}.$$
    Here the last inequality follows from the inequality $n \leq 2^{n-1}$ which holds for all natural numbers $n$. This shows \eqref{e.calc2} holds and finishes the proof.
\end{proof}

\begin{remark}
  Assume $r \geq 1$ and let ${v_1,\dots, v_r}$ be  a basis for $\cV$. One can check that the following is a $W(\eps)$-invariant generating set of size $(r+m-1)2^{r+1}$:
    $$\Gamma = \bigsqcup_{\begin{matrix} i = 1,\dots,r \end{matrix}} W(\eps)( e_{0,1} + e_{v_i,1}) \cup \bigsqcup_{\begin{matrix} j = 2,\dots,m \end{matrix}} W(\eps)( e_{0,1} + e_{v_1,j}).$$
    Therefore \eqref{e.symrankD_n} is in fact an equality if $r\geq 1$. Let $x=e_{0,1} +e_{v_i,1}$ or $x= e_{0,1}+e_{v_1,j}$ be one of the elements above. To see that $|W(\eps)x| =2^{r+1}$ one needs to check that $|\Stab_{W(\eps)}(x)|\geq 2^{n-2}$ and therefore \eqref{e.calc2} is an equality. For example, if $x = e_{0,1}+e_{v_i,1}$ then $\Stab_{W(\eps)}(x)$ contains $\langle v_i\rangle\ltimes (\bF_2^{K'})_0$ where $K' = K\setminus (\langle v_i \rangle\times \{1\})$. Therefore $|\Stab_{W(\eps)}(x)|\geq | \langle v_i\rangle\ltimes (\bF_2^{K'})_0| = 2^{ n -2}$. The computation in the case $x = e_{0,1}+e_{v_1,j}$ is similar. 
\end{remark}

\section{Essential dimension of \texorpdfstring{$\HSpin_{16}$ at $2$}{HSpin16}}

  In this section we prove Part (1) of Theorem~\ref{thm.concrete_bounds}. We have already proven $$\ed(E_8;2)= \ed(\HSpin_{16};2)$$ in Proposition~\ref{prop.partofconcreteboundsthm1}(1). It remains to show $\ed(\HSpin_{16};2)\geq 56$. By Proposition~\ref{prop.sym_rank_lower}, it suffices to prove:
\begin{proposition}\label{prop.rank_Hspin}
Let $T\subset \HSpin_{16}$ be a split maximal torus with corresponding Weyl group $W$. There exists a grading $\eps:X(T) \to \bF_2^{4}$ satisfying \eqref{e.31} and \eqref{e.32} such that
    \begin{equation}\label{e.symrankHspin}
    \SymRank(W(\eps),X(T);2) \geq 2^6.
\end{equation}
\end{proposition}

\subsection*{Definition of \texorpdfstring{$\eps$}{epsilon}}

We start by introducing the notation needed to define $\eps$. We set: $$\cV=\bF_2^{3} \text{ and } \nu = \frac{1}{2}\sum_{v\in V} e_v\in \bQ^\cV.$$  There exists a split maximal torus $T\subset \HSpin_{16}$ such that $X(T)$ is the following sublattice of $\bQ^{\cV}$ \cite[Section 3]{macdonald}: 
\begin{align}\label{e.D8_in_E8_lattice}
    X(T)  &= Q(D_8) \cup (\nu +Q(D_8)) \\
        &= \bigg\{ x\in\bQ^{\cV}\  {\Big|}\   \begin{matrix} \sum_{v\in \cV} x_v \text{ is even.}
            \\ \text{For all }v\in \cV, \ x_v\in \bZ  \ \text{or} 
            \\ \text{for all }v\in \cV, \ x_v \in 1/2+\bZ
\end{matrix}\bigg\}     
\end{align}
Here, as in the previous section, $Q(D_8)\subset \bZ^{\cV}$ is the sublattice:
$$Q(D_8) = \{ x\in \bZ^{\cV} \mid \sum_{v\in \cV} x_v \text{ is even}\}.$$
Let $\eps_0 :Q(D_8) \to \cV$ be the grading from the previous section given by:
$$\eps_0 (y)= \sum_{v\in \cV}y_v v.$$
Any vector $x \in X(T)$ can be written uniquely as $x =d\nu + y$, where $d \in \{0,1\}$ and $y\in Q(D_8)$. 
We define $\eps: X(T) \to \cV\oplus \bF_2$  by 
\begin{equation}\label{e.def_eps_E8}
    \eps(x) =\eps(d\nu + y) = (\eps_0(y), d) \in \cV \oplus \bF_2.
\end{equation}
One can check that $\eps$ is a surjective homomorphism. 
\subsection*{Verification of Condition \eqref{e.31}}
Any root $\alpha\in X(T)$ lies in $Q(D_8) \subset X(T)$ because it is a root of $\PGO^{+}_{16}$. We have $\eps_0(\alpha)\neq 0$ because we have shown $\eps_0$ satisfies \eqref{e.31} in the previous section. Therefore:
$$\eps(\alpha) = (\eps_0(\alpha),0)\neq 0.$$

\subsection*{Description of the Weyl group}
As in the previous section, the Weyl group of $\HSpin_{16}$ is the group
$$W = \Sym(\cV)\ltimes (\bF_2^{\cV})_0.$$
The action of $(\sigma,\delta)\in\Sym(\cV)\ltimes (\bF_2^{\cV})_0$ on $X(T)$ is the restriction of the action on $\bQ^\cV$ defined by:
\begin{equation}\label{e.act_W(eps)_HSpin}
(\sigma,\delta)e_v = (-1)^{\delta_v} e_{\sigma(v)}
\end{equation}
for all $v\in \cV$. Since $\eps_0:Q(D_8)\to \cV$ is the grading we considered in the previous section,  Lemma~\ref{lem.WepsDn} implies that the subgroup of $W$ preserving $\eps_0$ is the $2$-group:
$$W(\eps_0) = \cV \ltimes (\bF_2^{\cV})_0.$$
Note that $W(\eps) \subset W(\eps_0)$ by the definition of $\eps$. Since $\eps$ is a homomorphism, an automorphism $w\in W(\eps_0)$ is in $W(\eps)$ if and only if it satisfies:
$$\eps(w\nu)=\eps(\nu)=(0,1).  $$
Let $w = (u,\delta) \in W(\eps_0)$. We compute:
$$\eps(w \nu) = \eps(\nu - \sum_{\delta_v =1} e_v) = (\sum_{\delta_v =1} v, 1).$$
Therefore $w\in W(\eps)$  if and only if $\delta \in (\bF_2^\cV)_1$, where:
$$(\bF_2^\cV)_1 = \{ \delta \in (\bF^\cV_2)_0 \mid \\
    \sum_{\delta_v =1} v= 0
\}.$$
Clearly, $(\bF_2^\cV)_1$ is the kernel of the surjective homomorphism:
$$(\bF_2^{\cV})_0 \to \cV ,\ \ \delta \mapsto \sum_{v\in \cV} \delta_v v.$$
Therefore we have $\dim_{\bF_2} (\bF_2^\cV)_1 = \dim_{\bF_2} (\bF_2^\cV)_0- \dim_{\bF_2} \cV =7-3 =4$ and:
$$|W(\eps)| = |\cV\ltimes  (\bF_2^\cV)_1| = 2^{3 +4}=2^{7}.$$

\subsection*{Verification of Condition \eqref{e.32}}
The element $(1,\dots,1) \in (\bF_2^{\cV})_1\subset W(\eps)$ acts on $X(T)$ by multiplication by $-1$. Therefore $X(T)^{W(\eps)}=\{0\}$ and \eqref{e.32} is satisfied. We note that this implies a fact we used earlier in the proof of Proposition~\ref{prop.partofconcreteboundsthm1}(1).

\begin{lemma}\label{lem.HSpin_admits_aniso}
     There exists an anisotropic $\HSpin_{16}$-torsor over some $2$-closed field $k$ containing $k_0$. 
\end{lemma}
\begin{proof}
    Let $A_\eps \subset \HSpin_{16}$ be the finite diagonalizable $2$-subgroup defined by $\eps$ as in \eqref{e.def_Aeps}.
    By Lemma~\ref{lem.C(Aeps)_admits_aniso}, 
    $C_G(A_\eps)$ admits anisotropic torsors over some $2$-closed field $k_0\subset k$ because $\eps$ satisfies \eqref{e.32}. Therefore the result follows from Corollary~\ref{cor.if_cent_admits_aniso_so_does_G}.
\end{proof}

\subsection*{Proof of the lower bound on \texorpdfstring{$\SymRank(W(\eps),X(T);2)$}{the symmetric rank}}
\begin{proof}[Proof of \eqref{e.symrankHspin}]
     By Lemma~\ref{lem.symrank_fancy}, in order to prove \eqref{e.symrankHspin} it suffices to prove that for any $x\in X(T)$ with $\eps(x) = (v,d)\neq 0$, we have $$|W(\eps) x| \geq 2^4.$$
    We handle the cases $d\neq 0$ and $v\neq 0$ separately.

\smallskip
    \textbf{Assume $d\neq 0$.} Then \eqref{e.def_eps_E8} implies $$x = \nu +y$$ for some $y\in Q(D_8)$. In particular, the coordinates $x_v= \frac{1}{2}+y_v$ are all half-integers and therefore non-zero. We conclude that $(\bF_2^\cV)_1$ acts freely on $x$ by sign changes and we have:
        $$|W(\eps) x | \geq |(\bF_2^\cV)_1 x| = 2^4.$$

\smallskip
     \textbf{Assume $v\neq 0$.} If $d \neq 0$, then $|W(\eps) x|\geq 2^4$ by the previous case. Therefore we can assume $d =0$, which implies $x\in Q(D_8)$ by \eqref{e.def_eps_E8}. We set: $$Y_x = \{ v\in \cV \mid x_v \neq 0 \}.$$
        The stabilizer-orbit formula gives:
        $$|W(\eps)x| = |W(\eps)|/|\Stab_{W(\eps)}(x)| = 2^7/|\Stab_{W(\eps)}(x)|.$$
        Therefore to show $|W(\eps)x|\geq 2^4$ it suffices to prove:
        \begin{equation}\label{e.calc4}
            |\Stab_{W(\eps)}(x)| \leq 2^3.
        \end{equation}
        A similar argument to the proof of \eqref{e.calc3} shows:
        \begin{equation}\label{e.calc5}
             |\Stab_{W(\eps)}(x) | \leq |\Stab_{(\bF_2^\cV)_1}(x)| |\Stab_{\cV}(Y_x)|.
        \end{equation}
        By \eqref{e.act_W(eps)_HSpin}, the stabilizer $\Stab_{(\bF_2^{\cV})_1}(x)$ is the following subspace of $(\bF_2^{\cV})_1$:  
        $$(\bF_2^{\cV\setminus Y_x})_1 = \bigg\{ \delta\in (\bF_2^\cV)_1 \ {\Big|}\  \forall v\in Y_x: \delta_v = 0 \bigg\}.$$
        The set $(\bF_2^{\cV\setminus Y_x})_1$ can also be described as the set of all relations satisfied by an even number of elements of ${\cV\setminus Y_x}$ in $\cV$. This gives the explicit formula:
        \begin{equation}\label{e.dim_even_relations}
            | (\bF_2^{\cV\setminus Y_x})_1 |= \bigg | \bigg\{ \cC \subset \cV\setminus Y_x \ {\Big|}\   \begin{matrix}  |\cC| \text{ is even}\\
            \sum_{c\in \cC} c = 0
        \end{matrix} \bigg\} \bigg |
        \end{equation}
        We divide further into three cases based on $|\cV\setminus Y_x|$. If $|\cV\setminus Y_x|\leq 3$, then $(\bF_2^{\cV\setminus Y_x})_1= \{0 \}$. Indeed, if $(\bF_2^{\cV\setminus Y_x})_1 \neq \{0\}$, then \eqref{e.dim_even_relations} implies that there exists a pair $\{u,v\}\subset \cV\setminus Y_x$ such that
        $$ u + v = 0.$$
        This is impossible because in characteristic two, the above equation implies $u=v$. Plugging $(\bF_2^{\cV\setminus Y_x})_1= \{0 \}$ into \eqref{e.calc5} gives:
        $$|\Stab_{W(\eps)}(x) | \leq  |\Stab_{\cV}(Y_x)| \leq   |\cV| =  2^3.$$
        Therefore \eqref{e.calc4} holds if $|\cV\setminus Y_x|\leq 3$.

        \smallskip
        If $|\cV\setminus Y_x| =4$, then by the previous case, the only non-empty subset of $\cV\setminus Y_x$  that might show up in the set \eqref{e.dim_even_relations} is $\cV\setminus Y_x$ itself. Therefore \eqref{e.dim_even_relations} gives:
        $$ |(\bF_2^{\cV\setminus Y_x})_1| \leq 2.$$
        Together with \eqref{e.calc5} this implies:
        $$|\Stab_{W(\eps)}(x) | \leq |(\bF_2^{\cV\setminus Y_x})_1 | |\Stab_{\cV}(Y_x)| \leq  2|\Stab_{\cV}(Y_x)|.$$ 
        Apply Lemma~\ref{lem.elementary_act_on_subsets} to get:
        $$|\Stab_{W(\eps)}(x) | \leq  2|\Stab_{\cV}(Y_x)| \leq 2|Y_x|.$$
        This implies \eqref{e.calc4} because $|Y_x| = |\cV| - |\cV \setminus Y_x| =4$.

        \smallskip
        If $|\cV\setminus Y_x| \geq 5$, then $\cV\setminus Y_x$ contains an $\bF_2$-basis $v_1,v_2,v_3$ for $\cV$ and two other vectors. Therefore $\cV\setminus Y_x$ contains either $0$ or $v_i + v_j$ for some $i\neq j$. Either way, the homomorphism
        $$f: (\bF_2^{\cV\setminus Y_x})_0 \to \cV, \ \delta \mapsto \sum_{v\in V} \delta_v v$$
        is easily seen to be surjective in this case (here $(\bF_2^{\cV\setminus Y_x})_0$ is defined as in \eqref{e.Dn_def_of_F^m_0}). Since $(\bF_2^{\cV\setminus Y_x})_1= \ker f$, this gives:
        $$\dim (\bF_2^{\cV\setminus Y_x})_1  = \dim (\bF_2^{\cV\setminus Y_x})_0 - \dim \cV = |\cV\setminus Y_x|-4.$$
        In particular, we have
        $$|(\bF_2^{\cV\setminus Y_x})_1 | = 2^{ |\cV\setminus Y_x|-4}.$$
        Plugging this into \eqref{e.calc5} and applying Lemma~\ref{lem.elementary_act_on_subsets} we get:
        $$|\Stab_{W(\eps)}(x) | \leq |(\bF_2^{\cV\setminus Y_x})_1 ||\Stab_{\cV}(Y_x)| \leq 2^{|\cV\setminus Y_x| - 4} |Y_x|.$$
        Note that we can apply Lemma~\ref{lem.elementary_act_on_subsets} because $x\neq 0$ implies $Y_x\neq \emptyset$. Using the inequality $|Y_x| \leq 2^{|Y_x| -1}$ we find:
        $$|\Stab_{W(\eps)}(x) | \leq 2^{|\cV\setminus Y_x| - 4} 2^{|Y_x| -1}= 2^{|\cV| -5} = 2^3.$$
        Therefore \eqref{e.calc4} holds in all cases and \eqref{e.symrankHspin} follows.
\end{proof}

\section{Essential dimension of \texorpdfstring{$E_{6}$ at $3$}{E6}}

 In this section we prove Part (3) of Theorem~\ref{thm.concrete_bounds}.  By Proposition~\ref{prop.sym_rank_lower}, it suffices to prove:
\begin{proposition}\label{prop.rank_E6}
Let $T\subset E_6^{\ad}$ be a split maximal torus with corresponding Weyl group $W$. There exists a grading $\eps:X(T) \to \bF_3^2$ satisfying \eqref{e.31} and \eqref{e.32} such that
    \begin{equation}\label{e.symrankE6}
    \SymRank(W(\eps),X(T);3) \geq 12.
\end{equation}
\end{proposition}
\subsection*{Definition of \texorpdfstring{$\eps$}{epsilon}}
We start by introducing the notation needed in order to define $\eps$.
We identify the character lattice $X(T)$ of a split maximal torus $T\subset E_6^{\ad}$ with the root lattice $Q(E_6)$. We will use the following labeling of the extended Dynkin diagram of $E_6$:

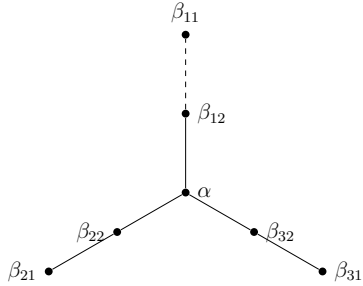
\begin{center}
\scalebox{0.7}{
\begin{tikzpicture}
    \tikzstyle{point}=[circle,fill=black,inner sep=1.5pt]

    \node[point, label=right:$\alpha$] (alpha) at (0,0) {};

    \node[point, label=above:$\beta_{11}$] (beta11) at (0,3) {};
    \node[point, label=right:$\beta_{12}$] (beta12) at (0,1.5) {};
    \draw (alpha) -- (beta12);
    \draw[dashed] (beta12) -- (beta11);

    \node[point, label=left:$\beta_{21}$] (beta21) at (-2.6,-1.5) {};
    \node[point, label=left:$\beta_{22}$] (beta22) at (-1.3,-0.75) {};
    \draw (alpha) -- (beta22) -- (beta21);

    \node[point, label=right:$\beta_{31}$] (beta31) at (2.6,-1.5) {};
    \node[point, label=right:$\beta_{32}$] (beta32) at (1.3,-0.75) {};
    \draw (alpha) -- (beta32) -- (beta31);

\end{tikzpicture}}
\end{center}
Here $-\beta_{11}$ is the highest root given by $$-\beta_{11} = 3\alpha +\beta_{12}+ 2\beta_{21} + \beta_{22}+ 2\beta_{31}+\beta_{32}.$$
Since the simple roots form a basis for $Q(E_6)$, any element in $x\in Q(E_6)$ can be written uniquely as
\begin{equation}\label{e.E6v}
    x = d \alpha + \sum_{i,j} a_{ij}\beta_{ij},
\end{equation}
for some $d\in \{0,1,2\}$ and integers $a_{ij}\in \bZ$. We define a grading $\eps: Q(E_6) \to \bF_3^2$ by:
$$\eps(x) = \eps(d \alpha + \sum_{i,j} a_{ij}\beta_{ij}) = (\sum_{i,j} a_{ij}, d).$$
Note that $\eps$ is surjective. 
\subsection*{Verification of Condition \eqref{e.31}}
Let $\beta\in \Phi$ be a root and express it as above:
$$ \beta =  d \alpha + \sum_{i,j} a_{ij}\beta_{ij},$$
for some $d\in \{0,1,2\}$ and integers $a_{ij}\in \bZ$. If $d \neq 0$, then clearly $$\eps(\beta) = (\sum_{i,j}a_{ij},d) \neq 0.$$
If $d =0$, then $\beta$ lies in the closed subsystem $A_2^{\times 3} \subset E_6$ generated by the $\beta_{ij}$'s; see \cite[Section 4]{chernousov2006another}. Therefore there exists $1\leq i\leq 3$ such that: 
$$\pm \beta \in \{ \beta_{i1}, \beta_{i2}, \beta_{i1}+\beta_{i2}\}_{\tiny i = 1,2,3}.$$
This implies $\eps(\beta) \neq 0$ and so $\eps$ satisfies \eqref{e.31}. 

\subsection*{Description of the Weyl group}
For any root $\beta\in \Phi\subset Q(E_6)$, let $r_\beta\in W$ be the corresponding reflection. Set $\sigma_i = r_{\beta_{i1}}r_{\beta_{i2}}$ for $1\leq i\leq 3$ and define:
$$\sigma := \sigma_1 \sigma_2 \sigma_3.$$
By \cite[Section 4]{chernousov2006another}, there exists an element $\tau\in W$ whose action on the set of roots $\Phi$ is determined by a $2\pi/3$-rotation of the extended Dynkin diagram above. Explicitly, we have $\tau(\alpha) = \alpha$ and for any $1\leq i\leq 3, 1\leq j\leq 2$ we have:
\begin{equation}\label{e.E6_def_tau}
    \tau(\beta_{ij}) = \beta_{i+1,j},
\end{equation}
where we set $\beta_{4,j} := \beta_{1,j}$. 
\begin{lemma}
    The elements $\sigma,\tau$ generate a subgroup $\langle \sigma,\tau\rangle \subset W(\eps)$ isomorphic to $\bZ/3\bZ \times \bZ/3\bZ$. 
\end{lemma}
\begin{proof}
    Since $\sigma = \sigma_1\sigma_2\sigma_3$ is invariant under permuting $1,2,3$, we conclude from \eqref{e.E6_def_tau} that $\tau$ commutes with $\sigma$. Since $\tau$ fixes $\alpha$ and commutes with $\sigma$, $\tau$ preserves $\eps$. To check that $\sigma$ preserves $\eps$, one first computes:
    $$\sigma_i(\alpha) = r_{\beta_{i1}}r_{\beta_{i2}}(\alpha) = r_{\beta_{i1}}(\alpha +\beta_{i2}) =\alpha + \beta_{i1} + \beta_{i2}.$$
    Since $\sigma_i(\beta_{jk}) = \beta_{jk}$ for any $j\neq i$, this gives:
    \begin{equation}\label{e.E6sigmaalpha}
    \sigma(\alpha) = \sigma_1 \sigma_2\sigma_3(\alpha) = \alpha + \sum_{\tiny i,j} \beta_{ij}.
    \end{equation}
    Here we are summing over all $1\leq i\leq 3, 1\leq j \leq 2$. Therefore in $\bF^2_3$:
    \begin{equation}\label{e.E6alpha}
        \eps(\sigma(\alpha))=(6,1) = (0,1) = \eps(\alpha).
    \end{equation} 
    Any vector $x\in Q(E_6)$ is of the form:
    $$ x =  d \alpha + \sum_{i,j} a_{ij}\beta_{ij},$$
    as in \eqref{e.E6v}. By \eqref{e.E6alpha}, to prove $\eps(\sigma(x)) = \eps(x)$ it suffices to check:
    \begin{equation}\label{e.E6subgroupofW}
        \eps(\sigma(\beta_{ij})) = \eps( \beta_{ij})= (1,0).
    \end{equation}
    for all $1\leq i\leq 3$ and $1\leq j\leq 2$. A computation shows that for all $1\leq i\leq 3$: 
    \begin{equation}\label{e.E6sigmabeta}
    \sigma(\beta_{i1}) = \beta_{i2},\  \sigma(\beta_{i2}) = -\beta_{i1}-\beta_{i2}.
    \end{equation}
    Therefore \eqref{e.E6subgroupofW} follows from the fact that in $\bF_3$ we have $-1 -1=1$.
\end{proof}

\begin{remark}
    One can show that $\langle\sigma,\tau \rangle$ is a Sylow $3$-subgroup of $W(\eps)$. We do not include this computation because it is not required for the proof of Theorem~\ref{thm.concrete_bounds}.
\end{remark}

\subsection*{Verification of Condition \eqref{e.32}}
Using \eqref{e.E6sigmabeta} one sees that for any $1\leq i\leq 3, 1\leq j\leq 2$:
    $$\beta_{ij} + \sigma(\beta_{ij}) + \sigma^2(\beta_{ij}) = 0.$$
    Since the $\beta_{ij}$'s generate a subgroup of finite index in $Q(E_6)$, it follows that:
    \begin{equation}\label{e.E6_identity_sigma}
    \id_{Q(E_6)} + \sigma + \sigma^2 = 0.
    \end{equation}
    If $\sigma(x)= x$ for some $x\in Q(E_6)$, then \eqref{e.E6_identity_sigma} implies $x= 0$ because
    $$0 = (\id_{Q(E_6)} + \sigma + \sigma^2)x = 3x.$$
    Therefore $Q(E_6)^{W(\eps)_{3}} = \{0\}$ for any Sylow $3$-subgroup $W(\eps)_{3}\subset W(\eps)$ containing $\sigma$.

\subsection*{Proof of the lower bound on \texorpdfstring{$\SymRank(W(\eps),X(T);3)$}{the symmetric rank}}
We start with two lemmas which help us understand the $\langle \sigma,\tau\rangle$-stabilizers of elements of $Q(E_6)$.

\begin{lemma}\label{lem.E6rank}
    Let $x\in Q(E_6)$ be an element and set $\eps(x) =(u,v) \in \bF_3^2$.
    \begin{enumerate}
        \item If $\sigma(x) = x$ then $x =0$.
        \item If $\sigma(x)\in \{ x, \tau(x),\tau^2(x)\}$ then $u =0$.
        \item If $\tau(x) = x$ then $u =0$.
    \end{enumerate}
\end{lemma}
\begin{proof}
    
    We proved Part (1) during the verification of Condition \eqref{e.32} above. To prove Part (2), we can assume $\sigma(x)= \tau^k(x)$ for $k=1$ or $k=2$ by Part (1). Apply \eqref{e.E6_identity_sigma} to get:
    \begin{equation}\label{e.E6ranklemma}
        0 = (\id_{Q(E_6)} + \sigma + \sigma^2)x = x + \tau^k(x) + \tau^{2k}(x)=  x+ \tau(x)+ \tau^2(x).
    \end{equation}
    Here in the last equality we used that $\tau^3 =1$. Express $x$ as a sum
    $$ x = d \alpha + \sum_{i,j} a_{ij}\beta_{ij}.$$
    for some $d\in \{0,1,2\}$ and integers $a_{ij}\in \bZ$. Comparing coefficients in \eqref{e.E6ranklemma} we see that for all $1 \leq j\leq 2$:
    $$a_{1j} + a_{2j} + a_{3j} = 0.$$
    Therefore:
    $$(u,v) = \eps(x) = (\sum_{j}a_{1j} + a_{2j} + a_{3j}, d) = (0,d).$$
    To prove Part (3), we assume:
    $$\tau(x) = x.$$
    Comparing coefficients of both sides we get for all $1\leq j\leq 2$:
    $$a_{1j} = a_{2j} = a_{3j}.$$
    This gives:
    \[(u,v) = \eps(x) = (\sum_{j}a_{1j} + a_{2j} + a_{3j}, d) = (\sum_j 3a_{1j},d)= (0,d).\qedhere\]
\end{proof}

\begin{corollary}\label{cor.E6stabxnotzero}
     Let $x\in Q(E_6)$ be an element and denote $\eps(x) = (u,v)$ for some $u,v \in \bF_3$. We note:
 \begin{enumerate}
     \item  If $u\neq 0$, then $|\langle \sigma, \tau\rangle x|  = 9$.
    \item If $v\neq 0$, then $ |\langle \sigma, \tau\rangle x| \geq 3$.
 \end{enumerate} 
\end{corollary}
\begin{proof}
If $u\neq 0$, then $\tau(x) \neq x $ and $\sigma(x) \not \in  \{x,\tau(x),\tau^2(x)\}$ by Lemma~\ref{lem.E6rank}. Therefore $|\langle \sigma,\tau \rangle x|\geq 4$. Since the size of the orbit $|\langle \sigma,\tau \rangle x|$ divides $|\langle \sigma,\tau \rangle | = 9$, we conclude that 
    $$|\langle \sigma, \tau\rangle x|  = 9.$$
If $v\neq 0$, then $\sigma(x) \neq x$ by Lemma~\ref{lem.E6rank}(1). Therefore 
\[|\langle \sigma, \tau\rangle x| \geq |\langle \sigma \rangle x| = 3.\qedhere\]
\end{proof}

Combining Corollary~\ref{cor.E6stabxnotzero} and Lemma~\ref{lem.symrank_fancy} gives
$$\SymRank(\langle \sigma,\tau \rangle, X(T);3) \geq 9+3= 12.$$
This proves \eqref{e.symrankE6} because by \eqref{e.symrank_monotone} we have
    $$\SymRank(W(\eps),X(T);3) \geq \SymRank(\langle \sigma,\tau \rangle, X(T);3).$$

\bibliographystyle{plain}
\bibliography{references2}

\end{document}